%% file: main.tex
\documentclass[preprint,12pt]{elsarticle}
\include{mymacros}
\usepackage{pdflscape}

\usepackage{amssymb}
\usepackage{amsmath}

\begin{document}

\begin{frontmatter}



\title{Quaternion-Valued Wavelets on the Plane: A Construction via the Douglas--Rachford Approach} 


\author[label1]{Neil D. Dizon}
\author[label2]{Jeffrey A. Hogan}

\affiliation[label1]{organization={School of Mathematics and Statistics, University of New South Wales},
            city={Sydney},
            postcode={2052}, 
            state={NSW},
            country={Australia}}
            
\affiliation[label2]{organization={School of Computing and Information Sciences, University of Newcastle},
    city={Callaghan},
    postcode={2308}, 
    state={NSW},
    country={Australia}}                        

 
\begin{abstract}
This paper presents a reformulation of the construction of nonseparable multiresolution quaternion-valued wavelets on the plane as a feasibility problem. The constraint sets in the feasibility problem are derived from the standard conditions of smoothness, compact support, and orthonormality. To solve the resulting feasibility problems, we employ a product space formulation of the Douglas–Rachford algorithm. This approach yields novel examples of nonseparable, multiresolution, compactly supported, smooth, and orthonormal quaternion-valued wavelets on the plane. Additionally, by introducing a symmetry-promoting constraint, we construct symmetric quaternion-valued scaling functions on the plane.
\end{abstract}



\begin{keyword}
Wavelet \sep Quaternion \sep Douglas--Rachford \sep Multiresolution \sep Projection algorithm \sep Feasibility problem

\vspace{1em}


\MSC[2020] 42C40 \sep 42B99 \sep 65T60 \sep 47N10 \sep 65K10 

\end{keyword}

\end{frontmatter}

    

\section{Introduction}\label{sec:mainintro}

\emph{Wavelets} have become an important mathematical tool with far-reaching applications ranging from signal processing through to the numerical solution of differential equations. Although the notion of wavelets appeared as a culmination of ideas in pure mathematics, engineering and physics, the suitability of wavelets in dealing with non-stationary signals served as a driving stimulus behind early wavelet research.  Key contributions to wavelet theory include Daubechies' construction of compactly supported, smooth and orthonormal wavelets on the line with multiresolution structure \cite{daubechies1,daubechiesbook,mallat, meyerchapter} which has had a profound impact on the practical use of wavelets in science and engineering.

Recently, an optimization approach for the construction of multidimensional wavelets has been developed \cite{FHTconference,FHTpaper,franklinthesis}. More precisely, wavelet construction was formulated as a \emph{feasibility problem} of finding a point on the intersection of constraint sets arising from the conditions of multiresolution analysis (MRA) and other design criteria. Feasibility problems are customarily solved using \emph{projection algorithms} including the \emph{method of alternating projections} (MAP) \cite{neumann}  and the \emph{Douglas--Rachford} (DR) scheme \cite{drachford}, among others. This approach has led to the development of new examples of nonseparable, smooth, compactly supported, orthonormal, complex-valued  (also real-valued) wavelets on the plane \cite{FHTpaper,franklinthesis}. When applied in one dimension, these methods can reproduce the Daubechies' wavelets,  but can also generate other examples. Additionally, this approach makes provision for expanding the design criteria to include other wavelet properties like \emph{cardinality} and \emph{symmetry} that are essential to certain image processing applications \cite{DHLi_centering,DDHTart,dizonthesis}.

In modern applications, higher-dimensional signals like color images are often vector-valued, requiring more advanced  transform methods. More precisely, a typical color image is viewed as a three-channel signal on the plane, where these channels are customarily the red, green and blue (RGB) components of the pixels. Alternative color image models include the luminance-chrominance (YUV) and
the cyan-magenta-yellow-key (CMYK) that use three and four channels, respectively. Most contemporary methods for handling color images analyze each channel independently, ignoring inter-channel correlations. However, it
is more advantageous to encode the pixel components using higher-dimensional algebras, e.g., the algebra of quaternions, in order to possibly capture inter-channel correlations in color images \cite{peicolour, sangwinecolour}.  Consequently, once a vector-valued signal is embedded into a higher-dimensional algebra, more sophisticated wavelet transforms become imperative.

While a plethora of research in the literature addresses quaternionic extensions of wavelet transforms, a rather limited number of scholarly works are devoted to the development of \textit{true} quaternionic wavelets.  In a comprehensive survey concerning the development of quaternionic wavelet transform (QWT) \cite{fletchersangwine}, it was noted that most of these extensions were only derived from real or complex wavelet filters, and are only separable wavelet transforms in disguise. Notable attempts to develop true QWTs include the matrix-valued transform approach in \cite{ginzbergwalden}, and the quaternionic wavelet theory presented in \cite{hoganmorris1,morristhesis}. More recently, quaternion-valued scaling filters on the line were introduced in \cite{fletcher} by expanding the framework established in \cite{ginzbergwalden}. In \cite{hoganmorris1,morristhesis}, fundamental analogues of classical wavelet theory were derived, including the basic results needed to construct orthonormal, compactly supported and smooth quaternion-valued wavelets on the plane based on quaternionic MRA. However, no such wavelets were explicitly constructed.

In this paper, we build on the quaternionic wavelet theory developed in \cite{hoganmorris1,morristhesis}. Drawing inspiration from the complex-valued higher-dimensional wavelet construction in \cite{FHTpaper}, we construct quaternion-valued wavelets via a feasibility approach. Our key contributions are as follow: 
\begin{enumerate}
\setlength{\itemsep}{0pt}
\setlength{\parskip}{0pt}
    \item[\textit{(i)}] We reformulate quaternionic wavelet construction as a feasibility problem. This involves expressing the quaternionic MRA conditions and other design criteria (including smoothness, compact support, orthonormality and symmetry) in terms of an appropriate quaternionic wavelet matrix.  The feasibility problem is derived from a suitable discretization of these conditions, structured into constraint sets that must be satisfied simultaneously. The discretized setup enables the application of numerical methods to solve the feasibility problem.
    \item[\textit{(ii)}]  We compute the necessary projection operators onto the constraint sets, and employ the Douglas--Rachford algorithm to solve the resulting feasibility problem. The solutions of the wavelet feasibility problem readily provides quaternionic scaling and wavelet filters. The scaling function and wavelets are plotted using a quaternionic cascade algorithm.
    \item[\textit{(iii)}] We provide explicit examples of nonseparable, multiresolution, compactly supported, smooth and orthonormal quaternion-valued wavelets on the plane obtained from the solutions of the quaternionic wavelet feasibility problems. To the best of our knowledge, these are the first examples of wavelets that meet all these specific criteria.
\end{enumerate}

The rest of the paper is organized as follows. In Section~\ref{sec:prelim} we fix notation and introduce relevant concepts for quaternion algebra, quaternionic Fourier transforms, and optimization theory. Section~\ref{sec:Qwaveletdesignconditions} contains a brief discussion of quaternionic MRA, scaling functions and wavelets. Here we revisit fundamental quaternionic analogues of classical wavelet theory. We also add a point symmetry condition into the design criteria for the scaling function. Section~\ref{sec:Qwaveletconstruction} details the key step of reformulating the wavelet design criteria in terms of a suitable quaternionic wavelet matrix, followed by discretization through uniform sampling. In Section~\ref{sec:Qwaveletfeasibility}, we list the constraint sets that comprise the quaternionic wavelet feasibility problem, deriving the respective projectors onto these constraint sets. Finally, in Section~\ref{sec:Qnumericalsolutions} we solve specific cases of the quaternionic wavelet feasibility problem using the Douglas--Rachford algorithm leading to explicit examples of nonseparable, multiresolution, compactly supported, smooth, and orthonormal quaternion-valued wavelets on the plane.

\section{Preliminaries}
\label{sec:prelim}

In this section, we fix notation and provide key concepts in quaternion algebra, quaternionic Fourier transforms, and optimization theory.

\subsection{Basic notation}\label{sec:basicnotation}

The set of complex numbers is denoted by $\mathbb{C}$. If $z\in \mathbb{C}$, then $\Re(z)$ is its \emph{real} part and $\Im(z)$ is its \emph{imaginary} part. We view $\mathbb{C}^d$ as a real inner product space with $\langle \cdot, \cdot\rangle_{\mathbb{R}} : \mathbb{C}^d \times \mathbb{C}^d \to \mathbb{R}$ defined by $
    \langle x,y\rangle_{\mathbb{R}} = \sum_{k=1}^d \big(\langle \Re(x_k),\Re(y_k) \rangle + \langle \Im(x_k),\Im(y_k) \rangle \big).
$

The collection of $a \times b$ matrices with complex entries is denoted by $\mathbb{C}^{a \times b}$, and the collection of $c \times c$ unitary matrices is $\mathcal{U}(c)$. Moreover, $I_d \in \mathbb{C}^{d\times d}$ is the identity matrix. The matrix $\mbox{diag}(v_1,v_2,\ldots,v_d)$ is $d\times d$ with diagonal entries $v_1,v_2,\ldots,v_d \in \mathbb{C}$. We denote by $\mathbb{C}^{a \times b} \otimes \mathbb{C}^{c\times d}$ the set of all block diagonal matrices in $\mathbb{C}^{(a+c)\times (b+d)}$  of the form
\[
\begin{bmatrix}
    M_1 & 0_{a\times d}\\
    0_{c\times b} & M_2
\end{bmatrix}\quad
\text{where}\quad
M_1 \in \mathbb{C}^{a\times b},\
M_2 \in \mathbb{C}^{c\times d}.
\]
If $A \in \mathbb{C}^{a \times b}$,  then $A_{ij}$ (or $A[i,j]$ interchangeably) denotes the $(i,j)$-entry of $A$. For matrices $A,B \in \mathbb{C}^{d \times d}$, $A^{\ast}$ is the Hermitian conjugate of $A$, and the {Frobenius inner product} of $A$ and $B$ is defined by $\langle A,B \rangle_F =\Tr(A^{\ast}B) = \sum_{i,k=1}^{d}\overline{A_{ik}}B_{ik}$. The Frobenius matrix norm is given by $\|A\|_F^2 = \sum_{i,k=1}^{d}|A_{ik}|^2 = \sum_{i,k=1}^{d}((\Re(A_{ik}))^2 + (\Im(A_{ik}))^2).$

In a Hilbert space ${\mathcal H}$, the closed ball $B(z,\delta )$ of radius $\delta >0$ centred at $z\in{\mathcal H}$ is
$$B(z,\delta )=\{x\in{\mathcal H}:\, \|x-z\|\leq \delta\}.$$

\subsection{Quaternion algebra}
Let $\mathbb{R}_2 := \{x = x_0+x_1e_{1} + x_2e_{2}+x_{12}e_{12} \,: \, x_0,x_1,x_2,x_{12} \in \mathbb{R}, \, e_1^2 =e_2^2 = e_{12}^2=e_1e_2e_{12}=-1\} $ be the set of \emph{quaternions}. Given $x = x_0 + x_1e_1 + x_2e_2 + x_{12}e_{12} \in \mathbb{R}_2$, we call $x_0$ the \emph{real} part of $x$, and $x_1,x_2,x_{12}$ the \emph{imaginary parts} of $x$. If $x_0=0$ then $x$ is a \emph{pure quaternion}. If $x_0=x_{12}=0$ then $x$ is a \emph{vector}. If $x_1=x_2=x_{12}=0$, then $x$ is a \emph{scalar}. We may also decompose $\mathbb{R}_2$ as $\mathbb{R}_2 =\Lambda_0 \oplus \Lambda_1 \oplus  \Lambda_2,$ where $\Lambda_0 :=\mathbb{R}$, $\Lambda_1 := \{x_1 e_1 + x_2e_2 : x_1,x_2 \in\mathbb{R}\}$, and $\Lambda_{2} := \{x_{12} e_{12} : x_{12}\in\mathbb{R}\}$. We write $[x]_k$ to refer to the $\Lambda_k$-part of $x = x_0 + x_1e_1 + x_2e_2 + x_{12}e_{12} \in \mathbb{R}_2$. Note that $\Lambda_0$ is the set of {scalars}, and $\Lambda_1$ is the collection of {vectors}.  In view of $\Lambda_1$, we consider $\mathbb{R}^2$ to be embedded in $\mathbb{R}_2$ as vectors.  If $ x,y \in \mathbb{R}_2$ are vectors, then $x^2 = -|x|^2  \text{~~and~~} xy = -\langle x,y \rangle + x \wedge y \in \Lambda_0 \oplus \Lambda_1,$ where  $\langle \cdot, \cdot \rangle$ is the usual inner product in $\mathbb{R}^2$ and $x \wedge y :=  e_{12}(x_1y_2 - x_2y_1)$ is the \emph{wedge product} of $x$ and $y$. We remark that a quaternion $x = x_0 + x_1e_1 + x_2e_2 + x_{12}e_{12}$ has a polar form representation given by 
\begin{equation}\label{eqn:polarquaternion}
    x = |x|e^{\mu_x \phi_x} = (\cos(\phi_x) + |x|\mu_x \sin(\phi_x))
\end{equation}
where 
\begin{equation*}
        |x| = \sqrt{x_0^2 + x_1^2 + x_2^2 +x_{12}^2}, \
        \mu_x  = \dfrac{x_1e_1 + x_2e_2 + x_{12}e_{12}}{\sqrt{x_1^2 + x_2^2 +x_{12}^2}},
\end{equation*}	
and the angle $\phi_x \in [0,\pi]$ is defined by
\[
\phi_x :=
\begin{cases}
    \tan^{-1}\!\left(\dfrac{\sqrt{x_1^2 + x_2^2 + x_{12}^2}}{x_0}\right),
    & x_0  \neq 0, \\
     \frac{\pi}{2}, & x_0 = 0.
\end{cases}
\]
When $x_1=x_2=x_3=0$, the direction $\mu_x$ is undefined; and in this case, we simply set $\mu_x=0$,  and the polar representation reduces to $x=|x|e^{0}$ (since $\phi_x=0$).

We further define $x_s:=x_0 + x_{12}e_{12}$ and $x_v = x_1e_1 + x_2e_2$ as the \emph{spinor} and \emph{vector parts} of $x$, respectively. The spinor and vector parts of any $x,y \in \mathbb{R}_2$ satisfy the following commutation relations \cite{hoganmorris1,morristhesis}: (a) $x_sy_s = y_sx_s$, (b) $x_sy_v = y_v\overline{x_s}$, and (c) $x_vy_v = -e_1y_vx_ve_1 = -e_2y_vx_ve_2$.
Moreover, for a function $f: X \rightarrow \mathbb{R}_2$ where $X \subseteq \mathbb{R}^2$, we write $f(x) = f_0(x) + f_1(x)e_1 + f_2(x) + f_{12}(x)$ where $f_0,f_1,f_2,f_{12}: X \to \mathbb{R}$. The spinor part $f_s$ of $f$ is defined as $f_s(x):= (f(x))_s = f_0(x) +f_{12}(x)e_{12}$ while its vector part $f_v$ is given by $f_v(x):=(f(x))_v = f_1(x)e_1 + f_2(x)e_2$. We also define the \emph{spinor-vector matrix of a function} $f:\mathbb{R}^2 \to \mathbb{R}_2$ by
\begin{equation*}
    [ f(x) ] := \begin{bmatrix}
        f_s(x) & f_v(x)\\
        f_v(-x) & f_s(-x)
    \end{bmatrix}.
\end{equation*}
In general, a matrix $A \in \mathbb{R}_2^{2\times 2}$ (not necessarily attached to a quaternion-valued function) is a \emph{spinor-vector matrix} if it is of the form $$A=\begin{bmatrix}s_1&v_1\\v_2&s_2\end{bmatrix} \mbox{ where } s_1,s_2 \in \Lambda_0 \oplus \Lambda_{2} \mbox{ and } v_1,v_2 \in \Lambda_1.$$
Henceforth, we let $\mathsf{S}$ to be the set of all spinor-vector matrices.  Furthermore, if $Z:=\{x=[x_s \quad  x_v ]^{\top} : x_s \in \Lambda_0 \oplus \Lambda_{2}, \, x_v \in \Lambda_1\}$, then $A \in \mathsf{S}$ determines a linear transformation on $Z$. In fact, $Z$ is a right inner product module with inner product $\langle x,y \rangle = \overline{x_s}y_s + \overline{x_v}y_v$ satisfying $\overline{\langle x,y \rangle} = \langle y,x \rangle$, $\langle x,ya \rangle = \langle x,y \rangle a$ and $\langle xa,y \rangle =\overline{a}\langle x,y \rangle$ for all $x,y\in Z$, $a\in \Lambda_0 \oplus \Lambda_{2}$. Moreover, if $A\in\mathsf{S}$ and $A^{\ast} = \overline{A}^{\top}$ is the \emph{Hermitian adjoint} of $A$, then $\langle Ax,y \rangle = \langle x,A^{\ast} y \rangle$. We say that $A$ is \emph{self-adjoint} whenever $A = A^{\ast}$, i.e., $s_1,s_2\in \mathbb{R}$ and $v_2=-v_1$. We also define an \textit{eigenvector} $x$ of $A$ to be a nonzero vector $x=\begin{bmatrix}x_s &  x_v \end{bmatrix}^{\top} \in Z$ such that $Ax=x\lambda$, where $\lambda \in \Lambda_0 \oplus \Lambda_{2}$ is the \textit{eigenvalue}. The two eigenvalues of $A$ are computed using the formula (see, e.g., \cite[Section~1.5]{morristhesis}) given by
\begin{equation}\label{eqn:SVeigenvalue}
    \lambda = \frac{s_1 + \overline{s_2} \pm \sqrt{(s_1-\overline{s_2})^2 + 4v_1v_2}}{2}.
\end{equation}
In addition, if $A$ is self-adjoint then its eigenvalues are real. We also say that $A$ is \emph{positive definite} if $\langle Ax,x\rangle >0$ for all nonzero $x\in Z$. Furthermore, $A$ is positive definite if and only if all of its eigenvalues are positive. Given two self-adjoint $A,B \in \mathsf{S}$, we say $A<B$ whenever $\langle Ax,x \rangle < \langle Bx,x \rangle$ for every $x\in Z$. For a comprehensive discussion on the linear algebra of spinor-vector matrices, refer to \cite[Section~1.2]{hoganmorris1} or \cite[Section~1.5]{morristhesis}.

We remark that $\Lambda_0 \oplus \Lambda_2$ is a subalgebra of $\mathbb{R}_2$. It is straightforward to check that $\Lambda_0 \oplus \Lambda_2 \cong \mathbb{C}$  by defining the mapping $\vartheta: \Lambda_0 \oplus \Lambda_2 \to \mathbb{C}$ where $\vartheta(x_0 + x_{12}e_{12}) = x_0 + x_{12}i$. Henceforth, we use $\mathbb{C}$ and $\Lambda_0 \oplus \Lambda_2$ interchangeably and apply theorems valid for $\mathbb{C}^d$ to $(\Lambda_0 \oplus \Lambda_2)^d$.

Finally, we denote the collection of $m\times m$ spinor-vector block matrices by 
\begin{equation*}\mathsf{S}^{m\times m}:=\left\{\begin{bmatrix}A_{11}&\cdots&A_{1m}\\\vdots&\ddots&\vdots\\A_{m1}&\cdots&A_{mm} \end{bmatrix} \ : \ A_{ij} \in \mathsf{S} \mbox{ for } i,j\in\{1,2,\ldots,m\}\right\},
\end{equation*}
and by $\mathcal{U}^{m\times m}$ the set of unitary $m\times m$ spinor-vector block matrices.

\subsection{Quaternionic Fourier transform}
Throughout the paper, we employ the quaternionic Fourier transform introduced by Brackx, De Schepper and Sommen \cite{brackx1,brackx2}, which is defined as follows.
\begin{definition} 
    Let $f \in L^1(\mathbb{R}^2, \mathbb{R}_2)$. The quaternionic Fourier transform $\hat f=\Fq f$ of $f$ is given by
    \begin{equation*}
        \hat f(\xi )=\Fq f(\xi) :=\hat{f}(\xi):=\int_{\mathbb{R}^2} e^{2\pi \xi \wedge x} f(x) \, dx.
    \end{equation*}
\end{definition}
In this paper, we use $\Fq f$ or $\hat{f}$ interchangeably to denote the Fourier transform of $f$. 

\begin{proposition}
    Let $f,g \in L^2(\mathbb{R}^2, \mathbb{R}_2)$. The quaternionic Fourier transform satisfies the following properties:
    \begin{enumerate}[label=(\alph*)]
        \itemsep0em
        \item $\langle f,g \rangle = \langle \Fq{f}, \Fq g\rangle$, 
        \item $\|f\|_2 = \|\Fq f\|_2$,
        \item $\Fq (f(\cdot-a))(\xi) = e^{2\pi \xi \wedge a}\Fq f(\xi)$ ,
        \item $\Fq (e^{2\pi \cdot \wedge a}f)(\xi) = (\Fq f)(\xi - a)$, 
        \item $\displaystyle \Fq(f(b\cdot))(\xi) = \frac{1}{b^2}\Fq f\left(\frac{\xi}{b} \right)$,
        \item $(\Fq)^2 = \Id $, 
    \end{enumerate} where $a\in\mathbb{R}^2$ and $b>0$. Furthermore, if $f \ast g$ is the convolution of $f$ and $g$, then  
    \begin{equation}\label{eq:quat_convolution} 
        [ \Fq(f \ast g)(\xi) ] = [ \Fq f(\xi) ] [ \Fq g(\xi) ]
    \end{equation}
    which is also referred to as the \emph{convolution theorem} for the quaternionic Fourier transform. 
\end{proposition}
The convolution theorem in \eqref{eq:quat_convolution} is proved in \cite[Theorem~2.16]{hoganmorris2} and is understood in terms of spinor-vector matrices. It is also worth noting that the convolution of two functions is defined similarly as in the classical setting with a caveat that, in general, $f \ast g \neq g \ast f$. For a more elaborate discussion on the quaternionic Fourier transform, refer to \cite{hoganmorris1,morristhesis}.

We note in the next proposition that the spinor-vector matrix of the integral kernel of the quaternionic Fourier transform commutes with the spinor-vector matrix of any quaternion-valued function on the plane.

\begin{proposition}\label{prop:kernelcommutes}
    If $f: \mathbb{R}^2 \to \mathbb{R}_2$, then for any $x,\xi \in \mathbb{R}^2$ we have
    $$[e^{2\pi \xi \wedge x}][f(x)] = [f(x)][e^{2\pi \xi \wedge x}].$$
\end{proposition}
\begin{proof}
    See, e.g., \cite[Lemma~1.19]{morristhesis}.
\end{proof}

\subsection{Feasibility problems and the Douglas--Rachford algorithm}\label{sec:optimization}
Since our construction approach relies heavily on optimization, we also revisit relevant concepts in optimization concerning feasibility problems and how to solve them.

A \emph{feasibility problem} is the task of finding a point in the intersection of a finite family of sets. Formally, given sets $K_1, K_2, \dots, K_r$ contained in a Hilbert space $\mathcal{H}$, the corresponding feasibility problem is 
\begin{equation*}
    \text{find~} x^\ast \in K :=\bigcap_{j=1}^r K_j. 
\end{equation*}
In the literature, \emph{projection algorithms} are often used to solve feasibility problems. The {Douglas--Rachford} (DR) algorithm \cite{drachford} is a well-known example of a projection algorithm that solves a two-set feasibility problem. It has been observed to perform remarkably well in certain classes of nonconvex problems \cite{aacampoy,bsims,dtam}. 

 If $C$ is a nonempty subset of $\mathcal{H}$, the \emph{projector} onto $C$ is the set-valued operator $P_{C}\colon \mathcal{H} \rightrightarrows C$ defined by
\begin{equation*}
    P_{C}(x) = \{ c \in C: \|x-c\| = \inf_{z\in C}\|x-z\| \};
\end{equation*}
and the \emph{reflector}  with respect to $C$ is the set-valued operator $R_{C}\colon \mathcal{H} \rightrightarrows \mathcal{H}$ defined by
\begin{equation*}
    R_{C} := 2P_{C} - \Id.
\end{equation*}
An element of $P_{C}(x)$ is called a \emph{projection} of $x$ onto $C$. Similarly, an element of $R_C(x)$ is called a \emph{reflection} of $x$ with respect to $C$. Note that the use of ``$\rightrightarrows$'' is to emphasize that an operator is set-valued. For examples of projectors that form part of more complicated projectors used in this paper, refer to \ref{app1}.

We now introduce the DR algorithm.

\begin{definition}
    Given two nonempty subsets $A$ and $B$ of $\mathcal{H}$, the \emph{DR operator} $T_{A,B}$ is defined as
    \begin{equation*}
        T_{A,B} := \frac{\Id+R_BR_A}{2}.
    \end{equation*}
\end{definition}
It is worth noting (see \cite[Equations~(20)--(23)]{bdao}) that if $P_A$ is single-valued, then 
\begin{equation*}
    T_{A,B} = \Id - P_A + P_BR_A \text{~~and~~} P_A(\Fix T_{A,B}) = A\cap B,
\end{equation*}
where $\Fix T_{A,B}:=\{x\in\mathcal{H} : x \in T_{A,B}(x)\}$ is the set of of fixed points of the operator $T_{A,B}$. If $A$ and $B$ are closed convex subsets of $\mathcal{H}$ with $A\cap B \neq \varnothing$, then, for any $x_0 \in \mathcal{H}$, the sequence $(x_n)_{n\in \mathbb{N}}$ generated by $x_{n+1} = T_{A,B}(x_n)$ converges weakly to a point $x^\ast \in \Fix T_{A,B}$, and the \emph{shadow sequence} $(P_A(x_n))_{n\in \mathbb{N}}$ converges weakly to $P_A(x^\ast) \in A\cap B$ \cite{lions,svaiter}.

Although originally formulated for two-set feasibility problems, the DR algorithm has been adapted to solve many-set feasibility problems by employing Pierra's {product space reformulation} \cite{pierra} (among other extensions). To be more precise, given $K_1, K_2, \dots, K_r \subseteq \mathcal{H}$, with corresponding projectors $P_{K_1}, P_{K_2}, \dots, P_{K_r}$, the sets $C$ and $D$ in the product Hilbert space $\mathcal{H}^r$ are defined by
\begin{subequations}
    \begin{align} 
        C &:= K_1 \times K_2 \times \dots \times K_r \text{ and} \label{productC}\\
        D &:= \{(x_1, x_2, \dots, x_r) \in \mathcal{H}^r: x_1 = x_2 = \cdots = x_r\}. \label{productD}
    \end{align}
\end{subequations}
The $r$-set feasibility problem is equivalent to the two-set feasibility problem on $C$ and $D$ in the sense that 
\begin{equation}\label{productspaceequivalence}
    x^{\ast} \in \bigcap_{j=1}^r K_j  \ \iff \vec{x}^{\ast}:=(x^{\ast},x^{\ast},\dots,x^{\ast}) \in C \cap D.
\end{equation}
Furthermore, the projectors onto $C$ and $D$ are given by
\begin{align*}
    P_{C}(\vec{x}) =  P_{K_1}(x_1) \times P_{K_2}(x_2) \times \dots \times P_{K_r}(x_r) \text{ and }\, 
    P_{D}(\vec{x}) = \left(\frac{1}{r}\sum_{j=1}^r x_j, \frac{1}{r}\sum_{j=1}^r x_j, \dots, \frac{1}{r}\sum_{j=1}^r x_j\right)
\end{align*}
for any $\vec{x}=(x_1,x_2,\dots, x_r) \in \mathcal{H}^r$; see, e.g., \cite[Proposition~29.3 and Proposition~26.4(iii)]{bcombettes}.  The pseudocode for implementing the DR is outlined in \ref{app2}.

\section{Quaternionic wavelets on the plane}\label{sec:Qwaveletdesignconditions}

In this section, we recall important results in \cite{hoganmorris1,morristhesis} that are required to construct orthonormal quaternion-valued wavelets on the plane with compact support and prescribed regularity based on quaternionic multiresolution analysis.

\subsection{MRA, scaling function and wavelets}

\begin{definition}\label{def:QMRA}
    A multiresolution analysis for $L^2(\mathbb{R}^2,\mathbb{R}_2)$ is a collection of closed right linear $\mathbb{R}_2$-submodules $(V_j)_{j \in \mathbb{Z}}$ of $L^2(\mathbb{R}^2,\mathbb{R}_2)$ and a function $\varphi \in V_0$ such that the following conditions hold:
    \begin{enumerate}[label=(\alph*)]
        \itemsep0em
        \item $V_{j}\subset V_{j+1}$  for all $j \in \mathbb{Z}$,\label{def:QMRAcond1}
        \item $ \overline{\bigcup_{j \in \mathbb{Z}} V_j} = L^2(\mathbb{R}^2,\mathbb{R}_2)$ and $\, \bigcap_{j \in \mathbb{Z}}V_j = \{0\}$,\label{def:QMRAcond2}
        \item $f(\cdot) \in V_0$ if and only if $f(\cdot - k) \in V_0$ for all $k \in \mathbb{Z}^2$,\label{def:QMRAcond3}
        \item $f(\cdot) \in V_0$ if and only if $f(2^j\cdot) \in V_{j}$ for all $j \in \mathbb{Z}$, and\label{def:QMRAcond4}
        \item $\{\varphi(\cdot-k)\}_{k \in \mathbb{Z}^2}$ forms an orthonormal basis for $V_0$.\label{def:QMRAcond5}
    \end{enumerate}
\end{definition}

The function $\varphi \in V_0$ in Definition~\ref{def:QMRA} is called the \emph{scaling function}. Combining conditions \ref{def:QMRAcond4}--\ref{def:QMRAcond5} of Definition~\ref{def:QMRA}, we see that $\varphi(\frac{x}{2}) \in V_{-1} \subset V_0$ and there exists $(a_k^0)_{k\in\mathbb{Z}^2} \in \ell^2(\mathbb{Z}^2,\mathbb{R}_2)$ such that
\begin{equation}\label{eqn:Qscaling}
    \frac{1}{4}\varphi\left(\frac{x}{2}\right) = \sum_{k\in \mathbb{Z}^2} \varphi(x-k) a_k^0.
\end{equation}
We refer to \eqref{eqn:Qscaling} as the \emph{ scaling} or \emph{dilation equation}. In the Fourier domain, this is best viewed in terms of the spinor-vector matrix.

\begin{proposition}\label{prop:Qfourierscaling}
    Let $\varphi \in L^2(\mathbb{R}^2,\mathbb{R}_2)$  be a scaling function as in Definition~\ref{def:QMRA}. Then \eqref{eqn:Qscaling} is equivalent to
    \begin{equation}\label{eqn:Qfourierscaling}
        [\hat{\varphi}(2\xi)]  = [\hat{\varphi}(\xi)]  [m_0(\xi)] 
    \end{equation}
    where $m_0(\xi) = \sum_{k\in\mathbb{Z}^2} e^{2\pi \xi \wedge k}a_k^0$.
\end{proposition}
\begin{proof}
    See \cite[Lemma~2]{hoganmorris1} or \cite[Lemma~4.1]{morristhesis}.
\end{proof}

We call $m_0$ the \emph{scaling filter}. Repeated application of  \eqref{eqn:Qfourierscaling} yields
\begin{equation*}
    [ \hat{\varphi}(\xi)]  =  \left[\varphi\left({\frac{\xi}{2^J}}\right)\right]\left[m_0\left({\frac{\xi}{2^J}}\right)\right]\left[m_0\left({\frac{\xi}{2^{J-1}}}\right)\right]\cdots \left[m_0\left({\frac{\xi}{2}}\right)\right],
\end{equation*}
and the infinite product $[ \hat{\varphi}(\xi)]  = \prod_{j=1}^{\infty} \left[  m_0\left(\frac{\xi}{2^j}\right)\right] $ gives a way to recover $\varphi$ from $m_0$ whenever the product converges.

We further note that condition~\ref{def:QMRAcond4} in Definition~\ref{def:QMRA} may be generalized to $f(\cdot) \in V_0 \Longleftrightarrow f(M^j\cdot) \in V_{j+1}$ for all $j \in \mathbb{Z}$, where $M$ is a $2\times 2$ dilation matrix, i.e., $M$ has integer entries and its eigenvalues are greater than one in absolute value. For any dilation matrix $M$, there are $\det(M)$ cosets of $M\mathbb{Z}^2$ in $\mathbb{Z}^2$ \cite[Proposition~2.1.1]{krivoshein}. For our purpose, $M=2I_2$ and $\det(M)-1=3$. Thus, we have three other cosets of $2\mathbb{Z}^2$ in $\mathbb{Z}^2$ and each of them corresponds to a \emph{wavelet function} (or simply \emph{wavelet}) $\psi^{\eps}$ ($\eps \in \{1,2,3\}$) such that
$$\left\{2^{j} \psi^{\eps}(2^j\cdot  -k)\ : \ j\in \mathbb{Z},\ k \in \mathbb{Z}^2,\ \eps \in \{1,2,3\} \right\}$$
forms an orthonormal basis for $L^2(\mathbb{R}^2,\mathbb{R}_2)$.

Since each $\psi^{\eps}(\frac{x}{2}) \in V_{-1}\subset V_0$, for each $\eps \in \{1,2,3\}$, there exists $(a_k^{\eps})_{k \in \mathbb{Z}^2} \in \ell^2(\mathbb{Z}^2,\mathbb{R}_2)$ such that 
\begin{equation}\label{eqn:Qwaveletscaling}
    \frac{1}{4}\psi^{\eps}\left(\frac{x}{2}\right) = \sum_{k\in \mathbb{Z}^2} \varphi(x-k) a_k^{\eps}.
\end{equation}
Taking the quaternionic Fourier transform of both sides of \eqref{eqn:Qwaveletscaling} yields an equivalent expression in terms of spinor-vector matrices.

\begin{proposition}\label{prop:Qfourierwaveletscaling}
    Let $\varphi \in L^2(\mathbb{R}^2,\mathbb{R}_2)$  be a scaling function as in Definition~\ref{def:QMRA} and $\psi^{\eps}$ be the associated wavelets ($\eps \in \{1,2,3\}$). Then \eqref{eqn:Qwaveletscaling} is equivalent to
    \begin{equation}\label{eqn:Qfourierwaveletscaling}
        [ \hat{\psi}^{\eps}(2\xi)]  = [ \hat{\varphi}(\xi)]  [  m_{\eps}(\xi)] 
    \end{equation}
    where $m_{\eps}(\xi) = \sum_{k\in\mathbb{Z}^2} e^{2\pi \xi\wedge k}a_k^{\eps}$, for each $\eps \in \{1,2,3\}$.
\end{proposition}
\begin{proof}
    See \cite[Lemma~2]{hoganmorris1} or \cite[Lemma~4.1]{morristhesis}.
\end{proof}

\subsection{Quaternionic wavelet design conditions}
In some instances, we write $\psi^0 = \varphi$ for convenience in collectively referring to the scaling function and its associated wavelets in the \textit{wavelet ensemble} $\{\psi^{\eps}\}_{\eps=0}^{3}$. We also adopt the notation
\begin{equation*}
    \psi_{jk}^{\eps}(x):=2^{j} \psi^{\eps}(2^jx-k)	
\end{equation*}
to refer to the normalized versions of the translated and dilated scaling function and wavelets. With this notation, $\{\varphi_{jk}\}_{k\in\mathbb{Z}^2}$ and $\{\psi_{jk}^{\eps} \, : \, \eps \in \{1,2,3\}\}_{k\in\mathbb{Z}^2}$ are orthonormal bases for $V_j$ and $W_j$, respectively.

\subsubsection*{Orthogonality}\label{ssec:Qorthogonality}

We first revisit the necessary and sufficient conditions for the orthonormality of the shifts of the scaling function and the wavelets. Henceforth, we let 
\begin{equation}\label{eqn:defV2}
    V^2 := \{v_0:=(0,0), v_1:=(1,0), v_2:=(0,1), v_3:=(1,1)\}
\end{equation}	
be the vertices of the unit square in $\mathbb{R}^2$.

\begin{proposition}\label{prop:QQMFscaling}
    Let  $\varphi$ be a scaling function that generates an MRA for $L^2(\mathbb{R}^2,\mathbb{R}_2)$ where $\{\varphi_{0k}\}_{k\in\mathbb{Z}^2}$ forms an orthonormal basis for $V_0$. If $m_0$ is the scaling filter associated with $\varphi$, then 
    \begin{equation}\label{eqn:QQMFscaling}
        \sum_{j=0}^{3} \left[  m_0\left(\xi + {\textstyle \frac{v_j}{2}}\right)\right] ^{\ast}\left[  m_0\left(\xi + {\textstyle \frac{v_j}{2}}\right)\right] = I_2
    \end{equation} 
    for almost every $\xi \in \mathbb{R}^2$.
\end{proposition}
\begin{proof}
    See \cite[Theorem~4]{hoganmorris1} or \cite[Theorem~4.3]{morristhesis}.
\end{proof}

A similar result is provided for the wavelet functions.

\begin{proposition}\label{prop:QQMFscalingwavelet}
    Let $\varphi$ be a scaling function associated with an MRA for $L^2(\mathbb{R}^2,\mathbb{R}_2)$ and $\{\psi^{\eps}\}_{\eps=1}^{3}$ be the collection of wavelets corresponding to $\varphi$. Suppose further that $m_0$ is the scaling filter and $\{m_{\eps}\}_{\eps=1}^{3}$ are the wavelet filters. The set $\{\psi^{\eps}_{0k} \, : \, \eps\in \{1,2,3\}\}_{k \in \mathbb{Z}^2}$ is an orthonormal basis for $W_0$ if and only if 
    for each $\eps,\zeta \in \{1,2,3\}$ we have
    \begin{align}
        \sum_{j=0}^{3} \left[  m_{\eps}\left(\xi + {\textstyle \frac{v_j}{2}}\right)\right] ^{\ast} \left[  m_{0}\left(\xi + {\textstyle \frac{v_j}{2}}\right)\right]  &= [0]  \label{eqn:QQMFscalingwavelet}\\
        \mbox{ and }\quad \sum_{j=0}^{3} \left[  m_{\eps}\left(\xi + {\textstyle \frac{v_j}{2}}\right)\right] ^* \left[  m_{\zeta}\left(\xi + {\textstyle \frac{v_j}{2}}\right)\right]  &= \delta_{\eps-\zeta}I_2 .\label{eqn:QQMFwavelets}
    \end{align}
\end{proposition}
\begin{proof}
    See \cite[Theorem~5]{hoganmorris1} or \cite[Theorem~4.5]{morristhesis}.
\end{proof}

Altogether, \eqref{eqn:QQMFscaling}--\eqref{eqn:QQMFwavelets} are the \emph{quaternionic quadrature mirror filter} (QQMF) \emph{conditions}. Given the orthonormality of $\{\varphi_{0k}\}_{k\in\mathbb{Z}^2}$, \eqref{eqn:QQMFscalingwavelet} and  \eqref{eqn:QQMFwavelets} become equivalent to the orthonormality of the shifts of the wavelet functions.

We now define the \textit{wavelet matrix}, which when forced to be unitary almost everywhere will encapsulate the QQMF conditions. 

\begin{definition}\label{def:Qwaveletmatrix}
    Let $\varphi$ be a scaling function associated with an MRA for $L^2(\mathbb{R}^2,\mathbb{R}_2)$, $\{\psi^{\eps}\}_{\eps=1}^{3}$ be the collection of wavelets corresponding to $\varphi$, and $\{v_j\}_{j=0}^3$ be defined as in \eqref{eqn:defV2}. Suppose further that $m_0$ is the scaling filter and $\{m_{\eps}\}_{\eps=1}^{3}$ are the wavelet filters. We call the $\mathbb{Z}^2$-periodic matrix-valued function $U:\mathbb{R}^2 \to \mathbb{R}_2^{8 \times 8}$ the \emph{wavelet matrix} whose entry-wise function value is given by
    \begin{equation*}
        U(\xi) = \begin{bmatrix}
            [m_0(\xi) ]&[ m_1(\xi)] & [m_2(\xi)] & [m_3(\xi)]\\
            [m_0(\xi+{\textstyle \frac{v_1}{2}})] & [m_1(\xi+{\textstyle \frac{v_1}{2}})] & [m_2(\xi+{\textstyle \frac{v_1}{2}})] & [m_3(\xi+{\textstyle \frac{v_1}{2}})] \\
            [m_0(\xi+{\textstyle \frac{v_2}{2}})] & [ m_1(\xi+{\textstyle \frac{v_2}{2}})] & [m_2(\xi+{\textstyle \frac{v_2}{2}})] &[m_3(\xi+{\textstyle \frac{v_2}{2}})] \\
            [m_0(\xi+{\textstyle \frac{v_3}{2}})] & [m_1(\xi+{\textstyle \frac{v_3}{2}})] & [m_2(\xi+{\textstyle \frac{v_3}{2}})] &[m_3(\xi+{\textstyle \frac{v_3}{2}})]\\
        \end{bmatrix}.
    \end{equation*}
\end{definition}
In the next section, we take a closer look at $U(\xi)$ and investigate its properties. With the wavelet matrix $U(\xi)$, the QQMF conditions in \eqref{eqn:QQMFscaling}--\eqref{eqn:QQMFwavelets} are together equivalent to $U(\xi)$ being unitary at almost every $\xi \in \mathbb{R}^2$. 

It is important to note that \eqref{eqn:QQMFscaling} is necessary but not sufficient for the shifts of $\varphi$ to form an orthonormal system.
Fortunately, an extension of Cohen's sufficient orthonormality condition \cite{cohenpaper} for quaternion-valued wavelets is also available.

\begin{proposition}\label{prop:Qorthonormalitycheck}
    Let $m_0(\xi) = \sum_{k\in\mathbb{Z}^2} e^{2\pi \xi \wedge k}a_k^0$ satisfy the QMF condition \eqref{eqn:QQMFscaling} where only finitely many of $(a_k^0)_{k\in\mathbb{Z}^2}$ are nonzero and $m_0(0)=1$. Let $[  m_0(\xi)] ^{\ast}[  m_0(\xi)]  > cI_2$ for all $\xi \in [-\frac{1}{4},\frac{1}{4}]^2$ for some $c \in (0,1)$, and
    \begin{equation*}
        [  \hat{\varphi}(\xi)]  = \lim_{n \to \infty}\left[  m_0\left(\frac{\xi}{2^n}\right) \right] \left[  m_0\left(\frac{\xi}{2^{n-1}}\right) \right] \cdots \left[  m_0\left(\frac{\xi}{2}\right) \right].
    \end{equation*}
    Then $\{\varphi(\cdot - k)\}_{k \in \mathbb{Z}^2}$ is an orthonormal set in $L^2(\mathbb{R}^2,\mathbb{R}_2)$.	  
\end{proposition}
\begin{proof}
    Refer to \cite[Theorem~5.7]{morristhesis}.
\end{proof}
For all cases that we  consider, $m_0$ is trigonometric polynomial and hence continuous. Under such an assumption, the condition $[  m_0(\xi)] ^{\ast}[  m_0(\xi)]  > cI_2$ for all $\xi \in [-\frac{1}{4},\frac{1}{4}]^2$ (where $0 <c<1$) simply demands the matrix $[m_0(\xi)] ^{\ast}[m_0(\xi)]$ to be positive definite for all $\xi \in [-\frac{1}{4},\frac{1}{4}]^2$. Thus, if this positive definiteness condition is met by $m_0$, while $m_{\eps}$ satisfies \eqref{eqn:QQMFscalingwavelet} and \eqref{eqn:QQMFwavelets} for $\eps \in \{1,2,3\}$, then the desired orthonormalities for $\{\varphi_{0k}\}_{k\in \mathbb{Z}^2}$ and $\{\psi_{0k}^{\eps} \, : \, \eps \in \{1,2,3\}\}_{k \in \mathbb{Z}^2}$ are guaranteed.

\subsubsection*{Completeness}

Given a scaling function $\varphi$ whose integer shifts form an orthonormal basis for $V_0$, one can readily generate the other subspaces $V_j$ so that $\bigcap_{j \in \mathbb{Z}}V_j = \{0\}$. However, to make sure that $\overline{\bigcup_{j \in \mathbb{Z}}V_j} = L^2(\mathbb{R}^2,\mathbb{R}_2)$, we further require $|\hat{\varphi}(0)| = 1$. For simplicity, we set $\hat{\varphi}(0) = 1$.  

\begin{proposition}\label{prop:Qcompleteness}
    Let $\varphi$ be a scaling function associated with an MRA for $L^2(\mathbb{R}^2,\mathbb{R}_2)$, and let $\{\psi^{\eps}\}_{\eps=1}^{3}$ be the collection of wavelets corresponding to $\varphi$. Suppose further that $m_0$ is the scaling filter and $\{m_{\eps}\}_{\eps=1}^{3}$ are the wavelet filters. If $\hat{\varphi}(0) = 1$, then 
    \begin{equation}\label{eqn:Qcompleteness}
        m_0({\textstyle\frac{v_j}{2}}) = \delta_{j} \ \text{ and } \ m_{\eps}(0) = 0.
    \end{equation}
    for all $j\in \{0,1,2,3\}$ and for all $\eps \in \{1,2,3\}$.
\end{proposition}
\begin{proof}
    Combining $\hat{\varphi}(0) = 1$ with \eqref{eqn:Qfourierscaling}, we have $[  m_0(0)]  = I_2$ which implies that $m_0(0)=1$. We also invoke \eqref{eqn:QQMFscaling} to derive that $\left[  m_0(\frac{v_j}{2})\right]  = [  0 ] $ which implies that $m_0(\frac{v_j}{2}) = 0$ for each $j \in \{1,2,3\}$. Finally, using \eqref{eqn:QQMFscalingwavelet}, we obtain $[  m_{\eps}(0)]  = [  0 ] $ implying that $m_{\eps}(0) = 0$ for all $\eps \in \{1,2,3\}$.
\end{proof}

We refer to \eqref{eqn:Qcompleteness} as the completeness condition. Note that this has an equivalent expression in terms of the quaternionic wavelet matrix. That is, $U(0)=[  1 ]  \oplus \tilde{U}$ where $\tilde{U}$ is a $3\times 3$ block matrix of spinor-vector matrices.

\subsubsection*{Compact support}

Another important design criterion that we want to consider is compact support. Compactly supported scaling and wavelet functions facilitate speedy and accurate computation of the coefficients in the wavelet decomposition of a given signal. 

\begin{proposition}\label{prop:Qcompactscaling}
    Let $\varphi$ be a scaling function associated with an MRA for $L^2(\mathbb{R}^2,\mathbb{R}_2)$, and let $m_0$ be the scaling filter. Then $\varphi$ is compactly supported if and only if  $m_0$ is a trigonometric polynomial. 
\end{proposition}
\begin{proof}
    See \cite[Theorem~10]{hoganmorris1} or \cite[Theorem~5.6]{morristhesis}.
\end{proof}

Through the wavelet equation \eqref{eqn:Qwaveletscaling}, the compact support property of the scaling function is inherited by the wavelets. This is shown in the following proposition.

\begin{proposition}\label{prop:Qcompactwavelet}
    Let $\varphi$ be a compactly supported scaling function associated with an MRA for $L^2(\mathbb{R}^2,\mathbb{R}_2)$, and let $\{\psi^{\eps}\}_{\eps=1}^{3}$ be the collection of wavelets associated to $\varphi$. Suppose further that $m_0$ is the scaling filter and $\{m_{\eps}\}_{\eps=1}^{3}$ are the wavelet filters. Then $\psi^{\eps}$ is compactly supported if and only if $m_{\eps}$ is a trigonometric polynomial. 
\end{proposition}
\begin{proof}
    For the forward direction, we deduce from \eqref{eqn:Qwaveletscaling} and the orthonormal shifts of $\varphi$ that $$\langle \varphi(\cdot-k),{\textstyle\frac{1}{4}\psi^{\eps}(\frac{\cdot}{2}}) \rangle = \int_{\mathbb{R}^2} \overline{\varphi(x-k)}\sum_{\ell \in \mathbb{Z}^2}\varphi(x-\ell)a_{\ell}^{\eps} \ dx = a_k^{\eps}.$$
    Since $\varphi$ and $\psi^{\eps}$ are both compactly supported, when $k$ is large enough the respective supports of $\varphi(x-k)$ and $\psi_{\eps}(\frac{x}{2})$ will be disjoint and so $a_k^{\eps}=0$. Thus, the sum in \eqref{eqn:Qwaveletscaling} is finite and $m_{\eps}$ is a trigonometric polynomial. For the reverse direction, if $m_{\eps}$ is a trigonometric polynomial then the sum in \eqref{eqn:Qwaveletscaling} is finite. Combining this with the compact support assumption for $\varphi$ yields the desired conclusion for $\psi^{\eps}$.
\end{proof}

Henceforth, we let $\eta \geq 4$ be even and $Q_{\eta}^2:= \{0,1,\ldots,\eta-1\}^2$. We say that $m: \mathbb{R}^2 \to \mathbb{R}_2$ is a trigonometric polynomial of degree $\eta-1$ whenever it is of the form
$$m(\xi) = \sum_{k\in Q_{\eta}^2} a_ke^{2\pi \xi \wedge k},$$  
where $(a_k)_{k\in Q_{\eta}^2} \subset \mathbb{R}_2$. We choose the support of $\varphi$ and $\psi^{\eps}$ to be $[0,\eta-1]^2$ and assume that $m_0$ and $m_{\eps}$ are trigonometric polynomials of degree $\eta-1$, for all $\eps \in \{1,2,3\}$.

\subsubsection*{Regularity}

For each multi-index $\alpha=(\alpha_1,\alpha_2) \in \mathbb{Z}_+^2$, we define $|\alpha|=\alpha_1 + \alpha_2$. For multi-indices $\alpha$ and $\beta$,  we write $\beta \leq \alpha$ to mean that $\beta_k \leq \alpha_k$ for $k \in \{1,2\}$, and we define the partial differential operator $\partial^{\alpha}$ by $$\partial^{\alpha} = \left(\frac{\partial}{\partial x_1}\right)^{\alpha_1}\left(\frac{\partial}{\partial x_2}\right)^{\alpha_2}.$$
Moreover, for $x\in \mathbb{R}^2$, we define $x^{\alpha}: = x_1^{\alpha_1}x_2^{\alpha_2}$ and $\alpha! = \alpha_1!\alpha_2!$. 

Furthermore, we require a product rule for differentiation of spinor-vector matrices of differentiable functions. For $\Omega \subset \mathbb{R}^2$, if $f,g:\Omega \to \mathbb{R}_2$ have continuous first-order partial derivatives and $|\alpha|=1$, the rule is given by
$$\frac{\partial}{\partial x_j}([f(x)][g(x)]) = \frac{\partial}{\partial x_j}([f(x)])[g(x)]  + [f(x)]  \frac{\partial}{\partial x_j}([g(x)])$$
where $j\in \{1,2\}$. Note that generally $\partial^{\alpha}([f(x)]) \neq [\partial^{\alpha}f(x)].$

The next proposition appeared in \cite[Theorem~11]{hoganmorris1} and \cite[Theorem~5.8]{morristhesis}. In its original statement, the conclusion only holds for $|\alpha| < m$. We improve the result so that the conclusion now holds for $|\alpha| \leq m$. 

\begin{proposition}\label{prop:Qvanishingmoments} Let $\varphi$ be a scaling function associated with an MRA for $L^2(\mathbb{R}^2,\mathbb{R}_2)$, and let $\{\psi^{\eps}\}_{\eps=1}^{3}$ be the collection of wavelets corresponding to $\varphi$. For $\eps \in \{1,2,3\}$, suppose $\psi^{\eps} \in C^{\mu}(\mathbb{R}^2,\mathbb{R}_2)$ where $\mu \in \mathbb{Z}_+$, and that there exists $C_{\mu}>0$ such that $|\psi^{\eps}(x)| \leq C_{\mu}(1+|x|)^{-q}$ with $q>\mu+1$. For each multi-index $\alpha=(\alpha_1,\alpha_2)$ with $|\alpha| \leq \mu$, suppose there exists $D_\alpha>0$ for which $|\partial^{\alpha}\psi^{\eps}(x)| \leq D_{\alpha}$ for all $x \in \mathbb{R}^2$, and that $\partial^{\alpha}\psi^{\eps} \in L^2(\mathbb{R}^2,\mathbb{R}_2)$. Then 
    \begin{equation}\int_{\mathbb{R}^2}x^{\alpha}\psi^{\eps}(x) dx =0 \, \mbox{ for } 0\leq|\alpha|\leq \mu.\label{eqn:Qvanishingmoments}
    \end{equation}
\end{proposition}

\begin{proof} See \cite[Theorem~11]{hoganmorris1} or \cite[Theorem~5.8]{morristhesis}. The improvement is attained by replacing the version of Taylor's theorem used in the proof by another where the remainder is expressed in Peano's form.
\end{proof}

\begin{proposition}\label{prop:Qregularity}
    Let $\varphi$ and $\{\psi^{\eps}\}_{\eps=1}^{3}$ be as in Proposition~\ref{prop:Qvanishingmoments} where $\psi^{\eps}$ satisfies $\int_{\mathbb{R}^2}x^{\alpha}\psi^{\eps}(x) dx = 0$, for $\eps \in \{1,2,3\}$ and $|\alpha| \leq \mu$. Suppose further that $m_0$ is the scaling filter and $\{m_{\eps}\}_{\eps=1}^{3}$ are the wavelet filters. Then the following statements hold and are equivalent.
    \begin{enumerate}[label=(\alph*)]
        \itemsep0em
        \item $\partial^{\alpha}\hat{\psi}^{\eps}(\xi)\bigr|_{\xi=0}=\ 0$ for $\eps \in \{1,2,3\}$, $0\leq |\alpha| \leq \mu$,\label{prop:Qregularity1}
        \item $\partial^{\alpha}m_{\eps}(\xi)\bigr|_{\xi=0}=\ 0$  for $\eps \in \{1,2,3\}$, $0\leq  |\alpha| \leq \mu$, \label{prop:Qregularity2}
        \item $\partial^{\alpha}m_0(\xi)\bigr|_{\xi = \frac{v_j}{2}}=0$  for $j \in \{1,2,3\}$, $0\leq |\alpha| \leq \mu$.\label{prop:Qregularity3}
    \end{enumerate}
\end{proposition}
\begin{proof} We first note that with the given assumptions, statement \ref{prop:Qregularity1} holds as proved in \cite[Theorem~5.9]{morristhesis}. We now show that statements \ref{prop:Qregularity1}, \ref{prop:Qregularity2} and \ref{prop:Qregularity3} are equivalent. 
    
    \ref{prop:Qregularity1}$\iff$\ref{prop:Qregularity2}: The forward implication is proved as part of \cite[Theorem~5.9]{morristhesis}. The backward direction is obtained by reversing the arguments made to establish the forward implication.
    
    \ref{prop:Qregularity2}$\iff$\ref{prop:Qregularity3}: Differentiating both sides of \eqref{eqn:QQMFscalingwavelet} using the product rule and the resulting version of Leibniz rule for partial derivatives yields 
    \begin{equation}\label{eqn:derivativeQMFquat}
        \sum_{j=0}^3\sum_{\beta\leq \alpha}\binom{\alpha}{\beta}\partial^{\beta}\left([  m_{\eps}(\xi+{\textstyle \frac{v_j}{2}})] ^{\ast}\right) \partial^{\alpha-\beta}\left([  m_0(\xi+{\textstyle \frac{v_j}{2}})] \right)= [  0 ] 
    \end{equation}
    for each $\eps\in \{1,2,3\}$.  For the backward direction, suppose statement \ref{prop:Qregularity3} holds. We again establish the desired result by an induction on $|\alpha|$. If $|\alpha|=0$, then it readily follows from completeness conditions that $m_{\eps}(0)=0$ for all $\eps\in \{1,2,3\}$. Now, suppose that $\partial^{\alpha}m_{\eps}(\xi)\bigr|_{\xi=0}=\ 0$ for each $\alpha$ with $|\alpha| < \ell < \mu$. Fix an $\alpha$ with $|\alpha|=\ell$. Using statement \ref{prop:Qregularity3} and setting $\xi=0$ in \eqref{eqn:derivativeQMFquat} yields
    \begin{equation*}
        \sum_{\beta\leq \alpha}\binom{\alpha}{\beta}\big(\partial^{\beta}\left([m_{\eps}(\xi)]^{\ast}\right) \partial^{\alpha-\beta}\left([m_0(\xi)]\right)\big)\big|_{\xi=0}= [0].
    \end{equation*}
    By the inductive hypothesis, the sum collapses to
$\big(\partial^{\alpha}\left([m_{\eps}(\xi)]^{\ast}\right) [m_0(\xi)]\big)\big|_{\xi=0}= [0]$,
    and since $[m_0(0)]=I_2$, we conclude that $\partial^{\alpha}\left([m_\eps(\xi)]^{\ast}\right)\big|_{\xi=0} = \left(\partial^{\alpha}[m_{\eps}(\xi)]\right)^{\ast}\big|_{\xi=0} =[0]$ where $|\alpha|=\ell$. Consequently, $\partial^{\alpha}m_{\eps}(\xi)\big|_{\xi=0}=0$ for $|\alpha|\leq\ell$ and for each $\eps \in \{1,2,3\}$. Therefore, the conclusion holds by induction. The forward direction is proved by a similar argument. 

\end{proof}

\subsubsection*{Symmetry}

For simplicity, we only consider \emph{point symmetry} about the center of support of the scaling function $\varphi$. That is, if $\varphi$ is supported on $[0,\eta-1]^2$ (for even $\eta \geq 4$) and $P=\left(\frac{\eta-1}{2},\frac{\eta-1}{2}\right)$, we want to impose the condition $\varphi(x) = \varphi(2P-x)$ for all $x\in \mathbb{R}^2$.

\begin{theorem}\label{thm:Qsymmequiv2D}
    For even $\eta\geq 4$, let $P=\left(\frac{\eta-1}{2},\frac{\eta-1}{2}\right)$. Suppose $\varphi$ is a scaling function for an MRA of $L^2(\mathbb{R}^2,\mathbb{R}_2)$ which is supported on $[0,\eta-1]^2$, and let $m_0$ be the associated scaling filter. Then the following statements are equivalent.
    \begin{enumerate}[label=(\alph*)]
        \itemsep0em
        \item $\varphi(x) = \varphi(2P-x)$,\label{thm:Qsymmequiv2Da}
        \item $\hat{\varphi}(\xi) = e^{4\pi \xi \wedge P}\hat{\varphi}(-\xi)$,\label{thm:Qsymmequiv2Db}
        \item $m_0(\xi) = e^{4\pi\xi \wedge P} m_0(-\xi)$.\label{thm:Qsymmequiv2Dc}
    \end{enumerate}
\end{theorem}

\begin{proof}
    
    \ref{thm:Qsymmequiv2Da}$\iff$\ref{thm:Qsymmequiv2Db}: Using the definition of the quaternionic Fourier transform, we have
    \begingroup
    \allowdisplaybreaks
    \begin{align*}
        \varphi(x) = \varphi(2P-x) \iff \hat{\varphi}(\xi) 
        &=\int_{\mathbb{R}^2} e^{2\pi\xi\wedge x} \varphi(2P-x) dx
        =\int_{\mathbb{R}^2} e^{2\pi\xi\wedge (2P-x)} \varphi(x) dx\\
        &=e^{4\pi \xi \wedge P}\int_{\mathbb{R}^2} e^{-2\pi\xi\wedge x} \varphi(x) dx=e^{4\pi \xi \wedge P}\hat{\varphi}(-\xi).
    \end{align*}
    \endgroup
    
    \ref{thm:Qsymmequiv2Db}$\iff$\ref{thm:Qsymmequiv2Dc}: Suppose $\hat{\varphi}(\xi) = e^{4\pi \xi \wedge P}\hat{\varphi}(-\xi)$. In terms of spinor-vector matrices, this statement becomes $$[\hat{\varphi}(\xi)] = [e^{4\pi \xi \wedge P}\hat{\varphi}(-\xi)] = [e^{4\pi \xi \wedge P}][\hat{\varphi}(-\xi)].$$
    From \eqref{eqn:Qfourierscaling}, we also know that $		[\hat{\varphi}(2\xi)] = [\hat{\varphi}(\xi)] [m_0(\xi)]$. Thus, by Proposition~\ref{prop:kernelcommutes} we obtain
    \begin{align*}
        [\hat{\varphi}(\xi)][m_0(\xi)] &= [\hat{\varphi}(2\xi)] = [e^{8\pi \xi \wedge P}][\hat{\varphi}(-2\xi)]
        =[e^{8\pi \xi \wedge P}][\hat{\varphi}(-\xi)][m_0(-\xi)]\\
        &=[e^{8\pi \xi \wedge P}][e^{-4\pi\xi \wedge P}][\hat{\varphi}(\xi)][m_0(-\xi)]
        =[\hat{\varphi}(\xi)][e^{4\pi\xi \wedge P}][m_0(-\xi)].
    \end{align*}
    Consequently,  $[m_0(\xi)]=[e^{4\pi\xi \wedge P}m_0(-\xi)]$ which yields the desired conclusion. Conversely, suppose $m_0(\xi) = e^{4\pi\xi \wedge P} m_0(-\xi)$. Then
    \begin{align*}
        [\hat{\varphi}(\xi)] &= \prod_{j=1}^{\infty}[m_0(2^{-j}\xi)] = \prod_{j=1}^{\infty}\left([e^{4\pi \xi \wedge P/2^{-j}}][m_0(-2^{-j}\xi)]\right)
        = \left(\prod_{j=1}^{\infty}[e^{4\pi \xi \wedge P/2^{-j}}]\right) \left(\prod_{j=1}^{\infty}[m_0(-2^{-j}\xi)]\right)\\
        &= [e^{4\pi \xi \wedge P}][\hat{\varphi}(-\xi)] = [e^{4\pi \xi \wedge P}\hat{\varphi}(-\xi)]
    \end{align*}
    which yields the desired conclusion.	
\end{proof}

\section{Quaternionic wavelet construction}\label{sec:Qwaveletconstruction}

In this section, we investigate the properties of the wavelet matrix $U(\xi)$ which was introduced in Definition~\ref{def:Qwaveletmatrix}. The goal of this section is to express the wavelet design criteria in terms of the quaternionic wavelet matrix $U(\xi)$.


We first show that the wavelet matrix may be viewed as a matrix-valued trigonometric series. 

\begin{lemma}\label{lem:filterproduct} The spinor-vector matrix of a trigonometric series $m(\xi) = \sum_{k\in \mathbb{Z}}e^{2\pi\xi\wedge k} a_k$ takes the form $$[m(\xi)] = \sum_{k \in \mathbb{Z}^2}[e^{2\pi\xi\wedge k}][a_k].$$
\end{lemma}

\begin{proof}
    By definition of the spinor-vector matrix of a function, we have
    \begingroup
    \allowdisplaybreaks 
    \begin{align*}
        [m(\xi)] &= \begin{bmatrix}	(m(\xi))_s & (m(\xi))_v\\ (m(-\xi))_v & (m(-\xi))_s\end{bmatrix}
        = \sum_{k\in \mathbb{Z}^2}\begin{bmatrix}e^{2\pi\xi\wedge k} & 0 \\  0 & e^{-2\pi\xi\wedge k}\end{bmatrix}\begin{bmatrix} (a_k)_s &  (a_k)_v\\ (a_k)_v & (a_k)_s\end{bmatrix}
        = \sum_{k \in \mathbb{Z}^2} [e^{2\pi\xi\wedge k}][a_k].			 
    \end{align*}
    \endgroup
\end{proof}
By applying Lemma~\ref{lem:filterproduct} to each of the filters $\{m_{\eps}\}_{\eps=0}^{3}$ forming the wavelet matrix $U(\xi)$, we can write 
\begin{equation}\label{eqn:waveletmatrixvalued}
    U(\xi) = \sum_{k\in \mathbb{Z}^2}[e^{2\pi\xi \wedge k}] \ostar A_k
\end{equation}
where $A_k$ may be viewed either as an element of $\mathbb{R}_2^{8 \times 8}$ or an element of $\mathsf{S}^{4\times 4}$. In performing the multiplication $\ostar$, the spinor-vector matrix $[e^{2\pi \xi \wedge k}]$ is multiplied to each of the spinor-vector entries of $A_k$.

\subsection{Consistency}\label{sec:Qconsistency}

As an immediate consequence of its $\mathbb{Z}^2$-periodicity, the wavelet matrix $U(\xi)$ satisfies what we call \emph{consistency conditions}.

\begin{proposition}\label{thm:Qconsistencysigma}
    Let the wavelet matrix $U(\xi)$ be defined as in Definition~\ref{def:Qwaveletmatrix} and $V^2$ as in \eqref{eqn:defV2}. If $\sigma_0 := \mbox{diag}([1],[1],[1],[1]) = I_8$ and $\sigma_3 := \sigma_1 \sigma_2 = \sigma_2\sigma_1$ where
    \begin{equation*}
        \sigma_1 := \begin{bmatrix}
            [0] & [1] & [0] & [0] \\
            [1] & [0] & [0] & [0] \\
            [0] & [0] & [0] & [1] \\
            [0] & [0] & [1] & [0]
        \end{bmatrix}
        \text{~and~~}\,
        \sigma_2:= \begin{bmatrix}
            [0] & [0] & [1] & [0] \\
            [0] & [0] & [0] & [1] \\
            [1] & [0] & [0] & [0] \\
            [0] & [1] & [0] & [0]
        \end{bmatrix},	
    \end{equation*} then $U(\xi)$ satisfies the following equivalent statements.
    \begin{enumerate}[label=(\alph*)]
        \itemsep0em
        \item $U(\xi + \frac{v_j}{2}) = \sigma_jU(\xi)$ for all $\xi \in \mathbb{R}^2$ and for every $v_j \in V^2$.\label{thm:Qconsistencysigma1}
        \item $U(\xi + \frac{v_{1}}{2}) =\sigma_{1}U(\xi)$ and $U(\xi + \frac{v_{2}}{2}) =\sigma_{2}U(\xi)$ for all $\xi \in \mathbb{R}^2$.\label{thm:Qconsistencysigma2} 
    \end{enumerate}
\end{proposition}
\begin{proof}
    Showing that the wavelet matrix satisfies \ref{thm:Qconsistencysigma1} is straightforward. Moreover, \ref{thm:Qconsistencysigma1}$\implies$\ref{thm:Qconsistencysigma2} directly holds. 
    Suppose statement \ref{thm:Qconsistencysigma2} holds. Then $U(\xi + \frac{v_1}{2}+\frac{v_1}{2}) = U(\xi) = \sigma_0U(\xi)$. Moreover, since $\sigma_3 = \sigma_1 \sigma_2 = \sigma_2\sigma_1$, we also have 
    \begin{equation*}
        \sigma_3U(\xi )=U(\xi + {\textstyle\frac{v_3}{2}}) = U(\xi + {\textstyle\frac{v_1}{2}+\frac{v_2}{2}}) =\sigma_{1}U(\xi + \textstyle{\frac{v_2}{2}}) = \sigma_1 \sigma_2 U(\xi).
    \end{equation*}
\end{proof}

Observe also that the spinor-vector matrix of the filter $m_{\eps}$ is of the form $$[m_{\eps}(\xi)] = \begin{bmatrix}	(m(\xi))_s & (m(\xi))_v\\ (m(-\xi))_v & (m(-\xi))_s\end{bmatrix}$$ 
for each $\eps \in \{0,1,2,3\}$, and therefore satisfies $[m_{\eps}(-\xi)] = \tau_0 [m_{\eps}(\xi)]\tau_0^{-1} = \tau_0 [m_{\eps}(\xi)]\tau_0$ where $\tau_0= \begin{bmatrix} 0 & 1 \\ 1 &0 \end{bmatrix} = \tau_0^{-1}.$ As a consequence, the wavelet matrix $U(\xi)$ satisfies another consistency condition as given in the following proposition.

\begin{proposition}\label{thm:QconsistencySV}
    Let the wavelet matrix $U(\xi)$ be defined as in Definition~\ref{def:Qwaveletmatrix}. If $\tau_0:= \begin{bmatrix} 0 & 1 \\ 1 &0 \end{bmatrix}$ and $\tau = \tau_0 \otimes \tau_0 \otimes \tau_0 \otimes \tau_0 = \mbox{diag}(\tau_0,\tau_0,\tau_0,\tau_0)$, then $$U(-\xi) = \tau U(\xi)\tau.$$
\end{proposition}
\begin{proof}
    This follows from the fact that $\tau^{-1} = \tau$ and $[m_{\eps}(-\xi)] = \tau_0 [m_{\eps}(\xi)]\tau_0$ for each $\eps \in \{0,1,2,3\}$, which comprise $U(\xi)$.
\end{proof}

\subsection{Design conditions for the wavelet matrix}

The quaternionic wavelet design conditions that were introduced in Section~\ref{sec:Qwaveletdesignconditions} are all expressible in terms of the wavelet matrix.  In what follows, we define $\partial^{\alpha}U(\xi)$ as the spinor-vector block matrix of partial derivatives of the spinor-vector matrices of the filters comprising $U(\xi)$, i.e., 
\begin{equation}\label{eqn:QpartialU}
    \resizebox{0.9\hsize}{!}{%
        $\partial^{\alpha}U(\xi):=\begin{bmatrix}
            \partial^{\alpha}[m_0(\xi) ]&\partial^{\alpha}[m_1(\xi)] & \partial^{\alpha}[m_2(\xi)] & \partial^{\alpha}[m_3(\xi)]\\
            \partial^{\alpha}[m_0(\xi+{\textstyle \frac{v_1}{2}})] & \partial^{\alpha}[m_1(\xi+{\textstyle \frac{v_1}{2}})] & \partial^{\alpha}[m_2(\xi+{\textstyle \frac{v_1}{2}})] & \partial^{\alpha}[m_3(\xi+{\textstyle \frac{v_1}{2}})] \\
            \partial^{\alpha}[m_0(\xi+{\textstyle \frac{v_2}{2}})] & \partial^{\alpha}[m_1(\xi+{\textstyle \frac{v_2}{2}})] & \partial^{\alpha}[m_2(\xi+{\textstyle \frac{v_2}{2}})] &\partial^{\alpha}[m_3(\xi+{\textstyle \frac{v_2}{2}})] \\
            \partial^{\alpha}[m_0(\xi+{\textstyle \frac{v_3}{2}})] & \partial^{\alpha}[m_1(\xi+{\textstyle \frac{v_3}{2}})] & \partial^{\alpha}[m_2(\xi+{\textstyle \frac{v_3}{2}})] &\partial^{\alpha}[m_3(\xi+{\textstyle \frac{v_3}{2}})]\\
        \end{bmatrix}.$}
\end{equation}  
The design criteria and the consistency conditions expressed in terms of $U(\xi)$ are summarized as follow.

\begin{theorem}\label{prop:Qwaveletdesignmatrix}
    Let $\eta\geq4$ be even, $P=\left(\frac{\eta-1}{2},\frac{\eta-1}{2}\right)$, $\varphi$ be a scaling function associated with an MRA for $L^2(\mathbb{R}^2,\mathbb{R}_2)$, and $\{\psi^{\eps}\}_{\eps=1}^{3}$ be the collection of wavelets corresponding to $\varphi$. Let $\sigma_1,\sigma_2$ be as in Proposition~\ref{thm:Qconsistencysigma} and $\tau$ be as in Proposition~\ref{thm:QconsistencySV}. Suppose $m_0$ is the scaling filter and $\{m_{\eps}\}_{\eps=1}^{3}$ are the wavelet filters that appear in the wavelet matrix $U(\xi)$ as in Definition~\ref{def:Qwaveletmatrix}. Then the following statements hold.
    \begin{enumerate}[label=(\alph*)]
        \itemsep0em
        \item If the integer shifts of $\psi^{\eps}$ ($\eps \in \{1,2,3\}$) form an orthonormal system, then $U(\xi)$ is unitary for almost every $\xi \in \mathbb{R}^2$.\label{prop:Qwaveletdesignmatrix1}
        \item If $\hat{\varphi}(0) = 1$, then $U(0) \in [1] \otimes \mathcal{U}^{3\times 3}$.\label{prop:Qwaveletdesignmatrix2}
        \item For $\eps \in \{1,2,3\}$, $\varphi$ and $\psi^{\eps}$ are compactly supported on $[0,\eta-1]^2$ if and only if $U(\xi)$ is a matrix-valued trigonometric polynomial of degree $\eta-1$ of the form 
        \begin{equation*}U(\xi) = \sum_{k\in Q_{\eta}^2} [e^{2\pi\xi\wedge k}]\ostar A_k \end{equation*} 
        where $A_k \in \mathsf{S}_2^{4 \times 4}$ for each $k\in Q_{\eta}^2$.\label{prop:Qwaveletdesignmatrix3}
        \item If $\psi^{\eps}$ has $\mu+1$ vanishing moments, then $$\partial^{\alpha}U(\xi)\big|_{\xi=0} \in \mathsf{S}\otimes\mathsf{S}^{3\times 3} \, \mbox{ for } 0\leq|\alpha|\leq \mu.$$ \label{prop:Qwaveletdesignmatrix4}
        \item $U(\xi + \frac{v_{1}}{2}) =\sigma_{1}U(\xi)$ and $U(\xi + \frac{v_{2}}{2}) =\sigma_{2}U(\xi)$ for all $\xi \in \mathbb{R}^2$.\label{prop:Qwaveletdesignmatrix5}
        \item $U(-\xi) = \tau U(\xi)\tau$ for all $\xi \in \mathbb{R}^2$.\label{prop:Qwaveletdesignmatrix6}
        \item $\varphi(x)=\varphi(2P-x)$ for all $x\in \mathbb{R}^2$ if and only if $$U(\xi)[0,0] + U(\xi)[0,1] = e^{4\pi\xi\wedge P}(U(-\xi)[0,0] + U(-\xi)[0,1])$$ for all $\xi \in \mathbb{R}^2$.\label{prop:Qwaveletdesignmatrix7}
    \end{enumerate}
\end{theorem}
\begin{proof}
    Statement~\ref{prop:Qwaveletdesignmatrix1} follows from Propositions~\ref{prop:QQMFscaling}--\ref{prop:QQMFscalingwavelet}; \ref{prop:Qwaveletdesignmatrix2} is deduced from Proposition~\ref{prop:Qcompleteness}; \ref{prop:Qwaveletdesignmatrix3} is  equivalent to  Propositions~\ref{prop:Qcompactscaling}--\ref{prop:Qcompactwavelet}; and \ref{prop:Qwaveletdesignmatrix4} is derived from Proposition~\ref{prop:Qregularity}. Statements \ref{prop:Qwaveletdesignmatrix5}--\ref{prop:Qwaveletdesignmatrix6} are just the consistency conditions from Propositions~\ref{thm:Qconsistencysigma}-- \ref{thm:QconsistencySV}, respectively. And finally, \ref{prop:Qwaveletdesignmatrix7} is the symmetry condition from Theorem~\ref{thm:Qsymmequiv2D}.
\end{proof}

\subsection{Discretisation by uniform sampling}

As in Proposition~\ref{prop:Qwaveletdesignmatrix}\ref{prop:Qwaveletdesignmatrix3}, the compact support condition allows us to write $U(\xi)$ as a matrix-valued trigonometric polynomial of degree $\eta-1$, i.e., 
\begin{equation}U(\xi) = \sum_{k\in Q_{\eta}^2} [e^{2\pi\xi\wedge k}] \ostar A_k.\label{eqn:Qsupportedwaveletmatrix}
\end{equation}
Such a polynomial has at most $\eta^2$ nonzero coefficients. Hence, it is completely determined by $\eta^2$ distinct samples.

Furthermore, the wavelet matrix $U(\xi)$ is $\mathbb{Z}^2$-periodic in $L^2([0,1]^2,\mathsf{S}^{4\times 4})$. We define a \emph{sampling operator} $D_{\eta}^2: C([0,1]^2,\mathsf{S}^{4\times 4}) \to (\mathsf{S}^{4\times 4})^{Q_{\eta}^2}$ by
\begin{equation}
    D_{\eta}^2(U(\xi)) = \left(U_j:=U({\textstyle\frac{j}{\eta}})\right)_{j \in Q_{\eta}^2}.
\end{equation}
Here, $C([0,1]^2,\mathsf{S}^{4\times 4})$ is the collection of continuous functions from $[0,1]^2$ to $(\mathsf{S}^{4\times 4})^{Q_{\eta}^2}$ and $(\mathsf{S}^{4\times 4})^{Q_{\eta}^2}$ is the collection of \emph{matrix ensembles} (not necessarily attached to a trigonometric polynomial) that can be indexed by $Q_{\eta}^2$. That is, 
\begin{equation}\label{eqn:Qsetofensemble}
    (\mathsf{S}^{4\times 4})^{Q_{\eta}^2} = \left\{\vec{B}=(B_j)_{j \in Q_{\eta}^2} \, : \, B_j \in \mathsf{S}^{4\times 4}\right\}.
\end{equation}
The output of sampling operator $D_\eta^2$ is a matrix ensemble $\vec{U}:=(U_j)_{j \in Q_{\eta}^2}$ which contains the desired $\eta^2$ samples of $U(\xi)$. Henceforth, we refer to  $\vec{U}=(U_j)_{j \in Q_{\eta}^2}$ as the \emph{standard ensemble of samples} (or \emph{standard samples}) of $U(\xi)$. Observe that the ensemble $\vec{A}:=(A_k)_{k \in Q_{\eta}^2}$ of coefficients in \eqref{eqn:Qsupportedwaveletmatrix} also lives in $(\mathsf{S}^{4\times 4})^{Q_{\eta}^2}$.

Moreover, the entries in the matrix ensemble $\vec{U}=(U_j)_{j \in Q_{\eta}^2}$ are of the form \begin{equation}\label{eqn:Qsamples}
    U_j = \sum_{k\in Q_{\eta}^2} [e^{2\pi j\wedge k/\eta}]\ostar A_k
\end{equation} 
for each $j\in Q_{\eta}^2$. We can relate the ensemble $\vec{U}=(U_j)_{j \in Q_{\eta}^2}$ of samples and the ensemble $\vec{A} = (A_k)_{k \in Q_{\eta}^2}$ of coefficients using an appropriate discrete Fourier transform.

\subsubsection*{Discrete Fourier transform}

We now define an appropriate discrete Fourier transform that acts on matrix ensembles in $(\mathsf{S}^{4\times 4})^{Q_{\eta}^2}$.
\begin{definition}\label{def:QFourierU}
    For any $\vec{B} = (B_k)_{k \in Q_{\eta}^2} \in  (\mathsf{S}^{4\times 4})^{Q_{\eta}^2}$, we define the \emph{discrete Fourier transform} $\mathcal{F}_{\eta}:(\mathsf{S}^{4\times 4})^{Q_{\eta}^2}\to (\mathsf{S}^{4\times 4})^{Q_{\eta}^2}$ by
    \begin{equation}
        (\mathcal{F}_{\eta} \vec{B})_j = \sum_{k\in Q_{\eta}^2}[e^{2\pi j\wedge k/\eta}]\ostar B_k.
    \end{equation}
\end{definition}
Notice that evaluating the wavelet matrix $U(\xi)$ in \eqref{eqn:Qsupportedwaveletmatrix} at each of the sample points in $(\frac{j}{\eta})_{j\in Q_{\eta}^2}$ yields equation
\eqref{eqn:Qsamples}. In view of the discrete Fourier transform, we may write $\vec{U} = (U_j)_{j\in Q_{\eta}^2}$ as $\vec{U} = \mathcal{F}_{\eta}\vec{A}$. In the following proposition, we provide the appropriate inversion of $\mathcal{F}_{\eta}$.

\begin{proposition}
    Let $\mathcal{F}_{\eta}$ be as in Definition~\eqref{def:QFourierU}. The inverse of $\mathcal{F}_{\eta}$ is given by
    \begin{equation}\label{eqn:QinverseFourierU}
        (\mathcal{F}_{\eta}^{-1}{\vec{C}})_k=\dfrac{1}{\eta^2}\sum_{j \in Q_{\eta}^2} [e^{2\pi k\wedge j/\eta}]\ostar C_j 
    \end{equation}
    where $\vec{C} = (C_j)_{j \in Q_{\eta}^2} \in  (\mathsf{S}^{4\times 4})^{Q_{\eta}^2}$.
\end{proposition}
\begin{proof}
    Let $\vec{B} = (B_k)_{k \in Q_{\eta}^2}$ and $\vec{C} = \mathcal{F}_{\eta}\vec{B}$, i.e., $C_j =\sum_{k\in Q_{\eta}^2}[e^{2\pi j\wedge k/\eta}]\ostar B_k$. It suffices to show that  $\mathcal{F}_{\eta}^{-1}\vec{C} = \vec{B}$. Now, by expansion and by using the definition of the wedge product we have
    \begin{align}
        (\mathcal{F}_{\eta}^{-1}\vec{C})_{\ell} &= \frac{1}{\eta^2}\sum_{j\in Q_{\eta}^2}[e^{2\pi \ell\wedge j/\eta}]\ostar C_j \nonumber
        = \frac{1}{\eta^2}\sum_{j\in Q_{\eta}^2}[e^{2\pi \ell\wedge j/\eta}]\ostar\left( \sum_{k\in Q_{\eta}^2}[e^{2\pi j\wedge k/\eta}]\ostar B_k \right)\nonumber\\
        &= \sum_{k\in Q_{\eta}^2} \frac{1}{\eta^2} \sum_{j\in Q_{\eta}^2} [e^{2\pi \ell\wedge j/\eta}] [e^{2\pi j\wedge k/\eta}]\ostar B_k 
        = \sum_{k\in Q_{\eta}^2} \left(\frac{1}{\eta^2}  \sum_{j\in Q_{\eta}^2} [e^{2\pi (\ell-k)\wedge j/\eta}]\right)\ostar B_k.\label{eqn:inverseQ}
    \end{align}
    By direct calculation we have $\frac{1}{\eta^2}\sum_{j\in Q_{\eta}^2}e^{2\pi(\ell-k)\wedge j/\eta} = \delta_{\ell-k}$ 
 and therefore, 
    $\frac{1}{\eta^2}\sum_{j\in Q_{\eta}^2}[e^{2\pi(\ell-k)\wedge j/\eta}]=\delta_{\ell-k}I_2.$ Combining this with \eqref{eqn:inverseQ} yields $(\mathcal{F}_{\eta}^{-1}\vec{C})_{\ell} = \sum_{k\in Q_{\eta}^2}\delta_{\ell-k}B_k = B_{\ell.}$
\end{proof}

In summary, if $\vec{A}=(A_k)_{k\in Q_{\eta}^2}$ is the ensemble of matrix coefficients in \eqref{eqn:Qsupportedwaveletmatrix}, then the inverse Fourier transform gives $\vec{A} = \mathcal{F}_{\eta}^{-1}\vec{U}$. Thus, we are able to establish a connection between the ensemble of samples $\vec{U}$ of $U(\xi)$ and ensemble of coefficients $\vec{A}$ of $U(\xi)$.

At this point, we can also explicitly state the relationship between the coefficients of the (scaling or wavelet) filter $m_{\eps}$ ($\eps \in \{0,1,2,3\}$) and its standard samples. This immediately follows from the connection established between the standard ensemble of samples and the ensemble of coefficients of $U(\xi)$.

\begin{remark}\label{rem:Qfiltersamples}
    Let $\eta\geq4$ be even, $\varphi$ be a scaling function associated with an MRA for $L^2(\mathbb{R}^2,\mathbb{R}_2)$, and $\{\psi^{\eps}\}_{\eps=1}^{3}$ be the collection of wavelets corresponding to $\varphi$. Suppose $m_0$ is the scaling filter and $\{m_{\eps}\}_{\eps=1}^{3}$ are the wavelet filters that appear in the wavelet matrix $U(\xi) = \sum_{k\in Q_{\eta}^2}[e^{2\pi\xi\wedge k}]\ostar A_k$. For $1\leq \eps \leq 3$,
    $$m_{\eps}({\textstyle\frac{j}{\eta}}) = \sum_{k\in Q_{\eta}^2} e^{2\pi j\wedge k/\eta}a_k^{\eps} \text{~~and~~} a_k^{\eps} = \frac{1}{\eta^2}\sum_{j\in Q_{\eta}^2} e^{2\pi k\wedge j/\eta}m_{\eps}({\textstyle\frac{j}{\eta}}).$$
\end{remark}

\subsubsection*{Other ensembles of samples}
In some instances, we need to obtain other samples of $U(\xi)$ apart from the standard ensemble. Leveraging the properties of quaternion Fourier transform, we can represent these additional ensembles of samples in terms of the standard ensemble.

\begin{definition}\label{def:Qmodulation}
    For a fixed $v_{\ell} \in V^2$ as in \eqref{eqn:defV2}, we define a \emph{modulation operator} $\chi_{\ell}: (\mathsf{S}^{4\times 4})^{Q_{\eta}^2} \to (\mathsf{S}^{4\times 4})^{Q_{\eta}^2}$ by
    \begin{equation*}
        (\chi_{\ell} \vec{B})_k = [e^{\pi v_{\ell} \wedge k/\eta}]\ostar B_k.
    \end{equation*}
\end{definition}

\begin{proposition}
    Let $v_{\ell} \in V^2\backslash\{v_0\}$ and $\vec{U}=(U_k)_{k\in Q_{\eta}^2}$ be the standard ensemble of samples of $U(\xi)$. If $\mathcal{F}_{\eta}$ is as in Definition~\ref{def:QFourierU} and $\chi_{\ell}$ is as in Definition~\ref{def:Qmodulation}, then $$\left(U\left({\textstyle \frac{j+v_{\ell}/2}{\eta}} \right)\right)_{j\in Q_{\eta}^2}  =  (\mathcal{F}_{\eta}\chi_{\ell}\mathcal{F}_{\eta}^{-1}\vec{U})_j\quad (j\in Q_\eta^2).$$
\end{proposition}
\begin{proof}
    For a fixed $v_{\ell} \in V^2$, we have 
    \begingroup
    \allowdisplaybreaks
    \begin{align*}
        U\left({\textstyle \frac{j+v_{\ell}/2}{\eta}} \right) &= \sum_{k\in Q_{\eta}^2} [e^{2\pi(j+v_{\ell}/2)\wedge k/\eta}] \ostar A_k
        = \sum_{k\in Q_{\eta}^2} [e^{2\pi j\wedge k/\eta}][e^{\pi v_{\ell} \wedge k/\eta}] \ostar A_k\\
        &= \sum_{k\in Q_{\eta}^2} [e^{2\pi j\wedge k/\eta}]\ostar\left([e^{\pi v_{\ell} \wedge k/\eta}] \ostar A_k\right)
        = \sum_{k\in Q_{\eta}^2} [e^{2\pi j\wedge k/\eta}]\ostar(\chi_{\ell}\vec{A})_k\\
        & = \sum_{k\in Q_{\eta}^2} [e^{2\pi j\wedge k/\eta}]\ostar(\chi_{\ell}\mathcal{F}_{\eta}^{-1}\vec{U})_k
        = (\mathcal{F}_{\eta}\chi_{\ell}\mathcal{F}_{\eta}^{-1}\vec{U})_j		
    \end{align*}
    \endgroup
    for each $j \in Q_{\eta}^2$. 
\end{proof}
Notice now that for a fixed $v_{\ell} \in V^2$, we know that $\vec{U}^{(\ell)}:=\mathcal{F}_{\eta}\chi_{\ell} \mathcal{F}_{\eta}^{-1}\vec{U}$ is an ensemble of $\eta^2$ samples of $U(\xi)$ obtained from $\vec{U}$ via $U_{j}^{(\ell)} = U(\frac{j+v_{\ell/2}}{\eta})$. In particular, $\vec{U}^{(0)} = \vec{U}$ is the standard ensemble. Computing  $\mathcal{F}_{\eta}\chi_{\ell} \mathcal{F}_{\eta}^{-1}\vec{U}$ for each $v_{\ell} \in V^2$ generates a total of $(2\eta)^2$ samples, including the standard ones. Collecting all such samples is equivalent to generating the samples $(U(\frac{j}{2\eta}))_{j\in Q_{2\eta}^2}$. Hence, the composition $\mathcal{F}_{\eta}\chi_{\ell} \mathcal{F}_{\eta}^{-1}$ (for $v_{\ell} \in V^2\backslash \{v_0\}$) allows us to obtain other samples of $U(\xi)$ using only our knowledge of the standard ensemble $\vec{U}$. 

\subsubsection*{Discretized design conditions}

In dealing with samples of $U(\xi)$, we note that by the $\mathbb{Z}^2$-periodicity of $U(\xi)$, if $k=(k_1,k_2)\in\mathbb{Z}^2$ but $k\notin Q_{\eta}^2$ then we interpret $U_{(k_1,k_2)}$ as $U_{(k_1\Mod{\eta},(k_2\Mod{\eta})}$.

The next theorem writes the design criteria outlined in Theorem~\ref{prop:Qwaveletdesignmatrix} in terms of the standard ensemble of samples.
\begin{theorem}\label{prop:Qdiscretewaveletdesignmatrix}
    Let $\eta \geq 4$ be even, $P=\left(\frac{\eta-1}{2},\frac{\eta-1}{2}\right)$, $U(\xi) = \sum_{k\in Q_{\eta}^2} [e^{2\pi\xi\wedge k}] \ostar A_k$ be a matrix-valued trigonometric polynomial as in \eqref{eqn:Qsupportedwaveletmatrix} where $(A_k)_{k\in Q_{\eta}^2} \subset \mathsf{S}_2^{4 \times 4}$, and $\vec{U}:=(U_j = U(\frac{j}{\eta}))_{j \in Q_{\eta}^2}$. Let $\sigma_1,\sigma_2$ be as in Proposition~\ref{thm:Qconsistencysigma} and $\tau$ be as in Proposition~\ref{thm:QconsistencySV}. Then the following statements hold.
    \begin{enumerate}[label=(\alph*)]
        \itemsep0em
        \item $U(\xi)$ is unitary everywhere if and only if the samples in $\{\vec{U}^{(\ell)} = \mathcal{F}_{\eta}\chi_{\ell} \mathcal{F}_{\eta}^{-1}\vec{U}\}_{\ell = 0}^{3}$ are all unitary.\label{prop:Qdiscretewaveletdesignmatrix1}
        \item If $\hat{\varphi}(0) = 1$, then $U_0 \in [1] \otimes \mathcal{U}^{3\times 3}$.\label{prop:Qdiscretewaveletdesignmatrix2}
        \item If the wavelets $\{\psi^{\eps}\}_{\eps=1}^{3}$ have $\mu+1$ vanishing moments, then the following are equivalent:
        \begin{enumerate}[label=(\roman*)]
            \itemsep0em
            \item $\partial^{\alpha}U(\xi)\big|_{\xi=0} \in \mathsf{S}\otimes\mathsf{S}^{3\times 3}$ for $0\leq|\alpha|\leq \mu$;\label{prop:Qdiscretewaveletdesignmatrix3a}
            \item $\sum_{k\in Q_{\eta}^2} [k_2^{\alpha_1} k_1^{\alpha_2}]\ostar A_k \in \mathsf{S}\otimes\mathsf{S}^{3\times 3} \, \mbox{ for } 0\leq|\alpha|\leq \mu$; and\label{prop:Qdiscretewaveletdesignmatrix3b}
            \item $\sum_{j\in Q_{\eta}^2} [c_{\alpha j}]\ostar U_j \in \mathsf{S}\otimes\mathsf{S}^{3\times 3}  \, \mbox{ for } 0\leq|\alpha|\leq \mu$, \label{prop:Qdiscretewaveletdesignmatrix3c}  
        \end{enumerate}
        where $[c_{\alpha j}]=\sum_{k\in Q_{\eta}^2}[k_2^{\alpha_1}k_1^{\alpha_2}][e^{2\pi k\wedge j/\eta}]$.\label{prop:Qdiscretewaveletdesignmatrix3}
        \item The following statements are equivalent:
        \begin{enumerate}[label=(\roman*)]
            \itemsep0em
            \item $U(\xi + \frac{v_{1}}{2}) =\sigma_{1}U(\xi)$ and $U(\xi + \frac{v_{2}}{2}) =\sigma_{2}U(\xi)$ for all $\xi \in \mathbb{R}^2$;\label{prop:Qdiscretewaveletdesignmatrix4a}
            \item $\sigma_{1}A_k = (-1)^{k_2}A_k$ and $\sigma_{2}A_k = (-1)^{k_1}A_k$ for all $k=(k_1,k_2) \in  Q_{\eta}^2$; and \label{prop:Qdiscretewaveletdesignmatrix4b}
            \item $U_{j + \eta v_{1}/2} =\sigma_{1}U_{j}$ and $U_{j + \eta v_{2}/2} =\sigma_{2}U_{j}$ for all $j \in Q_{\eta}^2$.\label{prop:Qdiscretewaveletdesignmatrix4c}		
        \end{enumerate}\label{prop:Qdiscretewaveletdesignmatrix4}
        \item The following statements are equivalent:\label{prop:Qdiscretewaveletdesignmatrix5}
        \begin{enumerate}[label=(\roman*)]
            \itemsep0em
            \item $U(-\xi) = \tau U(\xi)\tau$ for all $\xi \in \mathbb{R}^2$;\label{prop:Qdiscretewaveletdesignmatrix5a}
            \item $A_k = \tau A_k \tau$ for all $k\in Q_{\eta}^2$; and \label{prop:Qdiscretewaveletdesignmatrix5b}
            \item $U_{\eta v_3-j} = \tau U_j \tau$ for all $j\in Q_{\eta}^2$.\label{prop:Qdiscretewaveletdesignmatrix5c}
        \end{enumerate}
        \item The following statements are equivalent: \label{prop:Qdiscretewaveletdesignmatrix6}
        \begin{enumerate}[label=(\roman*)]
            \itemsep0em
            \item $\varphi(x)=\varphi(2P-x)$ for all $x\in \mathbb{R}^2$; \label{prop:Qdiscretewaveletdesignmatrix6a0}
            \item $U(\xi)[0,0] + U(\xi)[0,1] = e^{4\pi\xi\wedge P}(U(-\xi)[0,0] + U(-\xi)[0,1])$ for all $\xi \in \mathbb{R}^2$;\label{prop:Qdiscretewaveletdesignmatrix6a}
            \item $A_k[0,0] + A_k[0,1] = A_{2P-k}[0,0]+A_{2P-k}[0,1]$ for all $k\in Q_{\eta}^2$; and \label{prop:Qdiscretewaveletdesignmatrix6b}
            \item $U_j[0,0] + U_{j}[0,1] = e^{4\pi j\wedge P/\eta}(U_{\eta v_3-j}[0,0] + U_{\eta v_3-j}[0,1])$.\label{prop:Qdiscretewaveletdesignmatrix6c}
        \end{enumerate}
    \end{enumerate}
\end{theorem}
\begin{proof}
    Let $U(\xi) = \sum_{k\in Q_{\eta}^2} [e^{2\pi\xi\wedge k}] \ostar A_k$ be a trigonometric polynomial.
    
    \ref{prop:Qdiscretewaveletdesignmatrix1}: The forward implication is straightforward because if $U(\xi)$ is unitary everywhere, then all samples in  $\{\vec{U}^{(\ell)} = \mathcal{F}_{\eta}\chi_{\ell} \mathcal{F}_{\eta}^{-1}\vec{U}\}_{\ell = 0}^{3}$ are unitary. For the backward direction, suppose that all samples in $\{\vec{U}^{(\ell)} = \mathcal{F}_{\eta}\chi_{\ell} \mathcal{F}_{\eta}^{-1}\vec{U}\}_{\ell = 0}^{3}$ are unitary. Letting  $J_{\eta}^2 = \{1-\eta, 2-\eta, \ldots, 0,1,\ldots,\eta-1\}^2$, we first note that
    \begingroup
    \allowdisplaybreaks
    \begin{align}
        U(\xi)^{\ast}U(\xi)	&= \left(\sum_{j\in Q_{\eta}^2} A_j^{\ast} \ostar [e^{-2\pi \xi \wedge j}]\right) \left(\sum_{k\in Q_{\eta}^2} [e^{2\pi \xi \wedge k}]\ostar A_k \right)\nonumber\\
        &=\sum_{k\in Q_{\eta}^2} \left(\sum_{j\in Q_{\eta}^2}  A_j^{\ast} \ostar [e^{-2\pi \xi \wedge j}][e^{2\pi \xi \wedge k}]\right)A_k \nonumber
        =\sum_{k\in Q_{\eta}^2}\sum_{j\in Q_{\eta}^2}  [e^{2\pi \xi \wedge (k-j)}] \ostar \left(A_j^{\ast} A_k \right) \nonumber\\
        &=\sum_{\ell \in J_{\eta}^2} \sum_{j\in Q_{\eta}^2} [e^{2\pi \xi \wedge \ell}] \ostar \left(A_j^{\ast} A_{\ell+j} \right) 
        =\sum_{\ell \in J_{\eta}^2}  [e^{2\pi \xi \wedge \ell}] \ostar B_{\ell}\label{eqn:Qdiscretewaveletdesignmatrix2}
    \end{align}
    \endgroup
    where $B_{\ell} = \sum_{j\in Q_{\eta}^2}\left(A_j^{\ast} A_{\ell+j} \right)$. By the unitarity of all samples in $\cup_{\ell=0}^3\{\vec{U}^{(\ell)} = \mathcal{F}_{\eta}\chi_{\ell} \mathcal{F}_{\eta}^{-1}\vec{U}\}=(U(\frac{j}{2\eta}))_{j\in Q_{2\eta}^2}$, we have 
    \begin{equation*}
        I_8 = U({\textstyle \frac{j}{2\eta}})^{\ast} U({\textstyle \frac{j}{2\eta}}) = \sum_{\ell\in J_{\eta}^2} [e^{2\pi j \wedge \ell/(2\eta)}] \ostar B_{\ell}
    \end{equation*}
    for each $j \in Q_{2\eta}^2$. 
    By orthonormality and completeness of the Fourier basis $\{e_{\ell}\}_{\ell\in Q_{2\eta}^2}$ (where $e_{\ell}(j) =  e^{-2\pi e_{12}j\wedge \ell/(2\eta)} $) in $\ell_2(J_{\eta}^2,\mathbb{R}_2)$ we deduce that $B_{\ell} = \delta_{\ell}I_8$. Combining this with \eqref{eqn:Qdiscretewaveletdesignmatrix2}, we conclude that $U(\xi)^{\ast} U(\xi) = I_8$ for every $\xi \in \mathbb{R}^2$.
    
    \ref{prop:Qdiscretewaveletdesignmatrix2}: This is the consistency condition proved in Proposition~\ref{prop:Qcompleteness}. 
    
    \ref{prop:Qdiscretewaveletdesignmatrix3}: With the definition of $\partial^{\alpha}U(\xi)$ in \eqref{eqn:QpartialU}, we may write 
    \begingroup
    \allowdisplaybreaks
    \begin{align}
        \partial^{\alpha}U(\xi) &= \sum_{k\in Q_{\eta}^2} \partial^{\alpha}[e^{2\pi\xi\wedge k}] \ostar A_k \nonumber\\
        &=\sum_{k\in Q_{\eta}^2} \begin{bmatrix}(2\pi e_{12})^{|\alpha|}k_2^{\alpha_1}(-k_1)^{\alpha_2}e^{2\pi\xi\wedge k}&0\\0&(2\pi e_{12})^{|\alpha|}(-k_2)^{\alpha_1}k_1^{\alpha_2}e^{-2\pi\xi\wedge k} \end{bmatrix}\ostar A_k\nonumber\\
        &=\sum_{k\in Q_{\eta}^2} \begin{bmatrix}(2\pi e_{12})^{|\alpha|}(-1)^{\alpha_2}&0\\0&(2\pi e_{12})^{|\alpha|}(-1)^{\alpha_1} \end{bmatrix}[k_2^{\alpha_1}k_1^{\alpha_2}e^{2\pi\xi\wedge k}]\ostar A_k \nonumber\\
        &=\begin{bmatrix}(2\pi e_{12})^{|\alpha|}(-1)^{\alpha_2}&0\\0&(2\pi e_{12})^{|\alpha|}(-1)^{\alpha_1} \end{bmatrix} \ostar \sum_{k\in Q_{\eta}^2} [k_2^{\alpha_1}k_1^{\alpha_2}e^{2\pi\xi\wedge k}]\ostar A_k \label{eqn:QpartialUexplicit}
    \end{align}
    \endgroup
    for each $0\leq|\alpha|\leq \mu$. Evaluating \eqref{eqn:QpartialUexplicit} at $\xi=0$, we deduce that for each $0\leq|\alpha|\leq \mu$
    \begin{align*}
        \partial^{\alpha}U(\xi)&|_{\xi=0}=\begin{bmatrix}(2\pi e_{12})^{|\alpha|}(-1)^{\alpha_2}&0\\0&(2\pi e_{12})^{|\alpha|}(-1)^{\alpha_1} \end{bmatrix} \ostar \sum_{k\in Q_{\eta}^2} [k_2^{\alpha_1}k_1^{\alpha_2}]\ostar A_k \in \mathsf{S}\otimes\mathsf{S}^{3\times 3}\\
        &\iff \sum_{k\in Q_{\eta}^2} [k_2^{\alpha_1}k_1^{\alpha_2}]\ostar A_k \in \mathsf{S}\otimes\mathsf{S}^{3\times 3}.\end{align*} This proves that \ref{prop:Qdiscretewaveletdesignmatrix3a} $\iff$ \ref{prop:Qdiscretewaveletdesignmatrix3b}. 
    
    Now suppose $\sum_{k\in Q_{\eta}^2} [k_2^{\alpha_1}k_1^{\alpha_2}] \ostar A_k \in \mathsf{S}\otimes\mathsf{S}^{3\times 3}$
    for all $0\leq|\alpha|\leq \mu$. Since $\vec{A} = \mathcal{F}_{\eta}^{-1}\vec{U}$, we have
    \begingroup\allowdisplaybreaks
 $$ \sum_{k\in Q_{\eta}^2} [k_2^{\alpha_1}k_1^{\alpha_2}] \ostar A_k = \sum_{k\in Q_{\eta}^2} [k_2^{\alpha_1}k_1^{\alpha_2}] \ostar \left({\textstyle \dfrac{1}{\eta^2}}\sum_{j \in Q_{\eta}^2} [e^{2\pi k\wedge j/\eta}]\ostar U_j \right)
        = {\textstyle \dfrac{1}{\eta^2}}\sum_{j \in Q_{\eta}^2} [c_{
            \alpha j}]\ostar U_j$$
    where $[c_{
        \alpha j}]=\sum_{k\in Q_{\eta}^2}[k_2^{\alpha_1}k_1^{\alpha_2}][e^{2\pi k\wedge j/\eta}]$. We conclude that \ref{prop:Qdiscretewaveletdesignmatrix3b}  $\implies$ \ref{prop:Qdiscretewaveletdesignmatrix3c}. A similar computation involving the discrete Fourier transform proves the converse.
    \endgroup
    
    \ref{prop:Qdiscretewaveletdesignmatrix4}: Suppose $U(\xi)$ satisfies $U(\xi + \frac{v_{1}}{2}) =\sigma_{1}U(\xi)$. Then we have 
    \begin{align*}
        U({\textstyle \xi + \frac{v_1}{2}}) = \sum_{k\in Q_{\eta}^2}[e^{2\pi\xi \wedge k}(-1)^{k_2}]\ostar A_k = \sum_{k\in Q_{\eta}^2}[e^{2\pi\xi \wedge k}]\ostar ((-1)^{k_2}A_k).
    \end{align*}
    On the other hand, $\sigma_{1}U(\xi) = \sum_{k\in Q_{\eta}^2}[e^{2\pi\xi \wedge k}]\ostar \sigma_1A_k$. Comparing the coefficient matrices in $U({\textstyle \xi + \frac{v_1}{2}})$ and $\sigma_{1}U(\xi)$ gives $\sigma_{1}A_k = (-1)^{k_2}A_k$. By a similar argument, it can be shown that $\sigma_{2}A_k = (-1)^k_{1}A_k$ and so we conclude that \ref{prop:Qdiscretewaveletdesignmatrix4a}  $\implies$ \ref{prop:Qdiscretewaveletdesignmatrix4b}. The converse is proved similarly. 
    
    Suppose now that $\sigma_{1}A_k = (-1)^{k_2}A_k$. Then
    \begingroup \allowdisplaybreaks
    \begin{align*}
        U_{j + \frac{\eta v_1}{2}} &= \sum_{k\in Q_{\eta}^2} [e^{2\pi(j + \eta v_1/2) \wedge k/\eta}]\ostar A_k
        = \sum_{k\in Q_{\eta}^2} [e^{2\pi j\wedge k/\eta}(-1)^{k_2}]\ostar A_k\\
        &= \sum_{k\in Q_{\eta}^2} [e^{2\pi j\wedge k/\eta}]\ostar ((-1)^{k_2}A_k)
        = \sum_{k\in Q_{\eta}^2} [e^{2\pi j\wedge k/\eta}]\ostar \sigma_1A_k\\
        &= \sigma_1\left(\sum_{k\in Q_{\eta}^2} [e^{2\pi j\wedge k/\eta}]\ostar A_k\right) = \sigma_1U_j
    \end{align*}
    \endgroup
    for each $j\in Q_{\eta}^2$. Similarly, it can be shown that $U_{j + \eta v_{2}/2} =\sigma_{2}U_{j}$ for each $j\in Q_{\eta}^2$. Thus, \ref{prop:Qdiscretewaveletdesignmatrix4b}  $\implies$ \ref{prop:Qdiscretewaveletdesignmatrix4c}. A similar calculation yields the converse. 
    
    \ref{prop:Qdiscretewaveletdesignmatrix5}: Suppose $U(-\xi) = \tau U(\xi)\tau$ for all $\xi \in \mathbb{R}^2$. Recall that $\tau = \tau_0 \otimes \tau_0 \otimes \tau_0 \otimes \tau_0$ where $\tau_0= \begin{bmatrix} 0 & 1 \\ 1 &0 \end{bmatrix}$ as in Proposition~\ref{thm:QconsistencySV}. We claim that $\tau_0 [e^{2\pi \xi \wedge k}] =  [e^{-2\pi \xi \wedge k}] \tau_0$. Indeed, 
    $$\begin{bmatrix} 0 & 1 \\ 1 &0 \end{bmatrix}\begin{bmatrix} e^{2\pi \xi \wedge k} & 0 \\ 0 & e^{-2\pi \xi \wedge k} \end{bmatrix} = \begin{bmatrix} 0 & e^{-2\pi \xi \wedge k} \\ e^{2\pi \xi \wedge k} &0 \end{bmatrix} = \begin{bmatrix} e^{-2\pi \xi \wedge k} & 0 \\ 0 & e^{2\pi \xi \wedge k} \end{bmatrix}\begin{bmatrix} 0 & 1 \\ 1 &0 \end{bmatrix}.$$
    Consequently,
    $\tau \ostar [e^{2\pi \xi \wedge k}] =  [e^{-2\pi \xi \wedge k}] \ostar \tau$. Thus,
    $$\tau U(\xi) \tau = \tau \left(\sum_{k\in Q_{\eta}^2}[e^{2\pi\xi \wedge k}]\ostar A_k\right) \tau = \sum_{k\in Q_{\eta}^2} [e^{-2\pi\xi \wedge k}]\ostar \tau A_k \tau.$$
    Comparing the coefficients of $\tau U(\xi) \tau$ with those of $U(-\xi) = \sum_{k\in Q_{\eta}^2}[e^{-2\pi\xi \wedge k}]\ostar A_k$, we conclude that $A_k = \tau A_k \tau$ for all $k\in Q_{\eta}^2$. Hence \ref{prop:Qdiscretewaveletdesignmatrix5a}$\implies$\ref{prop:Qdiscretewaveletdesignmatrix5b}. A similar calculation yields the converse. 
    
    Suppose now that $A_k = \tau A_k \tau$ for all $k\in Q_{\eta}^2$. Then
    \begin{align*}
        U_{\eta v_3-j} &= \sum_{k\in Q_{\eta}^2} [e^{2\pi (\eta v_3-j)\wedge k/\eta}]\ostar A_k
        = \sum_{k\in Q_{\eta}^2} [e^{-2\pi j\wedge k/\eta}]\ostar A_k\\
        &= \sum_{k\in Q_{\eta}^2} [e^{-2\pi j\wedge k/\eta}]\ostar \tau A_k \tau
        = \tau \left(\sum_{k\in Q_{\eta}^2} [e^{2\pi j\wedge k/\eta}]\ostar A_k\right) \tau= \tau U_{j}\tau
    \end{align*}
    for each $j\in Q_{\eta}^2$, and so we conclude that \ref{prop:Qdiscretewaveletdesignmatrix5b}  $\implies$ \ref{prop:Qdiscretewaveletdesignmatrix5c}. The converse is proved by a similar computation.
    
    \ref{prop:Qdiscretewaveletdesignmatrix6}: We first note that by definition of the quaternionic wavelet matrix $U(\xi)$, we have $$m_0(\xi) = \sum_{k\in Q_{\eta}^2}e^{2\pi \xi \wedge k}a_k^0 = U(\xi)[0,0] + U(\xi)[0,1].$$ Moreover, the equivalence of statements \ref{prop:Qdiscretewaveletdesignmatrix6a0} and \ref{prop:Qdiscretewaveletdesignmatrix6a} has already been established in Theorem~\ref{thm:Qsymmequiv2D}. Thus, we only need to prove that \ref{prop:Qdiscretewaveletdesignmatrix6a}, \ref{prop:Qdiscretewaveletdesignmatrix6b} and \ref{prop:Qdiscretewaveletdesignmatrix6c} are equivalent. 
    
    To show that \ref{prop:Qdiscretewaveletdesignmatrix6a}$\iff$\ref{prop:Qdiscretewaveletdesignmatrix6b}, observe that
    \begin{align*}
        m_0(\xi) = e^{4\pi\xi\wedge P}m_0(-\xi) \iff \sum_{k\in Q_{\eta}^2}e^{2\pi \xi \wedge k}a_k^0 &=  e^{4\pi\xi\wedge P}\sum_{k\in Q_{\eta}^2}e^{-2\pi \xi \wedge k}a_k^0
        =\sum_{k\in Q_{\eta}^2}e^{2\pi \xi \wedge (2P-k)}a_k^0\\
        &=\sum_{k\in Q_{\eta}^2}e^{2\pi \xi \wedge k}a_{2P-k}^0,
    \end{align*}
    which is equivalent to $a_k^0 = A_k[0,0] + A_k[0,1] = A_{2P-k}[0,0]+A_{2P-k}[0,1] = a_{2P-k}^0$ for all $k\in Q_{\eta}^2$. 
    
    In showing \ref{prop:Qdiscretewaveletdesignmatrix6b} $\iff$ \ref{prop:Qdiscretewaveletdesignmatrix6c}, the forward direction is proved by appealing to the equivalence of \ref{prop:Qdiscretewaveletdesignmatrix6a} and \ref{prop:Qdiscretewaveletdesignmatrix6b}, followed by standard sampling. For the converse, recall that $m_0(\frac{j}{\eta}) = U_j[0,0] + U_{j}[0,1]$ and so the premise may be rewritten as $m_0(\frac{j}{\eta}) = e^{4\pi j\wedge P/\eta}m_0(\frac{\eta v_3-j}{\eta})$. Combining this with Remark~\ref{rem:Qfiltersamples}, we obtain
    \begin{align*}
        a_k^0 &= \frac{1}{\eta^2}\sum_{j\in Q_{\eta}^2} e^{2\pi k \wedge j/\eta}m_0({\textstyle\frac{j}{\eta}}) = \frac{1}{\eta^2}\sum_{j\in Q_{\eta}^2} e^{2\pi k \wedge j/\eta} e^{4\pi j\wedge P/\eta}m_0({\textstyle\frac{\eta v_3-j}{\eta}})\\
        &=\frac{1}{\eta^2}\sum_{j\in Q_{\eta}^2}e^{2\pi j\wedge (2P-k)/\eta}m_0({\textstyle\frac{\eta v_3-j}{\eta}}) = \frac{1}{\eta^2}\sum_{j\in Q_{\eta}^2}e^{2\pi (2P-k)\wedge j/\eta}m_0({\textstyle\frac{j}{\eta}})= a_{2P-k}^0
    \end{align*} 
    for each $k\in Q_{\eta}^2$. Therefore, $a_k^0 = A_k[0,0] + A_k[0,1] = A_{2P-k}[0,0]+A_{2P-k}[0,1] = a_{2P-k}^0$ for all $k\in Q_{\eta}^2$.
\end{proof}

\section{The quaternionic wavelet feasibility problem}\label{sec:Qwaveletfeasibility}

In this section, we formally state the quaternionic wavelet feasibility problem. Specifically, we define the associated constraint sets and use them to give a precise mathematical formulation of the problem.

\subsection{Relevant Hilbert modules}

We first identify the Hilbert module that provides an appropriate setting for the quaternionic wavelet feasibility problem. Throughout this section, for a quaternion $q\in\mathbb{R}_2$, we write $q=q^{(0)}+q^{(1)}e_1+q^{(2)}e_2+q^{(3)}e_{12}$ where $q^{(0)}$ denotes the real part of $q$ and $q^{(1)}, q^{(2)}, q^{(3)}$ denote the imaginary parts of $q$. The real part of $q$ is also written as $[q]_0$, i.e., $q^{(0)}= [q]_0$. In this section, we clarify the appropriate inner product that makes the set $(\mathsf{S}^{4\times 4})^{Q_{\eta}^2}$ of matrix ensembles a Hilbert module.

\subsubsection*{Hilbert modules of quaternion vectors and matrices}

For a fixed $d\in\mathbb{N}$, we consider $\mathbb{R}_2^d$ as a right  $\mathbb{R}_2$-module. That is, for all $v\in \mathbb{R}_2^d$ and $a,b \in\mathbb{R}_2$, the scalar multiplication obeys the rule $v(ab) = (va)(b)$. For any $u=(u_k)_{k=1}^d,v=(v_k)_{k=1}^d \in \mathbb{R}_2^d$, we define the $\mathbb{R}_2$-valued right inner product $\langle \cdot,\cdot \rangle_{\mathbb{R}_2}$ on $\mathbb{R}_2^d$ by $$\langle u,v\rangle_{\mathbb{R}_2} = \sum_{k=1}^d \overline{u_k}v_k.$$
With this inner product, we say $u$ and $v$ are orthonormal if and only if $\langle u,v\rangle_{\mathbb{R}_2} =0$. 

Alternatively, we can regard $\mathbb{R}_2^d$ as a real inner product space with $\langle \cdot, \cdot\rangle_{\mathbb{R}} : \mathbb{R}_2^d \times \mathbb{R}_2^d \to \mathbb{R}$ defined by 
\begin{equation}\label{eqn:QrealinnerproductV}
    \langle u,v\rangle_{\mathbb{R}} = \sum_{\eps=0}^3\sum_{k=1}^d \langle u_k^{(\eps)},v_k^{(\eps)} \rangle.
\end{equation}
It is directly verifiable that $\langle u,v\rangle_{\mathbb{R}} = [\langle u,v\rangle_{\mathbb{R}_2}]_0$. 
Note that $[\langle u,v\rangle_{\mathbb{R}_2}]_0$ may also be computed by (a quaternionic extension of) the polarisation identity, namely
\begin{equation}\label{eqn:Qpolarization}
    [\langle u,v\rangle_{\mathbb{R}_2}]_0 = \frac{1}{4} \left(\|u+v\|^2 - \|u-v\|^2\right),
\end{equation}
where $\|\cdot\|$ is the norm induced by $\langle \cdot,\cdot \rangle_{\mathbb{R}_2}$. 

For matrices $A,B \in \mathbb{R}_2^{d \times d}$, the \emph{Frobenius inner product} $\langle A,B \rangle_F$ of $A$ and $B$ is defined by $\langle A,B \rangle_F =\Tr(A^{\ast}B) = \sum_{i,k=1}^{d}\overline{A_{ik}}B_{ik}$. The {Frobenius matrix norm} is given by $$\|A\|_F^2 = \sum_{i,k=1}^{d}|A_{ik}|^2 = \sum_{\eps=0}^3\sum_{i,k=1}^{d}\big(A_{ik}^{(\eps)}\big)^2.$$

\subsubsection*{Hilbert module of matrix ensembles}

Let $\vec{U}=(U_j)_{j \in Q_{\eta}^2}, \vec{V}=(V_j)_{j \in Q_{\eta}^2} \in (\mathsf{S}^{4\times 4})^{Q_{\eta}^2}$. We  view $(\mathsf{S}^{4\times 4})_{Q_{\eta}^2}$ as a Hilbert module over $\mathbb{R}$, and equip $(\mathsf{S}^{4\times 4})^{Q_{\eta}^2}$ with the $\mathbb{R}$-valued inner product $\langle \cdot,\cdot \rangle_{\mathbb{R}}$ given by 
\begin{equation}
    \langle \vec{U},\vec{V} \rangle_{\mathbb{R}} = \sum_{\eps=0}^3\sum_{j \in Q_{\eta}^2} \left(\langle  U_j^{(\eps)}, V_j^{(\eps)}\rangle\right).
\end{equation}
The choice of the real scalar field is preferred as we seek to employ optimization algorithms that are well-studied in Hilbert spaces over $\mathbb{R}$.

\subsection{Feasibility problem for quaternionic wavelets}

In order to further simplify our construction, we only work on the Hilbert module $(\mathsf{S}^{4\times 4})_{\sigma\tau}^{Q_{\eta}^2}$ (for $\eta \geq 4$) which we define to be the set of all matrix ensembles that satisfy the consistency conditions, i.e.,
\begin{align}\label{eqn:QHilbert}
    (\mathsf{S}^{4\times 4})_{\sigma\tau}^{Q_{\eta}^2} &:= \bigg\{\vec{B} = (B_j)_{j\in Q_{\eta}^2} \in (\mathsf{S}^{4\times 4})^{Q_{\eta}^2}  :  B_{j + \eta v_{i}/2} =\sigma_{i}B_{j}\ \  (i\in \{1,2\}), \nonumber \\ 
    &\qquad\qquad  \mbox{~and~} B_{\eta v_3-j} = \tau B_j \tau, \mbox{~for all~} j \in Q_{\eta}^2 \bigg\}
\end{align}
where $\sigma_1,\sigma_2$ are defined as in Proposition~\ref{thm:Qconsistencysigma} and $\tau$ is as given in Proposition~\ref{thm:QconsistencySV}. Note that the standard ensemble $\vec{U}$ of samples of $U(\xi)$ belongs to $(\mathsf{S}^{4\times 4})_{\sigma\tau}^{Q_{\eta}^2}$, but the coefficient ensemble $\vec{A}$ lives outside that Hilbert module. We are now ready to state the feasibility problem formulation for quaternionic wavelet construction in the context of this Hilbert module.

\begin{problem}[Quaternionic wavelet feasibility problem]\label{prob:Qwaveletfeasibilityproblem} Let $\eta \geq 4$ be even and $\mathcal{H} := (\mathsf{S}^{4\times 4})_{\sigma\tau}^{Q_{\eta}^2}$. For $1\leq \ell \leq 3$, define the constraint sets
    \vspace{-0.2cm}
    
    \begingroup
    \allowdisplaybreaks
    \resizebox{0.965\textwidth}{!}{%
        \begin{minipage}{1.1\textwidth}
            \begin{align}
                C_1^{(0)} &:= \{\vec{U}\in \mathcal{H} \, : \, U_0 \in [1] \otimes \mathcal{U}^{3\times 3} \mbox{ and } U_j \in \mathcal{U}^{4\times 4}\, \mbox{ for } j\in Q_{\eta/2}^2\backslash \{0\}\},\label{prob:Qwaveletfeasibilityproblem1}\\
                C_1^{(\ell)} &:=\{\vec{U} \in \mathcal{H}\,: \, (\mathcal{F}_{\eta}\chi_{\ell} \mathcal{F}_{\eta}^{-1}\vec{U})_j \in \mathcal{U}^{4\times 4} \mbox{ for } j\in Q_{\eta/2}^2 \},\label{prob:Qwaveletfeasibilityproblem2}\\
                C_2 &:= \{\vec{U} \in \mathcal{H} \, : \, {\textstyle\sum_{j\in Q_{\eta}^2} [c_{\alpha j}]\ostar U_j}\in \mathsf{S}\otimes\mathsf{S}^{3\times 3} \mbox{ for }1\leq |\alpha|\leq \mu\},\label{prob:Qwaveletfeasibilityproblem3}\\	
                C_3 &:= \{\vec{U} \in \mathcal{H} \,:\, (U_j[0,0] + U_{j}[0,1]) = e^{4\pi j\wedge P/\eta}(U_{\eta v_3-j}[0,0] + U_{\eta v_3-j}[0,1]) \forall j \in Q_{\eta}^2\}\label{prob:Qwaveletfeasibilityproblem4}
            \end{align}
    \end{minipage}}
    \endgroup 
    \vspace{0.2cm}
    
    \noindent where $[c_{\alpha j}]=\sum_{k\in Q_{\eta}^2}[k_2^{\alpha_1}k_1^{\alpha_2}][e^{2\pi k\wedge j/\eta}]$. 
    \begin{enumerate}[label=(\alph*)]
        \itemsep0em
        \item To construct a smooth, orthonormal and compactly supported scaling function, find $\vec{U} \in \left(\bigcap_{\ell=0}^{3} C_1^{(\ell)}\right) \cap C_2$.\label{prob:QwaveletfeasibilityproblemA}
        \item To construct a smooth, orthonormal and compactly supported scaling function with point symmetry property, find $\vec{U} \in \left(\bigcap_{\ell=0}^{3} C_1^{(\ell)}\right) \cap C_2 \cap C_3$.\label{prob:QwaveletfeasibilityproblemB}
    \end{enumerate}
\end{problem} 

Note that the requirement of compact support is implicit in the fact that we are working on discrete matrix ensembles. Moreover, the constraint set $C_1^{(0)}$ imposes the completeness condition coupled with unitarity of the original ensemble. For $1\leq \ell \leq 3$, $C_1^{(\ell)}$ captures the unitarity at the other ensembles of samples. $C_2$ represents the vanishing moments condition, and $C_3$ is the point symmetry constraint.

We apply Pierra's product space reformulation to rewrite Problem~\ref{prob:Qwaveletfeasibilityproblem} as a two-set feasibility problem. In particular, Problem~\ref{prob:Qwaveletfeasibilityproblem}\ref{prob:QwaveletfeasibilityproblemA} is equivalent to finding a point on the intersection of the sets 
\begin{align} 
    C &:=C_1^{(0)} \times C_1^{(1)}\times C_1^{(2)}\times C_1^{(3)} \times C_2,\label{eqn:QwaveletprodC}\\
    \text{and~~} D &:= \{(\vec{U}_0, \vec{U}_1,\ldots, \vec{U}_4) \in \mathcal{H}^{5}: \vec{U}_0 = \vec{U}_1= \cdots =\vec{U}_{4}\}. \label{eqn:QwaveletprodD}
\end{align}
We abuse notation by using $C$ and $D$ to denote the appropriate Pierra's product reformulation of Problem~\ref{prob:Qwaveletfeasibilityproblem}\ref{prob:QwaveletfeasibilityproblemB}.

To avoid interrupting the main exposition, the explicit formulas for the projection operators onto the constraint sets $C_1^{(\ell)}$, $1\le \ell \le 3$, $C_2$ and $C_3$ of the quaternionic wavelet feasibility problem are deferred to \ref{app1}.

\subsection{Generating wavelets from feasibility problem solutions}

When a case of the quaternionic wavelet feasibility problem in Problem~\ref{prob:Qwaveletfeasibilityproblem} is solved, we end up having a standard ensemble of samples $$\vec{U}=(U_j)_{Q_{\eta}^2} \in \left(\bigcap_{\ell=0}^{3} C_1^{(\ell)}\right) \cap C_2  \mbox{~~or~~} \vec{U}=(U_j)_{Q_{\eta}^2} \in \left(\bigcap_{\ell=0}^{3} C_1^{(\ell)}\right) \cap C_2 \cap C_3$$ from which the coefficient ensemble is readily computed by $\vec{A} = \mathcal{F}_{\eta}^{-1}\vec{U} = (A_k)_{k\in Q_{\eta}^2}$. At this stage, we discuss how the scaling function and wavelets are generated and plotted.

\subsubsection*{Quaternionic cascade algorithm}

From the coefficient ensemble $\vec{A} =(A_k)_{k\in Q_{\eta}^2}$, we follow an appropriate cascade algorithm to determine the values $\varphi$ and $\psi^{\eps}$ ($\eps \in \{1,2,3\}$) at a dense set of dyadic rationals in the interior of $[0,\eta-1]^2$. 

If $\nu:\{1,2,\ldots, \eta^2\} \to Q_{\eta}^2$ is again a bijection, then the scaling equation in \eqref{eqn:Qscaling} may be rewritten as
\begin{equation}\label{eqn:Qscalingphi}
    \varphi(x) = {4}\sum_{j=1}^{\eta^2}  \varphi(2x-\nu(j)) (A_{\nu(j)}[0,0] + A_{\nu(j)}[0,1]). 
\end{equation}
Sampling \eqref{eqn:Qscalingphi} at points $\nu(n)\in Q_{\eta}^2$ ($n\in \{1,2,\ldots, \eta^2\}$) and letting $v= (\varphi(\nu(n)))_{n=1}^{\eta^2} \in \mathbb{R}_2^{\eta^2}$ gives the matrix equation $v=Sv$ where $S \in \mathbb{R}_2^{\eta^2 \times \eta^2}$ is given by  
$$S_{mn} = \begin{cases} {2^d}A_{\nu(j)}[0,0]& \text{if~~} 2\nu(n)-\nu(m) \in Q_{\eta}^2\\ 0 & \text{otherwise}\end{cases}.$$
Hence, the unknown vector $v$ (containing the lattice samples of $(\varphi(j))_{j\in Q_{\eta}^2}$) can be identified as the eigevector of $S$ with eigenvalue $1$. Once $v$ is known, the half-lattice samples of $\varphi$ are computed using \eqref{eqn:Qscalingphi}. By repeatedly using \eqref{eqn:Qscalingphi}, the values of $\varphi$ at a dense set of dyadic rationals interior to its support $[0,\eta-1]^2$ can be obtained. This procedure is the quaternionic extension of the cascade algorithm. Similarly, by using \eqref{eqn:Qwaveletscaling}, dyadic samples of the wavelet functions are also obtained.

\subsubsection*{Plots of quaternion-valued functions on the plane}

We first recall that a quaternion $q = q_0 + q_1e_1 + q_2e_2 + q_{12}e_{12}$ has a polar form representation 
\eqref{eqn:polarquaternion}. 
Moreover, for any set $X \subseteq \mathbb{R}^2$, let $f: X \to \mathbb{R}_2$ be a quaternion-valued function, i.e., $f(x) = f_0(x) + f_1(x)e_1 + f_2(x)e_{2} + f_{12}(x)e_{12}$ where $f_0,f_1,f_2,f_{12}: X \to \mathbb{R}$. For a fixed $x=(x_1,x_2) \in X$, we write $f(x) = |f(x)|e^{\mu_{f(x)} \phi_{f(x)}}$ in polar form. Since $\mu_{f(x)} \phi_{f(x)}$ is a pure quaternion, we can write it as $$\mu_{f(x)} \phi_{f(x)} = R_{f(x)}e_1 + G_{f(x)}e_2 + B_{f(x)}e_{12}$$ with $R_{f(x)},G_{f(x)},B_{f(x)}$ the corresponding imaginary parts of $\mu_{f(x)} \phi_{f(x)}$. Thus, we may associate $(x,f(x))$ with a point in $\mathbb{R}^3$ with coordinates $(x_1,x_2, |f(x)|)$ and colored by $(R_{f(x)},G_{f(x)},B_{f(x)})$ injected into the RGB (red--green--blue) color space.

Therefore, in plotting our scaling functions and wavelets, we first run the cascade algorithm to obtain sampled values of $\varphi$ and $\psi^{\eps}$ ($\eps\in\{1,2,3\}$) over a sufficiently dense of set of dyadic rationals in the interior of $[0,\eta-1]^2$. These values are then expressed in polar form, and finally associated to a colored point in $\mathbb{R}^3$ for plotting.

\subsection{A check for orthonormality}\label{ssec:orthonormalitycheck}

As noted in Section~\ref{ssec:Qorthogonality}, the unitarity requirement for $U(\xi)$---which was later on imposed on the ensembles of samples of $U(\xi)$---is only a necessary condition for the orthonormality of the lattice shifts of $\varphi$ and $\{\psi_{\eps}\}_{\eps=1}^3$. Even though this unitarity condition ``promotes'' orthonormality, it is not a guarantee. However, we can use Proposition~\ref{prop:Qorthonormalitycheck} to check whether or not the scaling filter $m_0$ obtained using the feasibility approach satisfies the sufficient condition for orthonormality. With Proposition~\ref{prop:Qorthonormalitycheck} and the fact that $m_0$ is a trigonometric polynomial (and hence continuous), it is enough that $[  m_0(\xi)] ^{\ast}[  m_0(\xi)]  > cI_2$ for all $\xi \in [-\frac{1}{4},\frac{1}{4}]^2$ (where $0 <c<1$), or equivalently, that $[m_0(\xi)] ^{\ast}[m_0(\xi)]$ is positive definite for all $\xi \in [-\frac{1}{4},\frac{1}{4}]^2$. Thus, it suffices to check if all eigenvalues of $[m_0(\xi)] ^{\ast}[m_0(\xi)]$ are positive for all $\xi \in [-\frac{1}{4},\frac{1}{4}]^2$. Explicitly, 
\begin{equation*}
        [  m_0(\xi)] ^{\ast}[  m_0(\xi)]  = \begin{bmatrix}
        s(\xi)&v(\xi )\\\overline{v(\xi )}&s(-\xi )
        \end{bmatrix}
\end{equation*}
where $s(\xi):=|{m_0}_s (\xi)|^2 + |{m_0}_v (-\xi)|^2$ and $v(\xi) := \overline{{m_0}_s(\xi)}{m_0}_v(\xi) + \overline{{m_0}_v(-\xi)}{m_0}_s(-\xi)$.
With $\lambda_\pm (\xi)$ the two eigenvalues of $[  m_0(\xi)] ^{\ast}[  m_0(\xi)]$, the eigenvalue formula \eqref{eqn:SVeigenvalue} gives
$$\lambda_\pm(\xi) = \frac{s(\xi) + {s(-\xi)} \pm \sqrt{(s(\xi)-{s(-\xi)})^2 + 4|v(\xi )|^2}}{2}$$
for each $\xi \in\mathbb{R}^2$. Therefore, a scaling function $\varphi$ \emph{passes the orthonormality check} whenever 
\begin{equation}\label{eqn:Qorthonormalitycheck}
    \Lambda(\xi) := \min (\lambda_+(\xi),\lambda_-(\xi)) > 0 \mbox{~~for all~~} \xi \in [{\textstyle -\frac{1}{4},\frac{1}{4}}]^2.
\end{equation}

\subsection{A check for nonseparability}\label{ssec:nonseparabilitycheck}

The feasibility problem formulation for quaternionic wavelets is ``agnostic'' toward the concept of separability, i.e., it may or may not produce separable scaling functions and wavelets.  Hence, it is imperative that we first come up with an appropriate measure of separability in the quaternionic setting.

Suppose $\varphi^1,\varphi^2$ are quaternion-valued scaling functions on the line and $\varphi$ is a quaternion-valued scaling function on the plane. If $\varphi(x_1,x_2) = \varphi^1(x_1)\varphi^2(x_2)$ then
$$\int_{\mathbb{R}} \varphi(x_1,x_2) dx_2 = \varphi^1(x_1) \text{~and~} \int_{\mathbb{R}} \varphi(x_1,x_2) dx_1 = \varphi^2(x_2) $$
and so 
$$\varphi(x_1,x_2) = \left(\int_{\mathbb{R}} \varphi(x_1,s) ds\right) \left(\int_{\mathbb{R}} \varphi(t,x_2) dt \right).$$
Using this observation, a measure of separability is defined in the next proposition. We omit the proof for brevity.

\begin{proposition}\label{prop:Qnonsepmeasure}
    Let $\varphi$ be a quaternion-valued scaling function on the plane and define 
    \begin{equation*}
    \tilde{\zeta}(\varphi):= \int_{\mathbb{R}}\int_{\mathbb{R}} \left|\varphi(x_1,x_2) - \left(\int_{\mathbb{R}} \varphi(x_1,s) ds\right) \left(\int_{\mathbb{R}} \varphi(t,x_2) dt \right)  \right|^2dx_1dx_2
    \end{equation*}
    as a \emph{measure of separability for $\varphi$}. Then $\varphi$ is separable if and only if $\tilde{\zeta}(\varphi)=0$.
\end{proposition}

In the quaternionic wavelet feasibility problem, we do not have a constraint set to promote nonseparability. However, we include a separability check to determine whether or not we have obtained a separable scaling function. We conclude nonseparability whenever a scaling function $\varphi$ satisfies $\tilde{\zeta}(\varphi) >0$.

\section{Numerical results}
\label{sec:Qnumericalsolutions}

In this section, we attempt to numerically solve the product space reformulations of the standard quaternionic wavelet feasibility problems described in Problem~\ref{prob:Qwaveletfeasibilityproblem}. Relevant source codes, solutions and wavelet filters are available at \url{https://gitlab.com/nddizon1/waveletconstruction}, where we have also used the Quaternion Toolbox for MATLAB\textsuperscript{\textregistered} \cite{qtfm}.

For our purpose, we denote the product spaces in  Pierra's two-set reformulation of a many-set feasibility problem by $C$ and $D$ as in \eqref{eqn:QwaveletprodC} and \eqref{eqn:QwaveletprodD}, respectively. We again confine our choice of projection algorithm to the Douglas--Rachford (DR) as described in Algorithm~\ref{alg:DR}. We let $(x_n)_{n\in \mathbb{N}}$ be the sequence of iterates generated by product DR. In all our numerical implementations, we consider a tolerance $\epsilon:=10^{-9}$ and adopt a stopping criterion given by $\|P_DP_C(x_n) - P_C(x_n)\| < \epsilon$. We consider a projection algorithm to have \emph{solved} our feasibility problem if it satisfies the stopping criterion within the cutoff of $10,000$ iterates when $\eta=4$, or $300,000$ iterates when $\eta=6$. For some interesting runs, we also plot \emph{errors} given by $\|x_n-x_{n-1}\|$ and $\|P_C(x_n) - P_C(x_{n-1})\|$ for the DR iterates and shadow sequences, respectively.

\begin{table}[hbt!]
    \centering
    \resizebox{\textwidth}{!}{%
        \begin{tabular}{@{}ccccccccc@{}}
            \toprule
            & \multirow{2}{*}{\textbf{Parameters}} & \multirow{2}{*}{\textbf{\begin{tabular}[c]{@{}c@{}}cases\\ solved\end{tabular}}} & \multicolumn{6}{c}{\textbf{number of iterations (when solved)}} \\ \cmidrule(l){4-9} 
            &                 &        & \textbf{min} & \textbf{Q1} & \textbf{median} & \textbf{Q3} & \textbf{mean} & \textbf{max} \\ \midrule
            Problem~\ref{prob:Qwaveletfeasibilityproblem}\ref{prob:QwaveletfeasibilityproblemA} & $\eta=4, \mu=1$ & 72/100 & 2411         & 3929        & 4891            & 6525        & 5306          & 9884         \\
            Problem~\ref{prob:Qwaveletfeasibilityproblem}\ref{prob:QwaveletfeasibilityproblemA} & $\eta=6, \mu=2$ & 1/50   & 22172        & --       & 22172           & --       & 22712         & 22172        \\
            Problem~\ref{prob:Qwaveletfeasibilityproblem}\ref{prob:QwaveletfeasibilityproblemB} & $\eta=6, \mu=2$ & 35/50 & 60255        & 81355       & 97405           & 137698      & 110530        & 299497       \\ \bottomrule
        \end{tabular}%
    }\caption{Statistics on the performance of product Douglas--Rachford in solving different cases of Problem~\ref{prob:Qwaveletfeasibilityproblem}.}
    \label{tab:Qwaveletfeasiblityproblem}
\end{table}

Table~\ref{tab:Qwaveletfeasiblityproblem} gives a summary of the specific problem parameters chosen for each case of Problem~\ref{prob:Qwaveletfeasibilityproblem}. It contains the number of times product DR solved a particular problem using randomized starting points. Furthermore, the table provides statistics on the number of iterations incurred in solving the problem.

For Problem~\ref{prob:Qwaveletfeasibilityproblem}\ref{prob:QwaveletfeasibilityproblemA} with $\eta=4$ and $\mu=1$, we see from Table~\ref{tab:Qwaveletfeasiblityproblem} that product DR solved  $72\%$ of all the test cases. An example of a wavelet ensemble derived from a solution to this feasibility problem is given in Figure~\ref{fig:QM4D1_0}.

Similarly, product DR solved only $2\%$ of all initialisations made for Problem~\ref{prob:Qwaveletfeasibilityproblem}\ref{prob:QwaveletfeasibilityproblemA} with $\eta=6$ and $\mu=2$. This means that within the cutoff of $300,000$ iterates, $98\%$ of the initialisations were not able to reach the stopping criterion. This, however, does not imply that DR is not adept in solving this problem. Analysing these runs revealed that most starting points belonging to the latter group exhibited decreasing errors but were not enough to meet the threshold. This only suggests that more iterations should be allowed to meet the stopping criterion. An example of a wavelet ensemble derived from a solution to this feasibility problem is given in Figure~\ref{fig:QM6D2}. Another wavelet ensemble solution (after allowing more iterations) is given in Figure~\ref{fig:QM6D2_IMP}. Observe how the scaling function exhibits (near) symmetry even though we have not yet added the symmetry constraint in this problem.

For Problem~\ref{prob:Qwaveletfeasibilityproblem}\ref{prob:QwaveletfeasibilityproblemB}, we first note that it takes into account the symmetry constraint set. This means its feasible region is smaller than that of Problem~\ref{prob:Qwaveletfeasibilityproblem}\ref{prob:QwaveletfeasibilityproblemA}. For Problem~\ref{prob:Qwaveletfeasibilityproblem}\ref{prob:QwaveletfeasibilityproblemB} with $\eta=6$ and $\mu=2$, product DR solved $70\%$ of all initialisations. An example of wavelet ensemble derived from solutions of this particular feasibility problem is given in Figure~\ref{fig:QM6D2_S1}.

\begin{figure}[htbp]
    \centering
    \vspace{-1cm}
    
    \subfloat[Scaling function $\varphi$]{\includegraphics[width=0.23\linewidth]{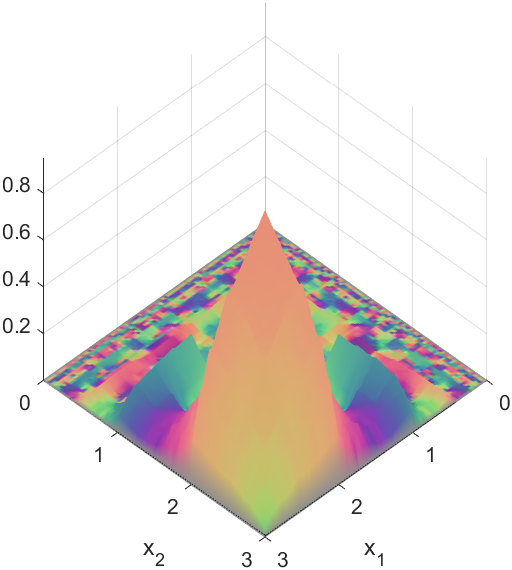}\label{fig:QM4D1_0_phi}}\hfill
    \subfloat[Wavelet $\psi_1$]{\includegraphics[width=0.23\linewidth]{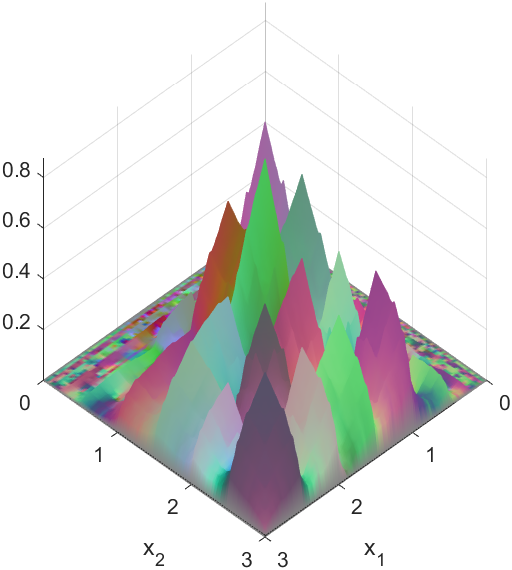}\label{fig:QM4D1_0_psi1}}\hfill
    \subfloat[Wavelet $\psi_2$]{\includegraphics[width=0.23\linewidth]{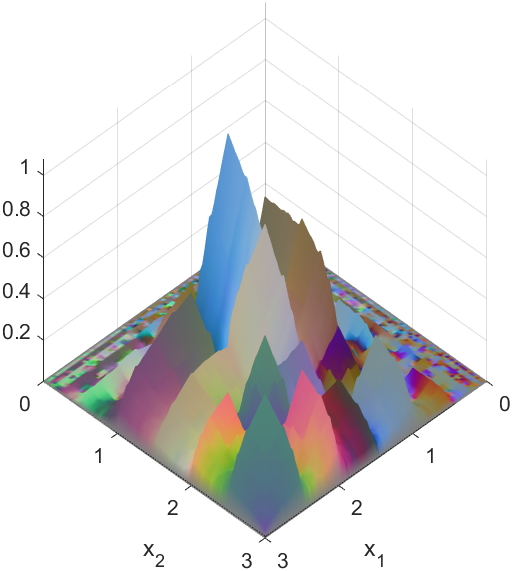}\label{fig:QM4D1_0_psi2}}	\hfill
    \subfloat[Wavelet $\psi_3$]{\includegraphics[width=0.23\linewidth]{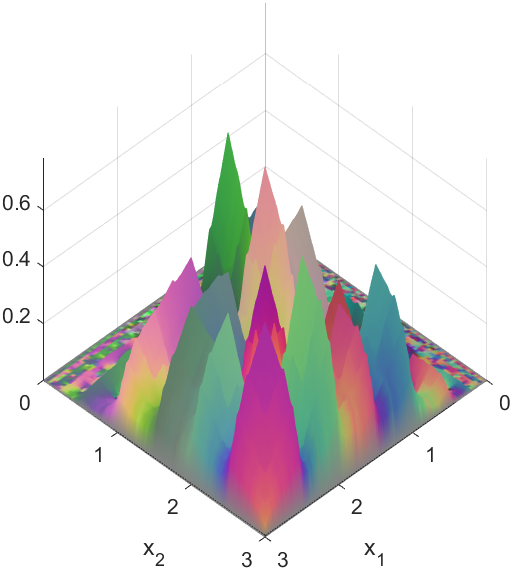}\label{fig:QM4D1_0_psi3}}\hfill
    \caption{Wavelet ensemble derived from  Problem~\ref{prob:Qwaveletfeasibilityproblem}\ref{prob:QwaveletfeasibilityproblemA} where $\eta=4$ and $\mu=1$, with  $\tilde{\zeta}(\varphi)\approx 2.3243$.}\label{fig:QM4D1_0}
    \vfill
    \subfloat[Scaling function $\varphi$]{\includegraphics[width=0.23\linewidth]{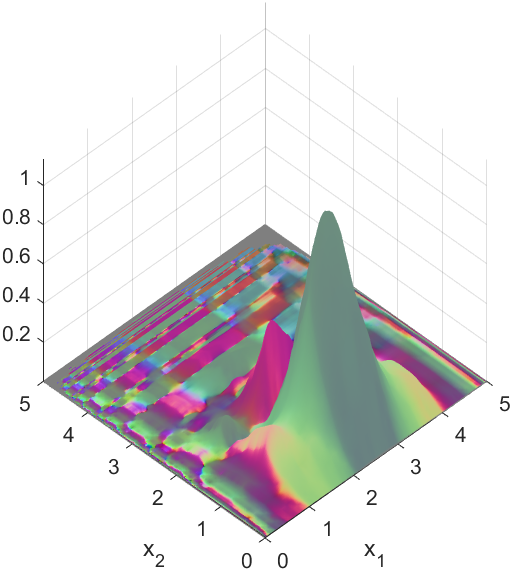}\label{fig:QM6D2_phi}}\hfill
    \subfloat[Wavelet $\psi_1$]{\includegraphics[width=0.23\linewidth]{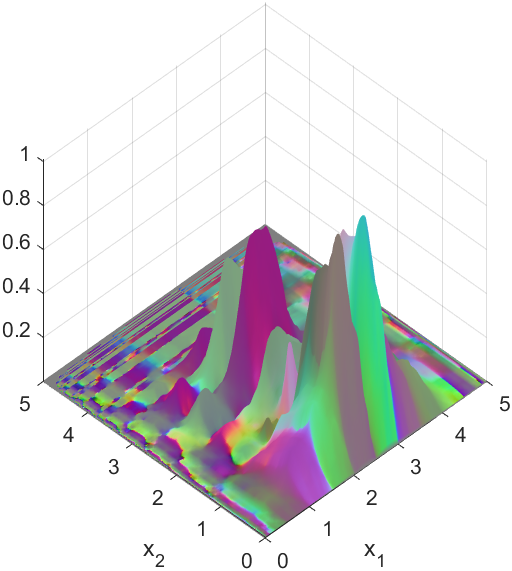}\label{fig:QM6D2_psi1}}\hfill
    \subfloat[Wavelet $\psi_2$]{\includegraphics[width=0.23\linewidth]{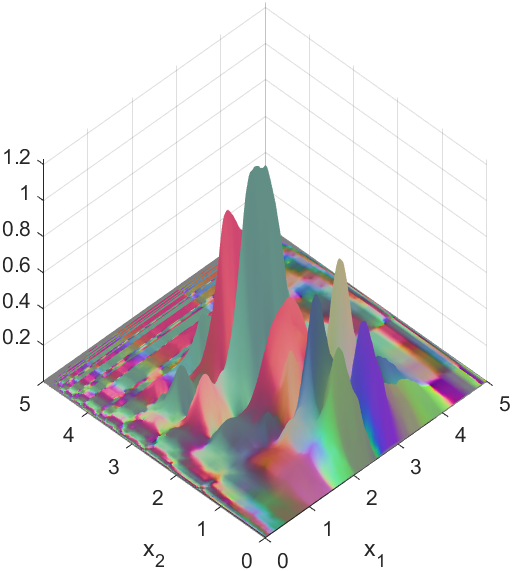}\label{fig:QM6D2_psi2}}	\hfill
    \subfloat[Wavelet $\psi_3$]{\includegraphics[width=0.23\linewidth]{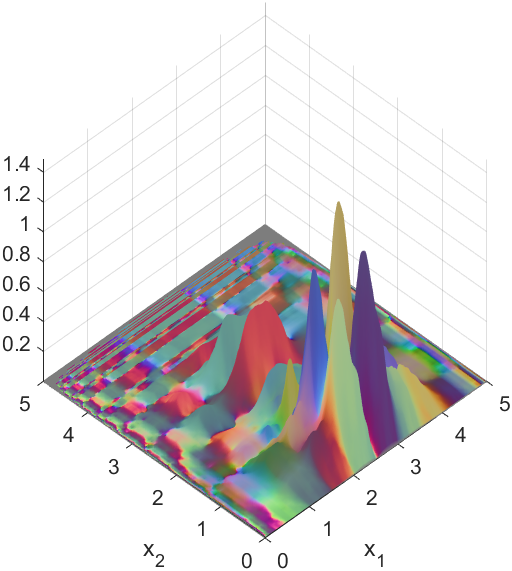}\label{fig:QM6D2_psi3}}\hfill
    \caption{Wavelet ensemble derived from Problem~\ref{prob:Qwaveletfeasibilityproblem}\ref{prob:QwaveletfeasibilityproblemA} where $\eta=6$ and $\mu=2$, with $\tilde{\zeta}(\varphi)\approx 0.5922$.}\label{fig:QM6D2}
    \vfill
    \subfloat[Scaling function $\varphi$]{\includegraphics[width=0.23\linewidth]{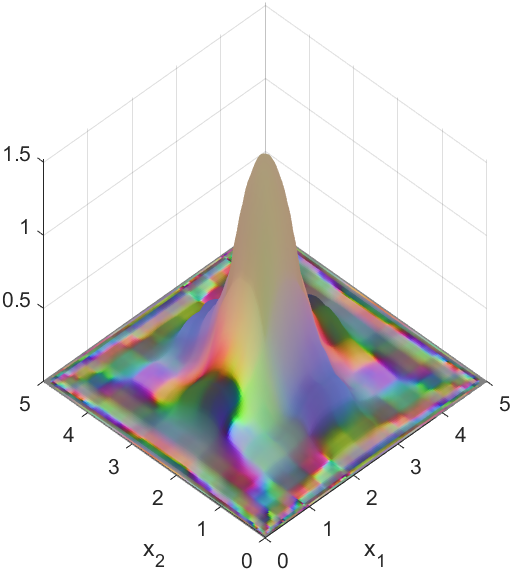}\label{fig:QM6D2_IMP_phi}}\hfill
    \subfloat[Wavelet $\psi_1$]{\includegraphics[width=0.23\linewidth]{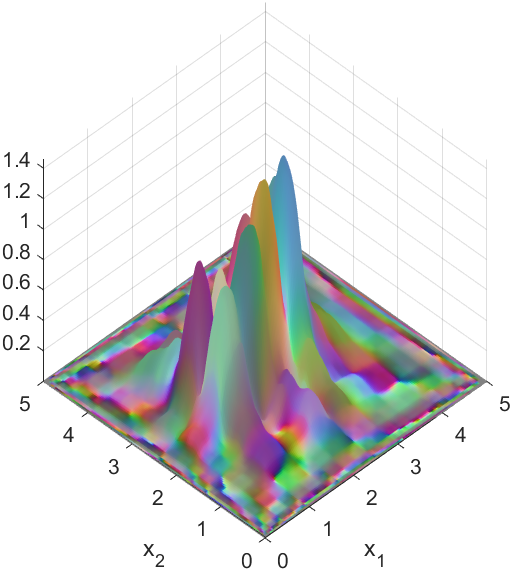}\label{fig:QM6D2_IMP_psi1}}\hfill
    \subfloat[Wavelet $\psi_2$]{\includegraphics[width=0.23\linewidth]{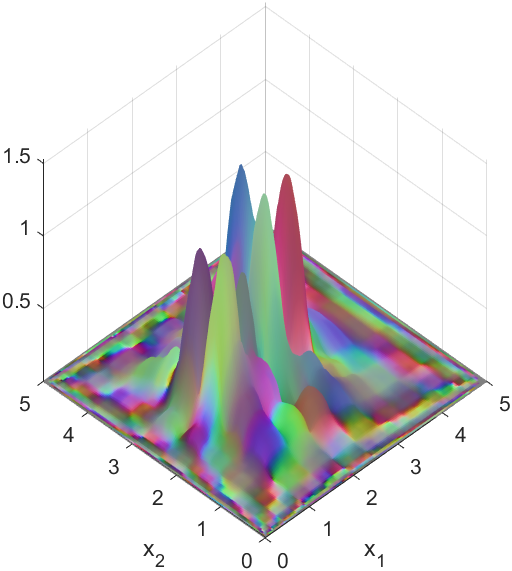}\label{fig:QM6D2_IMP_psi2}}	\hfill
    \subfloat[Wavelet $\psi_3$]{\includegraphics[width=0.23\linewidth]{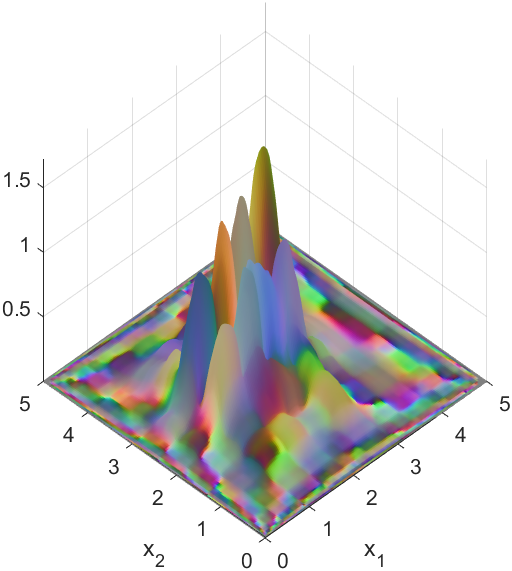}\label{fig:QM6D2_IMP_psi3}}\hfill
    \caption{Wavelet ensemble derived from Problem~\ref{prob:Qwaveletfeasibilityproblem}\ref{prob:QwaveletfeasibilityproblemA} where $\eta=6$ and $\mu=2$, with  $\tilde{\zeta}(\varphi)\approx 1.6764$.}\label{fig:QM6D2_IMP}
    \vfill
    \subfloat[Scaling function $\varphi$]{\includegraphics[width=0.23\linewidth]{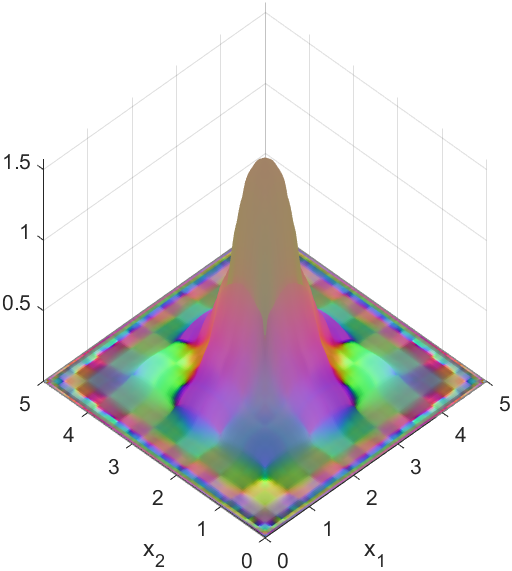}\label{fig:QM6D2_S1_phi}}\hfill
    \subfloat[Wavelet $\psi_1$]{\includegraphics[width=0.23\linewidth]{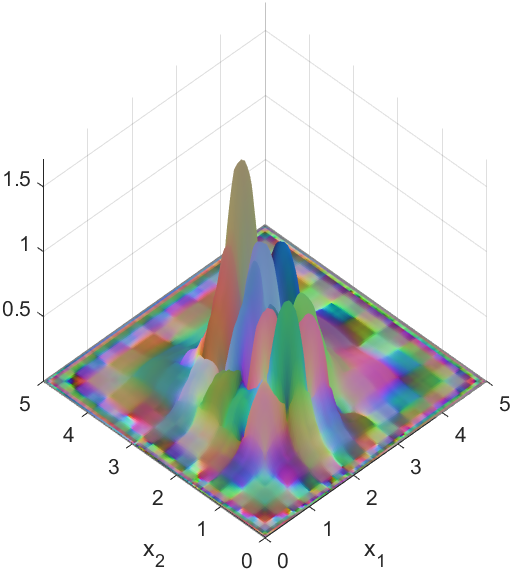}\label{fig:QM6D2_S1_psi1}}\hfill
    \subfloat[Wavelet $\psi_2$]{\includegraphics[width=0.23\linewidth]{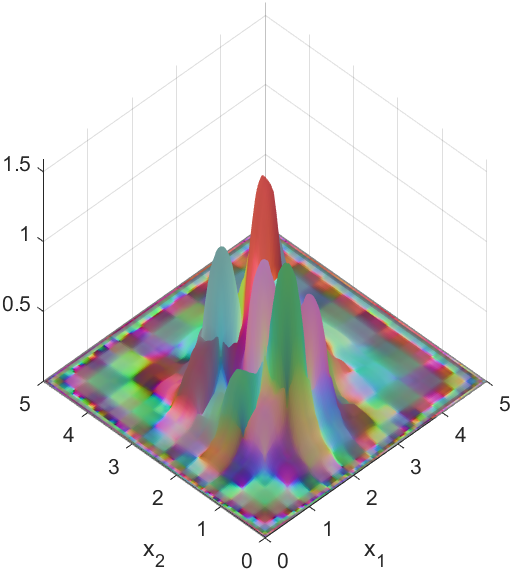}\label{fig:QM6D2_S1_psi2}}	\hfill
    \subfloat[Wavelet $\psi_3$]{\includegraphics[width=0.23\linewidth]{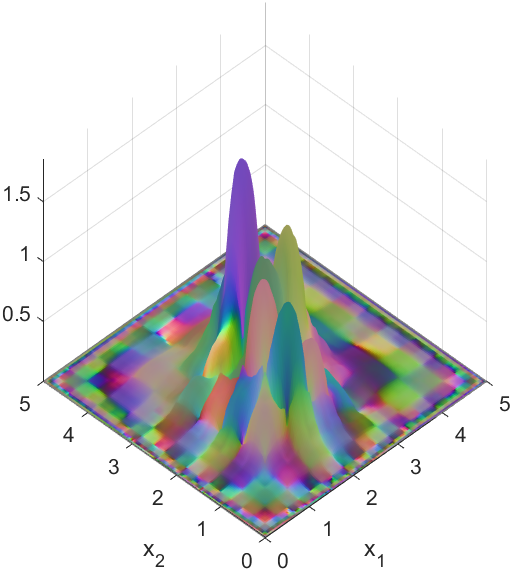}\label{fig:QM6D2_S1_psi3}}\hfill
    \caption{Wavelet ensemble derived from Problem~\ref{prob:Qwaveletfeasibilityproblem}\ref{prob:QwaveletfeasibilityproblemB} where $\eta=6$ and $\mu=2$, with  $\tilde{\zeta}(\varphi)\approx 1.4557$.}\label{fig:QM6D2_S1}
\end{figure}

The scaling functions and wavelets in Figures~\ref{fig:QM4D1_0}--\ref{fig:QM6D2_S1} passed the orthonormality check, as illustrated in Figure~\ref{fig:Qorthonormalitycheck} showing the positivity of $\Lambda(\xi)$ (defined in \eqref{eqn:Qorthonormalitycheck}) for each filter $m_0$ associated to the scaling functions in Figures~\ref{fig:QM4D1_0}--\ref{fig:QM6D2_S1}. We also note that each scaling function $\varphi$ that appeared in Figures~\ref{fig:QM4D1_0}--\ref{fig:QM6D2_S1} passed the nonseparability check in Proposition~\ref{prop:Qnonsepmeasure}, i.e.,  ${\tilde{\zeta}}(\varphi)> 0$. The convergence heuristics of DR for selected cases of Problem~\ref{prob:Qwaveletfeasibilityproblem} are also presented in Figure~\ref{fig:Uchange_Q}.

\begin{figure}[htbp]
    \subfloat[For $\varphi$ in Figure~\ref{fig:QM4D1_0}]{\includegraphics[width=0.23\linewidth]{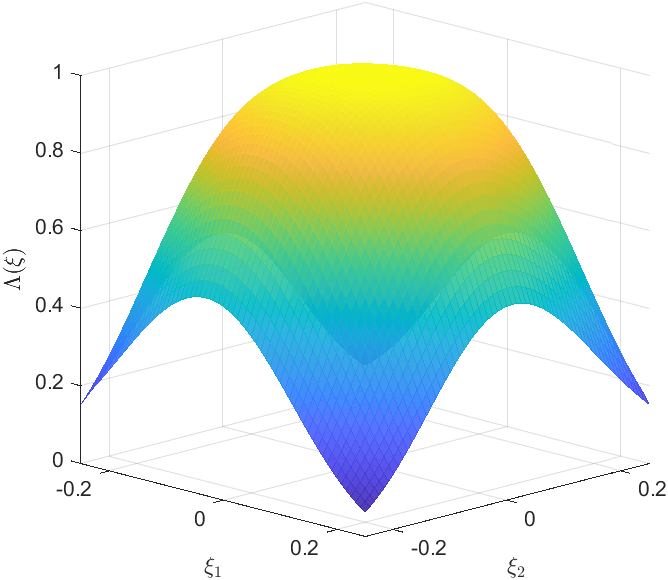}\label{fig:QM4D1_0_Lambda}}\hfill
    \subfloat[For $\varphi$ in Figure~\ref{fig:QM6D2}]{\includegraphics[width=0.23\linewidth]{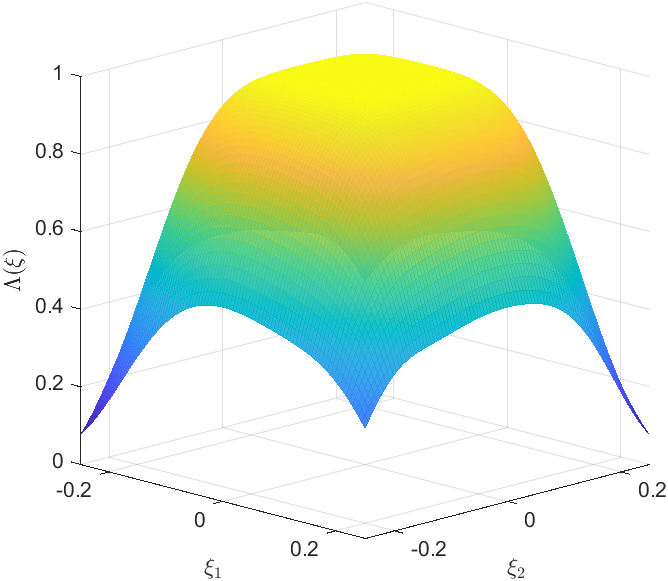}\label{fig:QM6D2_Lambda}}\hfill
    \subfloat[For $\varphi$ in Figure~\ref{fig:QM6D2_IMP}]{\includegraphics[width=0.23\linewidth]{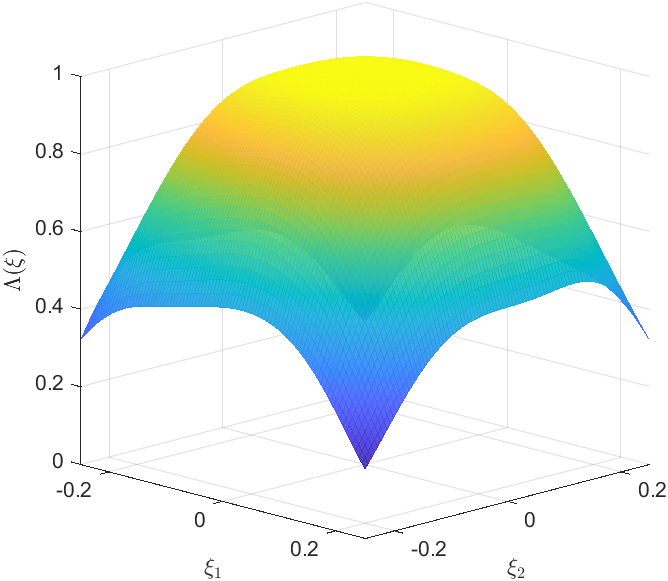}\label{fig:QM6D2_IMP_Lambda}}\hfill
    \subfloat[For $\varphi$ in Figure~\ref{fig:QM6D2_S1}]{\includegraphics[width=0.23\linewidth]{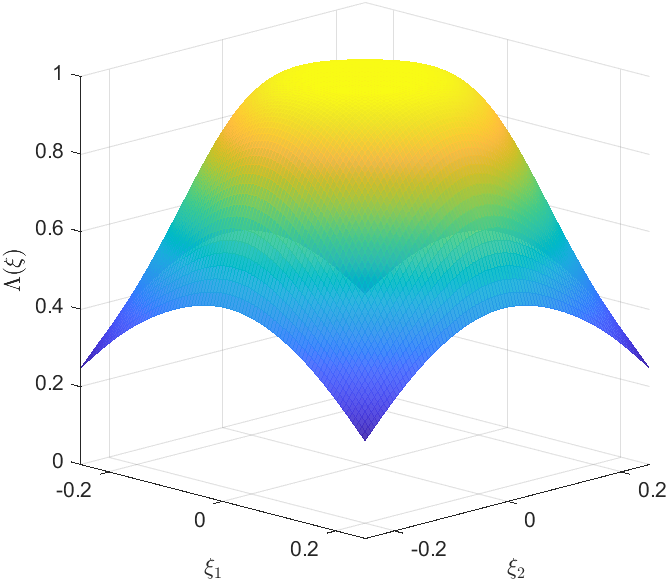}\label{fig:QM6D2_S1_Lambda}}	\hfill
    \caption{Orthonormality checks for wavelet ensemble solutions.}\label{fig:Qorthonormalitycheck}
    \vfill
    \subfloat[For Figure~\ref{fig:QM4D1_0}]{\includegraphics[width=0.23\linewidth]{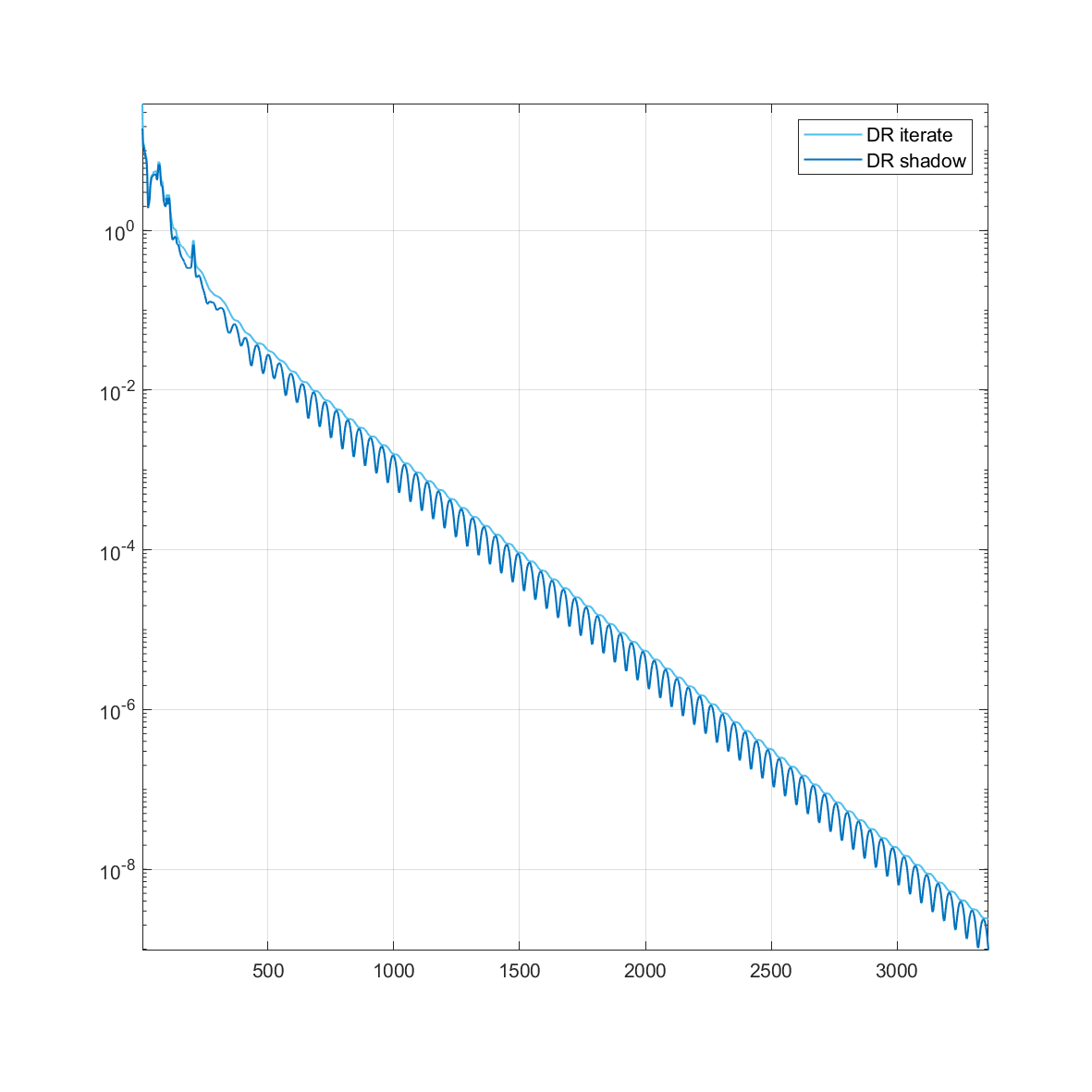}\label{fig:Change_QM4D1_0}}\hfill
    \subfloat[For Figure~\ref{fig:QM6D2}]{\includegraphics[width=0.23\linewidth]{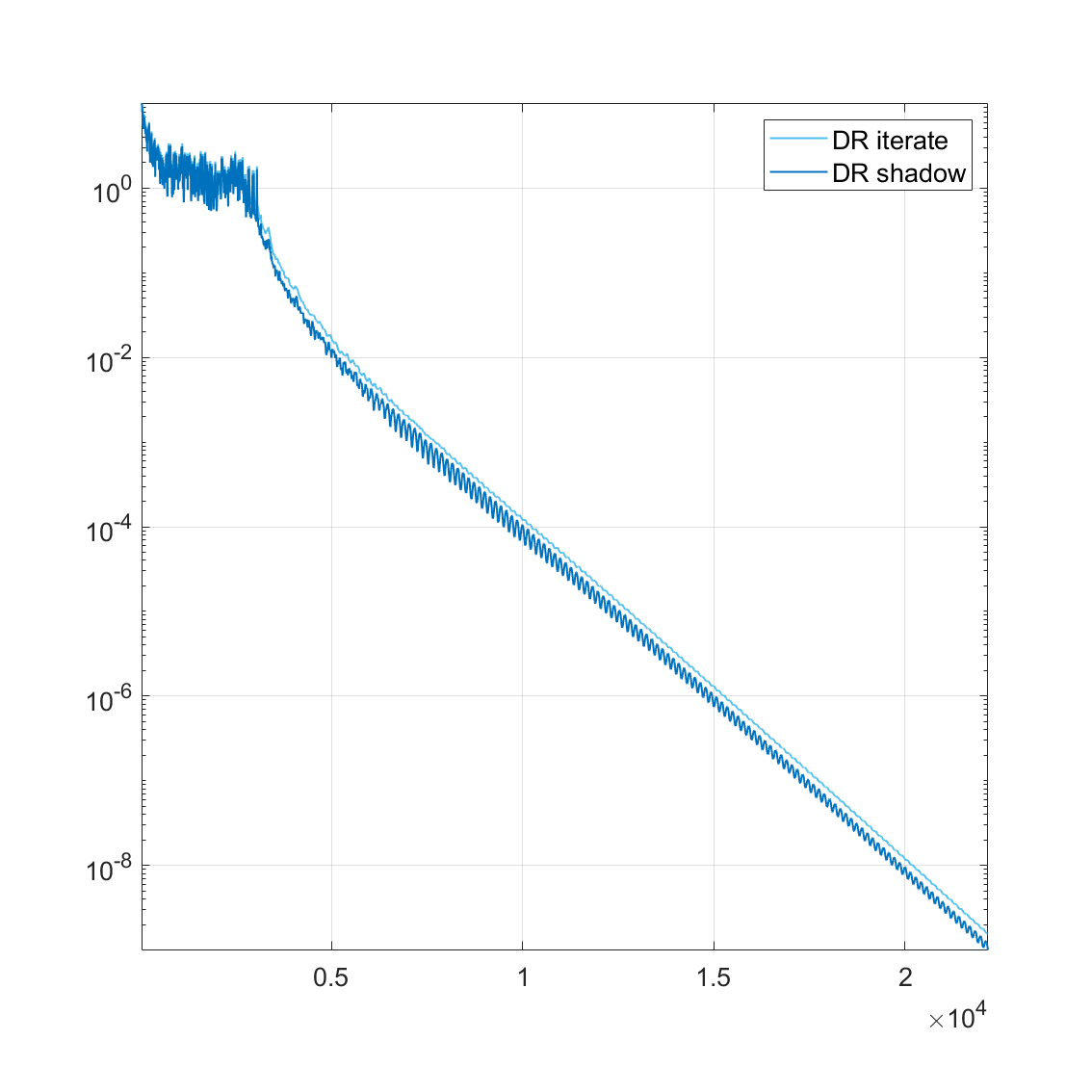}\label{fig:Change_QM6D2}}\hfill
    \subfloat[For Figure~\ref{fig:QM6D2_IMP}]{\includegraphics[width=0.23\linewidth]{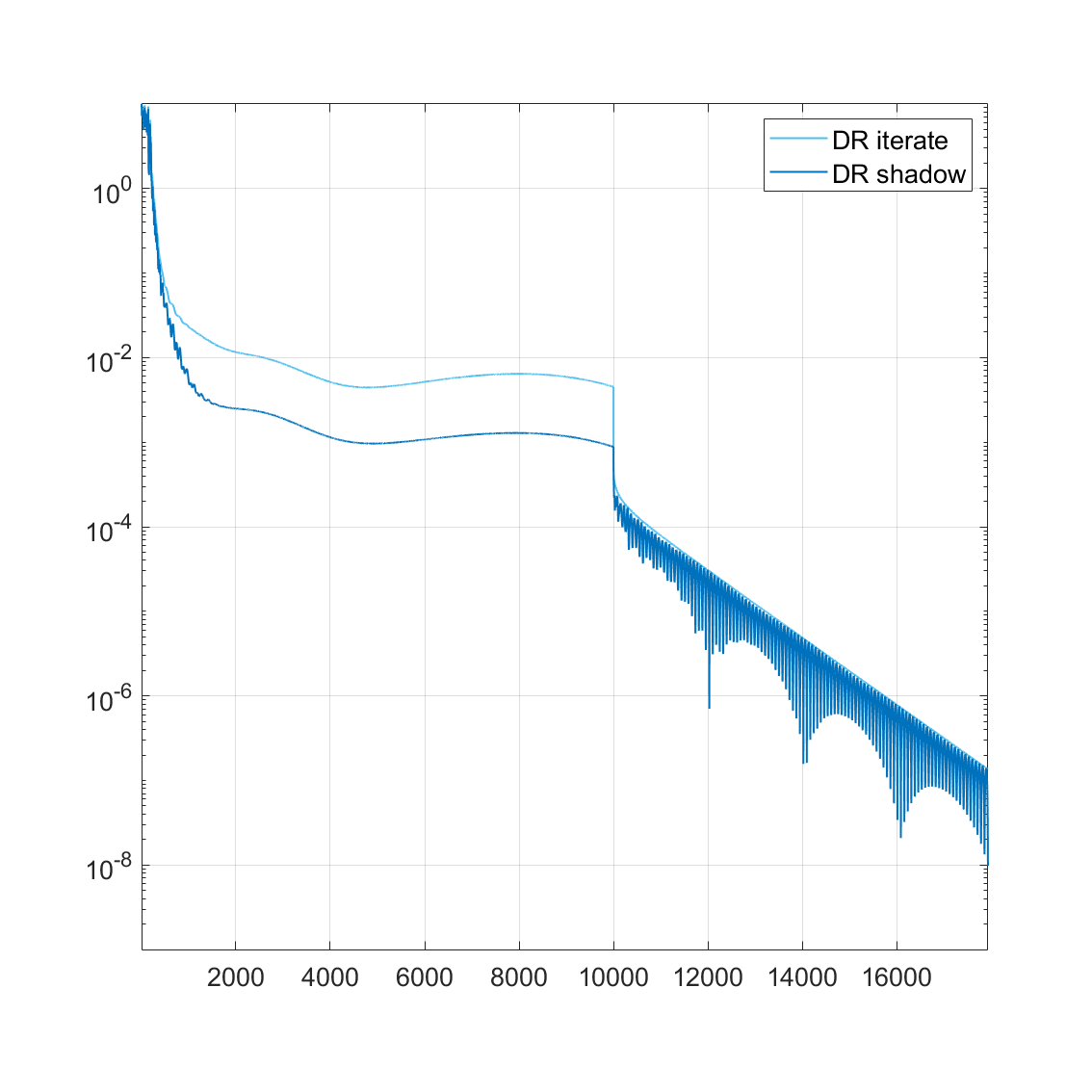}\label{fig:Change_QM4D1}}\hfill
    \subfloat[For Figure~\ref{fig:QM6D2_S1}]{\includegraphics[width=0.23\linewidth]{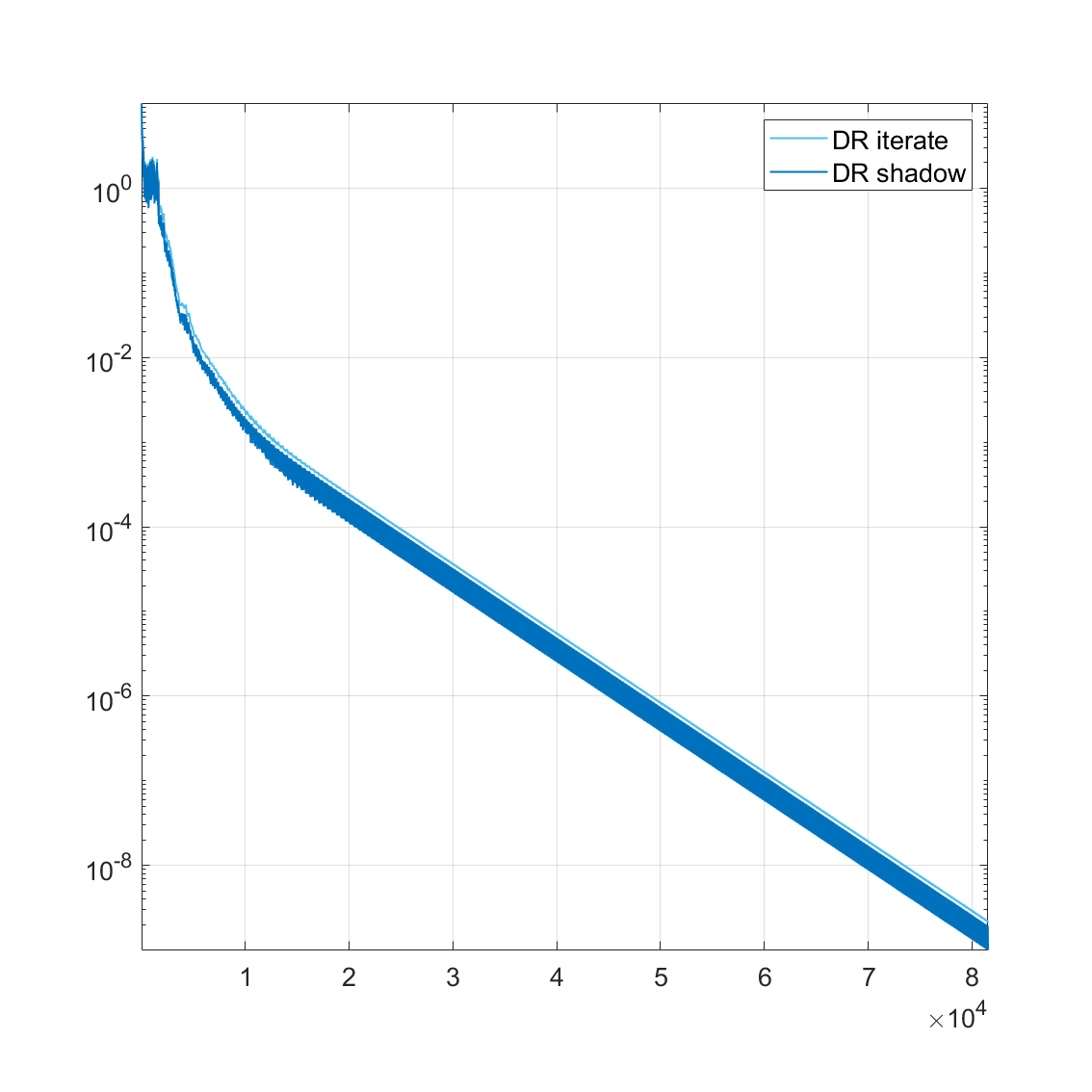}\label{fig:Change_QM6D2_S1}}	\hfill
    \caption{Convergence heuristics of DR for selected cases of Problem~\ref{prob:Qwaveletfeasibilityproblem}.}\label{fig:Uchange_Q}
\end{figure}

\section{Conclusion}
The feasibility approach to wavelet construction served as an alternative method for constructing classical higher-dimensional wavelets. It presented a path towards the construction of quaternionic wavelets on the plane. The successful architecture of nonseparable, multiresolution, compactly supported, smooth and orthonormal quaternion-valued wavelets on the plane leaves open many important avenues of research. With these wavelets, the pixel components of a color image may now be encoded into quaternions for holistic processing of signals using wavelet transforms. With such an approach, the potentially useful correlations between the pixel components are not lost. Using the quaternion-valued wavelets on the plane developed in this paper, a method for decomposing and reconstructing color images using quaternion-valued wavelets has been proposed and implemented for color image processing, as detailed in \cite{dizon2024holistic}. Finally, we note that cardinality and other types of symmetries may also be added in the design criteria for building quaternionic wavelets.

\appendix

\section{Projection operators}\label{app1}
\renewcommand{\thetheorem}{\Alph{section}.\arabic{theorem}}
\setcounter{theorem}{0} 

In this section, we describe the appropriate projectors onto 
each of the constraint sets in the quaternionic wavelet feasibility problem.

We begin with a proposition that summarizes several standard projection results in finite-dimensional Hilbert spaces and in the set of spinor-vector block matrices. These results will be repeatedly invoked in the subsequent derivation of projectors onto the constraint sets in the quaternionic wavelet feasibility problem.

\begin{proposition}\label{prop:projectors} Let $\mathcal{H}_1, \mathcal{H}_2$ be finite-dimensional Hilbert spaces, $\mathsf{S}^{n\times n}$ be the collection of $n$-by-$n$ spinor-vector block matrices, and $\mathcal{U}^{n\times n}$ be the set of unitary $n\times n$ spinor-vector block matrices. 
    \begin{enumerate}[label=(\alph*)]
        \itemsep0em
        \item If $T: \mathcal{H}_1 \to \mathcal{H}_2$ is linear, $C=\Null T := \{x \in \mathcal{H}_1 \, : \, Tx=0\}$ and $TT^{\ast}$ is invertible, then $P_C(x) = x-T^{\ast}(TT^{\ast})^{-1}T (x) \text{ for all } x\in\mathcal{H}_1.$ \label{prop:projectors1}
        \item If $0\neq a \in \mathcal{H}_1$, $\beta \in \mathbb{R}$ and $C=\{z\in \mathcal{H}_1\, : \, \langle a,z\rangle = \beta\}$, then for all $x \in \mathcal{H}_1$  $$P_{C}(x) = x -\frac{\langle a,x \rangle -\beta}{\|a\|^2}a.$$\label{prop:projectors2}
        \item If $C=\mathcal{B}(z,\delta) \subset \mathcal{H}_2$, then $$P_C(x) = \begin{cases}x, & \text{if~} \|x-z\| \leq \delta\\ z + \frac{\delta}{\|x-z\|}(x-z) & \text{if~} \|x-z\|> \delta.\end{cases}$$	\label{prop:projectors3}
        \item If $A \in \mathsf{S}^{n \times n}$ then
        $P_{\mathcal{U}^{n\times n}}(A) = \{UV^{\ast} \, : \, A = U\Sigma V^{\ast} \text{ is a singular value decomposition of $A$}\}.$\label{prop:Qprojectors1}
        \item Let $a \in \mathsf{S}$, $b\in \mathsf{S}^{1\times(n-1)}$, $c \in \mathsf{S}^{(n-1)\times 1}$, $D \in \mathsf{S}^{(n-1)\times(n-1)}$ and $A=\begin{bmatrix} a & b\\ 
            c & D\end{bmatrix}.$ Then $$P_{[1]\otimes \mathcal{U}^{(n-1)\times (n-1)}}(A)  = \left\{\begin{bmatrix}[1] & d^\top \\ d & \tilde{D}\end{bmatrix} : d=0\in\mathsf{S}^{(n-1)\times 1}, \tilde{D}\in P_{\mathcal{U}^{(n-1)\times (n-1)}}(D)\right\}.$$\label{prop:Qprojectors2}
        \item Let $Q$ be a countably finite index set, $(\mathsf{S}^{n\times n})^Q :=\{\vec{B}=(B_j)_{j\in Q}\, : \; B_j \in \mathsf{S}^{n\times n}\ \forall j \in Q\}$ and $(\mathcal{U}^{n\times n})^Q:=\{\vec{B}=(B_j)_{j \in Q} \, : \, B_j \in \mathcal{U}^{n\times n} \ \forall j \in Q\}$. For any $\vec{V} \in (\mathsf{S}^{n\times n})^Q$, $$\vec{W}=(W_j)_{j\in Q} \in P_{({\mathcal{U}^{n\times n}})^Q} (\vec{V}) \iff W_j \in P_{\mathcal{U}^{n\times n}}(V_j) \text{~~for all ~} j \in Q.$$\label{prop:Qprojectors3}
        \item Let $0 \neq \vec{a} \in \mathbb{C}^d$, $b\in \mathbb{C}$, $\epsilon\geq0$ and $C:=\{\vec{z} \in \mathbb{C}^d \,:\, |\langle \vec{a},\vec{z} \rangle + b | \leq \epsilon\}$. If $\beta(\vec{z}) = \frac{ \epsilon (\langle \vec{a},\vec{z} \rangle + b)}{|\langle \vec{a},\vec{z} \rangle + b|}$, then $P_C(z) =\vec{p}$ where
        $$\vec{p}= \begin{cases}\vec{z} & \text{if~} |\langle \vec{a},\vec{z} \rangle + b| \leq \epsilon\\ \vec{z} + \frac{\beta(\vec{z})-\langle \vec{a},\vec{z} \rangle-b}{\|\vec{a}\|^2}\vec{a} & \text{if~} |\langle \vec{a},\vec{z} \rangle + b| > \epsilon.\end{cases}$$\label{prop:projectors_symm}
    \end{enumerate}	
\end{proposition}
\begin{proof}
    For \ref{prop:projectors1}, see \cite[Example~29.17]{bcombettes}. For \ref{prop:projectors2}--\ref{prop:projectors3}, see \cite[Section~4.1]{cegielski}. For \ref{prop:Qprojectors1}, see \cite[Corollary~6.30]{dizonthesis}. Statements \ref{prop:Qprojectors2} and \ref{prop:Qprojectors3} directly follow from \ref{prop:Qprojectors1}. Lastly, for \ref{prop:projectors_symm}, see \cite[Proposition~1.8]{dizonthesis}.
\end{proof}

We now look at the projectors onto the constraint sets $C_1$, $C_2$ and $C_3$ of the quaternionic wavelet feasiblity problem defined in Problem~\ref{prob:Qwaveletfeasibilityproblem}.
It is important to note that this problem is posed on the Hilbert module $(\mathsf{S}^{4\times 4})_{\sigma\tau}^{Q_{\eta}^2}$ which contains ensembles that satisfy all of the consistency conditions. And so, when employing projection algorithms, the initial ensemble must lie in $(\mathsf{S}^{4\times 4})_{\sigma\tau}^{Q_{\eta}^2}$ as defined in \eqref{eqn:QHilbert}.

\subsubsection*{Projector onto $C_1$}

Recall that the constraint sets $C_1^{(\ell)}$ with $0 \leq \ell \leq 3$ together impose the unitary conditions at $4\eta^2$ samples of the wavelet matrix $U(\xi)$. In order to describe the projectors onto these constraint sets, we appeal to Proposition~\ref{prop:projectors}\ref{prop:Qprojectors1}, \ref{prop:projectors}\ref{prop:Qprojectors2} and \ref{prop:projectors}\ref{prop:Qprojectors3}.

\begin{proposition}\label{thm:Qprojunitary}
    Let $C_1^{(\ell)}$ with $0 \leq \ell \leq 3$ be defined as in Problem~\ref{prob:Qwaveletfeasibilityproblem}, and $\vec{U}=(U_j)_{j\in Q_{\eta}^2} \in (\mathsf{S}^{4\times 4})^{Q_{\eta}^2}$. Then the following hold. 
    \begin{enumerate}[label=(\alph*)]
        \itemsep0em
        \item For $C_{1}^{(0)}$, $\vec{V}=(V_j)_{j\in Q_{\eta}^2} \in P_{C_1^{(0)}}(\vec{U})$ if and only if  $$V_j  \in 
        \left\{\begin{array}{ll}
            \sigma_{k} P_{[1] \otimes \mathcal{U}^{3\times 3}}(U_0) & j=\frac{\eta v_{k}}{2}\\
            P_{\mathcal{U}^{n\times n}}(U_j) & \text{otherwise}									
        \end{array}\right.$$
        where $1\leq k\leq 3$.\label{thm:Qprojunitary1}	
        \item  For $C_{1}^{(\ell)}$ with $\ell \in \{1,2,3\}$,  $P_{C_1^{(\ell)}} (\vec{U}) = \mathcal{F}_{\eta}\chi_{\ell}^{-1} \mathcal{F}_{\eta}^{-1}P_{(\mathcal{U}^{4\times 4})^{Q_{\eta}^2}}\mathcal{F}_{\eta}\chi_{\ell} \mathcal{F}_{\eta}^{-1}\vec{U}.$ \label{thm:Qprojunitary2}
    \end{enumerate}
\end{proposition}
\begin{proof}
    Statement~\ref{thm:Qprojunitary1} follows from a combination of Proposition~\ref{prop:projectors}\ref{prop:Qprojectors2}--\ref{prop:projectors}\ref{prop:Qprojectors3}. Statement~\ref{thm:Qprojunitary2} also appeals to Proposition~\ref{prop:projectors}\ref{prop:Qprojectors3} in computing $P_{(\mathcal{U}^{4\times 4})^{Q_{\eta}^2}}(\mathcal{F}_{\eta}\chi_{\ell} \mathcal{F}_{\eta}^{-1}\vec{U})$. This is followed by an application of $\mathcal{F}_{\eta}\chi_{\ell}^{-1} \mathcal{F}_{\eta}^{-1}$ to re-express the resulting ensemble in terms of the standard ensemble.
\end{proof}

\subsubsection*{Projector onto $C_2$}


We consider the set $\mathcal{A} = \{\alpha=(\alpha_1,\alpha_2)\in \mathbb{Z}_+^2\, : \, 0<|\alpha| \leq \mu\}$ which contains a finite number of multi-indices. We also define two bijections 
\begin{equation}\label{eqn:bijectiondefinition}
    \nu:\{1,2,\ldots, \eta^2\} \to Q_{\eta}^2 \text{~~and~~} \rho: \{1,2,\ldots,|\mathcal{A}|\} \to \mathcal{A}
\end{equation}	
so that the sets
\begin{equation*}
    \{\nu(1), \nu(2), \ldots, \nu(\eta^2)\} \text{~~and~~} \{\rho(1), \rho(2), \ldots, \rho(|\mathcal{A}|)\}
\end{equation*}
are suitable enumerations of $Q_{\eta}^2$ and $\mathcal{A}$, respectively. In Theorem~\ref{prop:Qdiscretewaveletdesignmatrix}\ref{prop:Qdiscretewaveletdesignmatrix3}, the regularity condition on $\vec{U}$ has an equivalent formulation in terms of $\vec{A} = \mathcal{F}^{-1}\vec{U}$, namely,
\begin{equation}\label{eqn:regularityrecall}
    \sum_{k \in Q_{\eta}^2} [k_2^{\alpha_1}k_1^{\alpha_2}]\ostar A_k \in \mathsf{S}\otimes\mathsf{S}^{3\times 3} \, \mbox{ for } 0\leq|\alpha|\leq \mu 
\end{equation} where $k=(k_1,k_2)$ and $\alpha=(\alpha_1,\alpha_2)$. At this point, we view $\sum_{k \in Q_{\eta}^2} [k_2^{\alpha_1}k_1^{\alpha_2}]\ostar A_k$ as an $8\times 8$ matrix instead of the usual $4\times 4$ spinor-vector block matrix. We look at the appropriate entries of $\sum_{k \in Q_{\eta}^2} [k_2^{\alpha_1}k_1^{\alpha_2}]\ostar A_k$ that must be zero so that the regularity condition is satisfied. For a fixed $\alpha$,  condition \eqref{eqn:regularityrecall} is equivalent to
\begin{empheq}[left=\empheqlbrace]{align} 
    {\textstyle\sum_{n=1}^{\eta^2}} \big(\nu(n)_2^{\alpha_1}\nu(n)_1^{\alpha_2}\big) \big(A_{\nu(n)}[0,2j] + A_{\nu(n)}[0,2j+1] \big)&= 0 \label{eqn:QregularitycolumnsR0}\\
    {\textstyle	\sum_{n=1}^{\eta^2}} \big(\nu(n)_2^{\alpha_1}\nu(n)_1^{\alpha_2}\big) \big(A_{\nu(n)}[1,2j] + A_{\nu(n)}[1,2j+1] \big)&= 0 \label{eqn:QregularitycolumnsR1}\\
    {\textstyle	\sum_{n=1}^{\eta^2}} \big(\nu(n)_2^{\alpha_1}\nu(n)_1^{\alpha_2}\big) \big(A_{\nu(n)}[2j,0] + A_{\nu(n)}[2j,1] \big)&= 0 \label{eqn:QregularityrowsR0}\\
    {\textstyle	\sum_{n=1}^{\eta^2}} \big(\nu(n)_2^{\alpha_1}\nu(n)_1^{\alpha_2}\big) \big(A_{\nu(n)}[2j+1,0] + A_{\nu(n)}[2j+1,1] \big)&= 0 \label{eqn:QregularityrowsR1}
\end{empheq}
for each $1\leq j \leq 3$. However, we know from Theorem~\ref{prop:Qdiscretewaveletdesignmatrix}\ref{prop:Qdiscretewaveletdesignmatrix5} that $A_{\nu(n)} = \tau A_{\nu(n)}\tau$ for each $n \in\{1,2,\ldots, \eta^2\}$. This implies that $A_{\nu(n)}[1,2j] = A_{\nu(n)}[0,2j+1]$ and $A_{\nu(n)}[1,2j+1] = A_{\nu(n)}[0,2j]$ for each $1\leq j \leq 3$, and so \eqref{eqn:QregularitycolumnsR0} is equivalent to \eqref{eqn:QregularitycolumnsR1}. Similarly, $A_{\nu(n)}[2j+1,0]=A_{\nu(n)}[2j,1]$ and $A_{\nu(n)}[2j+1,1]=A_{\nu(n)}[2j,0]$ for each $1\leq j \leq 3$, and so \eqref{eqn:QregularityrowsR0} is equivalent to \eqref{eqn:QregularityrowsR1}. Thus, we only need to impose \eqref{eqn:QregularitycolumnsR0} and \eqref{eqn:QregularityrowsR0} to satisfy the regularity condition.

Equation \eqref{eqn:QregularitycolumnsR0} requires that the vector $(\nu(n)_2^{\alpha_1} \nu(n)_1^{\alpha_2})_{n=1}^{\eta^2}$ be orthogonal to the vector $(A_{\nu(n)}[0,2j] + A_{\nu(n)}[0,2j+1] )_{n=1}^{\eta^2}$. Since we want this condition to hold for each $\alpha$ satisfying $1\leq |\alpha| \leq \mu$, we form the matrix $R \in \mathbb{C}^{|\mathcal{A}| \times \eta^2}$ with entries 
\begin{equation}\label{eqn:regularitymatrixR}
    R_{mn} = \nu(n)_2^{\rho(m)_1}\nu(n)_1^{\rho(m)_2} \mbox{ where } 1 \leq m \leq |\mathcal{A}|,\, 1 \leq n \leq \eta^2,
\end{equation} and for each $1\leq j \leq 3$, we require $(A_{\nu(n)}[0,2j] + A_{\nu(n)}[0,2j+1])_{n=1}^{\eta^2}$  to be in the null space of $R$.

We further convert \eqref{eqn:QregularityrowsR0} to a condition in terms of $(A_{\nu(n)}[0,0] + A_{\nu(n)}[0,1])_{n=1}^{\eta^2}$. Note that $m_0(\xi) = \sum_{n=1}^{\eta^2}e^{2\pi \xi \wedge \nu(n)} (A_{\nu(n)}[0,0] + A_{\nu(n)}[0,1])$.  For a fixed $\alpha=(\alpha_1,\alpha_2)$, the regularity condition described by \eqref{eqn:QregularityrowsR0} requires \begin{align*}
    \partial^{\alpha}m_0(\xi)\bigr|_{\xi = \frac{v_{\ell}}{2}} &\! =\!(2\pi e_{12})^{|\alpha|}(-1)^{\alpha_2}\!\sum_{n=1}^{\eta^2}\nu(n)_1^{\alpha_2}\nu(n)_2^{\alpha_1}e^{\pi v_{\ell} \wedge \nu(n)}(A_{\nu(n)}[0,0] + A_{\nu(n)}[0,1]) = 0\\
    &\iff \sum_{n=1}^{\eta^2}\nu(n)_1^{\alpha_2}\nu(n)_2^{\alpha_1}e^{\pi v_{\ell} \wedge \nu(n)}(A_{\nu(n)}[0,0] + A_{\nu(n)}[0,1])\! =\! 0
\end{align*}
for each $1\leq \ell \leq 3$. Thus, \eqref{eqn:QregularityrowsR0} and \eqref{eqn:QregularityrowsR1} are equivalent to
\begin{equation*}
    \left\{\begin{array}{rl}
        {\textstyle	\sum_{n=1}^{\eta^2}}(-1)^{\nu(n)_2}\nu(n)_1^{\alpha_2}\nu(n)_2^{\alpha_1}(A_{\nu(n)}[0,0] + A_{\nu(n)}[0,1]) &= \,0\\
        {\textstyle	\sum_{n=1}^{\eta^2}}(-1)^{\nu(n)_1}\nu(n)_1^{\alpha_2}\nu(n)_2^{\alpha_1}(A_{\nu(n)}[0,0] + A_{\nu(n)}[0,1]) &= \,0\\
        {\textstyle	\sum_{n=1}^{\eta^2}}(-1)^{\nu(n)_1+\nu(n)_2}\nu(n)_1^{\alpha_2}\nu(n)_2^{\alpha_1}(A_{\nu(n)}[0,0] + A_{\nu(n)}[0,1]) &=\,0.
    \end{array}\right.
\end{equation*}
Define $S^{(1)},S^{(2)},S^{(3)} \in  \mathbb{C}^{|\mathcal{A}| \times \eta^2}$ with entries given by
\begin{align*}
    S^{(1)}_{mn} &= (-1)^{\nu(n)_2}\nu(n)_1^{\rho(m)_2}\nu(n)_2^{\rho(m)_1},\\
    S^{(2)}_{mn} &= (-1)^{\nu(n)_1}\nu(n)_1^{\rho(m)_2}\nu(n)_2^{\rho(m)_1},\\ 
    \text{and~} S^{(3)}_{mn} &= (-1)^{\nu(n)_1+\nu(n)_2}\nu(n)_1^{\rho(m)_2}\nu(n)_2^{\rho(m)_1},	
\end{align*}
where $1 \leq m \leq |\mathcal{A}|$ and $1 \leq n \leq \eta^2$ and form the matrix $S\in\mathbb{C}^{3|\mathcal{A}| \times \eta^2}$ by \begin{equation}\label{eqn:regularitymatrixS}
    S := \begin{bmatrix}S^{(1)}\\ S^{(2)}\\S^{(3)}\end{bmatrix}.
\end{equation} The requirement described in $\eqref{eqn:QregularityrowsR0}$ is equivalent to having $(A_{\nu(n)}[0,0])_{n=1}^{\eta^2}$ in the null space of $S$.

Therefore, satisfying the regularity condition entails projecting appropriate vectors onto the null space of the real matrices $R$ and $S$ as defined in \eqref{eqn:regularitymatrixR} and \eqref{eqn:regularitymatrixS}, respectively. Viewing $\mathbb{R}_2^{{\eta}^2}$ as a Hilbert space with inner product as given in \eqref{eqn:QrealinnerproductV}, we appeal to  Proposition~\ref{prop:projectors}\ref{prop:projectors1} in carrying out the projections onto  $\Null R$ and $\Null S$. That is, the projector $\tilde{R}$ onto $\Null R$ and the projector $\tilde{S}$ onto $\Null S$ are given by
\begin{align}
    \tilde{R} &= I_{\eta^2} - R^{\ast}(RR^{\ast})^{-1}R\label{eqn:QRtilde}\\ 
    \mbox{and~}\tilde{S} &= I_{\eta^2} - S^{\ast}(SS^{\ast})^{-1}S,\label{eqn:QStilde}
\end{align}
respectively. With $\tilde{R}$ and $\tilde{S}$, we can now describe the projector onto $C_2$ in the following proposition. The proof is omitted for brevity.

\begin{proposition}\label{thm:Qprojregularity}
    Let $\eta \geq 4$ be even, $\mu \in \mathbb{Z}_+$, $C_2$ be as given in \eqref{prob:Qwaveletfeasibilityproblem3}, and define $\mathcal{A} := \{\alpha=(\alpha_1,\alpha_2)\in \mathbb{Z}_+^2\, : \, 0<|\alpha| \leq \mu\}$. Suppose further that $\nu$ and $\rho$ are bijections defined as in \eqref{eqn:bijectiondefinition}, and $\tilde{R}$ and $\tilde{S}$ are the projectors defined as in \eqref{eqn:QRtilde} and in \eqref{eqn:QStilde}, respectively. Let $\vec{U} \in (\mathsf{S}^{4\times 4})_{\sigma\tau}^{Q_{\eta}^2}$ and $\vec{A}=(A_k)_{k\in Q_{\eta}^2} = \mathcal{F}_{\eta}^{-1} \vec{U}$. Define $\vec{A}^+ =(A_{\nu(n)}^+)_{n=1}^{\eta^2}$ by
    \begin{align*}
        (A_{\nu(n)}^+[0,0])_{n=1}^{\eta^2} &= \big(\tilde{S}(A_{\nu(n)}[0,0] + A_{\nu(n)}[0,1])_{n=1}^{\eta^2}\big)_s\\
        (A_{\nu(n)}^+[0,1])_{n=1}^{\eta^2} &= \big(\tilde{S}(A_{\nu(n)}[0,0] + A_{\nu(n)}[0,1])_{n=1}^{\eta^2}\big)_v\\
        (A_{\nu(n)}^+[0,2j])_{n=1}^{\eta^2} &= \big(\tilde{R}(A_{\nu(n)}[0,2j] + A_{\nu(n)}[0,2j+1])_{n=1}^{\eta^2}\big)_s\\
        (A_{\nu(n)}^+[0,2j+1])_{n=1}^{\eta^2} &= \big(\tilde{R}(A_{\nu(n)}[0,2j] + A_{\nu(n)}[0,2j+1])_{n=1}^{\eta^2}\big)_v
    \end{align*} for $1\leq j \leq 3$, and the other rows are given by 
    \begin{align*}
        A_{\nu(n)}^+[(2,3),(0,1,\ldots,7)] &=(-1)^{\nu(n)_2} A_{\nu(n)}^+[(0,1),(0,1,\ldots,7)]\\
        A_{\nu(n)}^+[(4,5,6,7),(0,1,\ldots,7)] &=(-1)^{\nu(n)_1} A_{\nu(n)}^+[(0,1,2,3),(0,1,\ldots,7)]
    \end{align*}
    for all $n\in\{1,2,\ldots \eta^2\}$. Then $P_{C_2}(\vec{U}) = \mathcal{F}_{\eta}\vec{A}^+$. 
\end{proposition}
\begin{proof} Refer to the preceding discussion.  
\end{proof}
Notice that the first two rows of all matrices in $\vec{A}^+$ are obtained by applying the projectors $\tilde{R}$ and $\tilde{S}$ to the vectors formed from (the appropriate columns of) the first and second rows of all matrices in $\vec{A}$. The other rows of all matrices in $\vec{A}^+$ are then filled out according to the consistency conditions in Proposition~\ref{prop:Qdiscretewaveletdesignmatrix}\ref{prop:Qdiscretewaveletdesignmatrix4}\ref{prop:Qdiscretewaveletdesignmatrix4b} for coefficient ensembles.

\subsubsection*{Projector onto $C_3$}

We again draw inspiration from the complex-valued case in describing the projector onto the point symmetry constraint set $C_3$. We start by noting that Proposition~\ref{prop:projectors}\ref{prop:projectors_symm} extends to the quaternionic setting. Notice also that $C_3$ imposes a ``perfect'' symmetry condition.

\begin{proposition}\label{prop:Qprojectors_symm}
    Let $0 \neq \vec{a} \in \mathbb{R}_2^d$, $b\in \mathbb{R}_2$, $\epsilon\geq 0$ and $C:=\{\vec{z} \in \mathbb{R}_2^d \,:\, |\langle \vec{a},\vec{z} \rangle_{\mathbb{R}_2} + b | \leq \epsilon\}$. If $\beta(\vec{z}) = \frac{ \epsilon (\langle \vec{a},\vec{z} \rangle_{\mathbb{R}_2} + b)}{|\langle \vec{a},\vec{z} \rangle_{\mathbb{R}_2} + b|}$, then $P_C(z) =\vec{p}$ where
    $$\vec{p}= \begin{cases}\vec{z} & \text{if~} |\langle \vec{a},\vec{z} \rangle_{\mathbb{R}_2} + b| \leq \epsilon\\ \vec{z} + \frac{\beta(\vec{z})-\langle \vec{a},\vec{z} \rangle_{\mathbb{R}_2}-b}{\|\vec{a}\|^2}\vec{a} & \text{if~} |\langle \vec{a},\vec{z} \rangle_{\mathbb{R}_2}  + b|> \epsilon.\end{cases}$$
\end{proposition}
\begin{proof} This follows the same argument as in the proof of Proposition~\ref{prop:projectors}\ref{prop:projectors_symm} using the appropriate real inner product in \eqref{eqn:QrealinnerproductV} and polarisation identity in \eqref{eqn:Qpolarization}.
\end{proof}

Now, for each $j \in Q_{\eta}^2$, define the quaternion vector $\vec{g}^{(j)}=(1,-e^{-2\pi j\wedge P/\eta}) \in\mathbb{R}_2^2$ and the set $G^{(j)} \subset \mathbb{R}_2^2$ by
\begin{equation}\label{eqn:G_C3}
    G^{(j)} := \{\vec{z}=(z_1,z_2)\in \mathbb{R}_2^2 \, : \, |\langle \vec{g}^{(j)},\vec{z}\rangle_{\mathbb{R}_2}| = 0\}.
\end{equation}
The projection onto $G^{(j)}$ is carried out by following Proposition~\ref{prop:Qprojectors_symm} (with $b=\epsilon=0$). The projector onto $C_3$ is given in the next proposition.

\begin{proposition}
    Let $\mathcal{H} = (\mathsf{S}^{4\times 4})_{\sigma \tau}^{Q_{\eta}^2}$, $C_3$ be defined as in \eqref{prob:Qwaveletfeasibilityproblem4}, and $G^{(j)}$ be defined as in \eqref{eqn:G_C3}. If $\vec{U}=(U_j)_{j\in Q_{\eta}^2}$, then $P_{C_3}(\vec{U}) = \vec{V} = (V_j)_{j \in Q_{\eta}^2}$ where
    \begin{align*}
        (V_j[0,0],V_{\eta v-j}[0,0]) &\!=\! (P_{G^{(j)}}((U_j[0,0]+U_j[0,1],U_{\eta v_3-j}[0,0]+U_{\eta v_3 -j}[0,1])))_s \, \forall j\in Q_{\eta}^2,\\
        (V_j[0,1],V_{\eta v-j}[0,1]) &\!=\! (P_{G^{(j)}}((U_j[0,0]+U_j[0,1],U_{\eta v_3-j}[0,0]+U_{\eta v_3 -j}[0,1])))_v \, \forall j\in Q_{\eta}^2,\\
        V_j[(0),(3,4,5,6,7,8)] &\!=\! U_j[(0),(3,4,5,6,7,8)] \, \forall j\in Q_{\eta}^2,
    \end{align*}
    and the other rows of $V_j$ are filled out (according to the consistency conditions) as
    $$V_j[(2k,2k+1), (0,1,\ldots,7)] = V_{k+v_j/2}[(0,1),(0,1,\ldots,7)]$$
    for all $j\in Q_{\eta}^2$, $k\in\{1,2,3\}$.
\end{proposition}

\section{Douglas--Rachford Algorithm}\label{app2}

\begin{algorithm}[H]\label{alg:DR}
    \SetAlgoLined
    \SetKwInOut{Input}{input}
    \SetKwInOut{Output}{output}
    \Input{$x_0 \in \mathcal{H}$}
    Compute $p_0 \in P_D(x_0)$\;
    \While{stopping criteria not satisfied}{
        Compute $x_{k+1}$ such that $x_{k+1} \in x_k + P_C(2p_k -x_k) - p_k$\;
        Compute $p_{k+1}$ such that $p_{k+1} \in P_D(x_{k+1})$\;
        Update $x_{k} = x_{k+1}$\;
    }
    \Output{$p_k$}
    \caption{DR}
\end{algorithm}



%
%
%

\end{document}

%% file: mymacros.tex
\usepackage[margin = 1.9cm]{geometry}
\usepackage{amssymb,amsthm,empheq,verbatim}
\usepackage{fancyhdr}
\usepackage{multirow}
\usepackage{epsfig}
\usepackage{graphicx,amsmath,float}
\usepackage[caption=false]{subfig}
\usepackage[colorlinks = true,
linkcolor = blue,
urlcolor  = blue,
citecolor = blue,
anchorcolor = blue]{hyperref}
\usepackage{rotating, color,multicol}
\usepackage{fancyhdr,array,enumerate}
\usepackage{enumitem}
\usepackage{todonotes}
\usepackage{eucal}
\usepackage{booktabs,stmaryrd}
\graphicspath{{figures/}}

\usepackage[ruled,vlined]{algorithm2e}

\newtheorem{theorem}{Theorem}[section]
\newtheorem{definition}[theorem]{Definition}
\newtheorem{proposition}[theorem]{Proposition}

\newtheorem{lemma}[theorem]{Lemma}
\newtheorem{remark}[theorem]{Remark}

\newtheorem{problem}[theorem]{Problem}

\newcommand{\Null}{\operatorname{null}}

\newcommand{\eps}{\varepsilon}
\newcommand{\Id}{\operatorname{Id}}

\newcommand{\Tr}{\operatorname{Tr}}
\newcommand{\Fix}{\operatorname{Fix }}

\newcommand{\Fq}{\ensuremath{\operatorname{\mathcal{F}_2^+}}}

\renewcommand{\vec}[1]{\mathbf{#1}}
\newcommand{\Mod}[1]{\,(\mathrm{mod}\ #1)}

\allowdisplaybreaks

\makeatletter
\newcommand{\ostar}{\mathbin{\mathpalette\make@circled\star}}
\newcommand{\make@circled}[2]{%
	\ooalign{$\m@th#1\smallbigcirc{#1}$\cr\hidewidth$\m@th#1#2$\hidewidth\cr}%
}
\newcommand{\smallbigcirc}[1]{%
	\vcenter{\hbox{\scalebox{0.77778}{$\m@th#1\bigcirc$}}}%
}

%% file: main.bbl
{\small

\begin{thebibliography}{35}
    \expandafter\ifx\csname natexlab\endcsname\relax\def\natexlab#1{#1}\fi
    \providecommand{\url}[1]{\texttt{#1}}
    \providecommand{\href}[2]{#2}
    \providecommand{\path}[1]{#1}
    \providecommand{\DOIprefix}{doi:}
    \providecommand{\ArXivprefix}{arXiv:}
    \providecommand{\URLprefix}{URL: }
    \providecommand{\Pubmedprefix}{pmid:}
    \providecommand{\doi}[1]{\href{http://dx.doi.org/#1}{\path{#1}}}
    \providecommand{\Pubmed}[1]{\href{pmid:#1}{\path{#1}}}
    \providecommand{\bibinfo}[2]{#2}
    \ifx\xfnm\relax \def\xfnm[#1]{\unskip,\space#1}\fi
    \bibitem[{Daubechies(1988)}]{daubechies1}
    \bibinfo{author}{I.~Daubechies},
    \newblock \bibinfo{title}{Orthonormal bases of compactly supported wavelets},
    \newblock \bibinfo{journal}{Commun. Pure Appl. Math.} \bibinfo{volume}{41}
    (\bibinfo{year}{1988}) \bibinfo{pages}{909--996}.
    \bibitem[{Daubechies(1992)}]{daubechiesbook}
    \bibinfo{author}{I.~Daubechies}, \bibinfo{title}{Ten Lectures on Wavelets},
    \bibinfo{publisher}{SIAM}, \bibinfo{address}{Philadelphia, Pennsylvania},
    \bibinfo{year}{1992}.
    \bibitem[{Mallat(1989)}]{mallat}
    \bibinfo{author}{S.~G. Mallat},
    \newblock \bibinfo{title}{Multiresolution approximations and wavelet
        orthonormal bases of ${L}_2(\mathbb{R})$},
    \newblock \bibinfo{journal}{Trans. Am. Math. Soc.} \bibinfo{volume}{315}
    (\bibinfo{year}{1989}) \bibinfo{pages}{69--87}.
    \bibitem[{Meyer(1989)}]{meyerchapter}
    \bibinfo{author}{Y.~Meyer},
    \newblock \bibinfo{title}{Orthonormal wavelets},
    \newblock in: \bibinfo{booktitle}{Wavelets}, \bibinfo{publisher}{Springer},
    \bibinfo{year}{1989}, pp. \bibinfo{pages}{21--37}.
    \bibitem[{Franklin et~al.(2019)Franklin, Hogan, and Tam}]{FHTconference}
    \bibinfo{author}{D.~Franklin}, \bibinfo{author}{J.~A. Hogan},
    \bibinfo{author}{M.~Tam},
    \newblock \bibinfo{title}{Higher-dimensional wavelets and the
        {D}ouglas--{R}achford algorithm},
    \newblock in: \bibinfo{booktitle}{13th International Conference on Sampling
        Theory and Applications (SampTA)}, \bibinfo{organization}{IEEE},
    \bibinfo{year}{2019}, pp. \bibinfo{pages}{1--4}.
    \bibitem[{Franklin et~al.(2024)Franklin, Hogan, and Tam}]{FHTpaper}
    \bibinfo{author}{D.~Franklin}, \bibinfo{author}{J.~A. Hogan},
    \bibinfo{author}{M.~K. Tam},
    \newblock \bibinfo{title}{Non-separable multidimensional multiresolution
        wavelets: A {D}ouglas--{R}achford approach},
    \newblock \bibinfo{journal}{Appl. Comput. Harmon. Anal.} \bibinfo{volume}{73}
    (\bibinfo{year}{2024}) \bibinfo{pages}{101684}.
    \bibitem[{Franklin(2018)}]{franklinthesis}
    \bibinfo{author}{D.~J. Franklin}, \bibinfo{title}{Projection algorithms for
        non-separable wavelets and {C}lifford {F}ourier analysis}, Ph.D. thesis,
    University of Newcastle, \bibinfo{year}{2018}.
    \bibitem[{von Neumann(1950)}]{neumann}
    \bibinfo{author}{J.~von Neumann}, \bibinfo{title}{Functional Operators Volume
        II: The Geometry of Orthogonal Spaces}, \bibinfo{publisher}{Princeton
        University Press}, \bibinfo{address}{New Jersey, USA}, \bibinfo{year}{1950}.
    \bibitem[{Douglas and Rachford(1956)}]{drachford}
    \bibinfo{author}{J.~Douglas}, \bibinfo{author}{H.~Rachford},
    \newblock \bibinfo{title}{On the numerical solution of heat conduction problems
        in two and three space variables},
    \newblock \bibinfo{journal}{Trans. Am. Math. Soc.} \bibinfo{volume}{82}
    (\bibinfo{year}{1956}) \bibinfo{pages}{421--439}.
    \bibitem[{Dizon et~al.(2021)Dizon, Hogan, and Lindstrom}]{DHLi_centering}
    \bibinfo{author}{N.~D. Dizon}, \bibinfo{author}{J.~A. Hogan},
    \bibinfo{author}{S.~B. Lindstrom},
    \newblock \bibinfo{title}{Centering projection methods for wavelet feasibility
        problems},
    \newblock in: \bibinfo{booktitle}{Current Trends in Analysis, its Applications
        and Computation: Proceedings of the 12th ISAAC Congress, Aveiro, Portugal,
        2019}, \bibinfo{organization}{Springer}, \bibinfo{year}{2021}, pp.
    \bibinfo{pages}{661--669}.
    \bibitem[{Dao et~al.(2021)Dao, Dizon, Hogan, and Tam}]{DDHTart}
    \bibinfo{author}{M.~N. Dao}, \bibinfo{author}{N.~D. Dizon},
    \bibinfo{author}{J.~A. Hogan}, \bibinfo{author}{M.~K. Tam},
    \newblock \bibinfo{title}{Constraint reduction reformulations for projection
        algorithms with applications to wavelet construction},
    \newblock \bibinfo{journal}{J. Optim. Theory Appl.} \bibinfo{volume}{190}
    (\bibinfo{year}{2021}) \bibinfo{pages}{201--233}.
    \bibitem[{Dizon(2021)}]{dizonthesis}
    \bibinfo{author}{N.~D. Dizon}, \bibinfo{title}{Optimization in the Construction
        of Multidimensional Wavelets}, Ph.D. thesis, University of Newcastle,
    \bibinfo{year}{2021}.
    \bibitem[{Pei and Cheng(1997)}]{peicolour}
    \bibinfo{author}{S.-C. Pei}, \bibinfo{author}{C.-M. Cheng},
    \newblock \bibinfo{title}{A novel block truncation coding of color images using
        a quaternion-moment-preserving principle},
    \newblock \bibinfo{journal}{IEEE Trans. Commun.} \bibinfo{volume}{45}
    (\bibinfo{year}{1997}) \bibinfo{pages}{583--595}.
    \bibitem[{Sangwine(1996)}]{sangwinecolour}
    \bibinfo{author}{S.~J. Sangwine},
    \newblock \bibinfo{title}{Fourier transforms of colour images using quaternion
        or hypercomplex numbers},
    \newblock \bibinfo{journal}{Electron. Lett.} \bibinfo{volume}{32}
    (\bibinfo{year}{1996}) \bibinfo{pages}{1979--1980}.
    \bibitem[{Fletcher and Sangwine(2017)}]{fletchersangwine}
    \bibinfo{author}{P.~Fletcher}, \bibinfo{author}{S.~J. Sangwine},
    \newblock \bibinfo{title}{The development of the quaternion wavelet transform},
    \newblock \bibinfo{journal}{Signal Process.} \bibinfo{volume}{136}
    (\bibinfo{year}{2017}) \bibinfo{pages}{2--15}.
    \bibitem[{Ginzberg and Walden(2012)}]{ginzbergwalden}
    \bibinfo{author}{P.~Ginzberg}, \bibinfo{author}{A.~T. Walden},
    \newblock \bibinfo{title}{Matrix-valued and quaternion wavelets},
    \newblock \bibinfo{journal}{IEEE Trans. Signal Process.} \bibinfo{volume}{61}
    (\bibinfo{year}{2012}) \bibinfo{pages}{1357--1367}.
    \bibitem[{Hogan and Morris(2012)}]{hoganmorris1}
    \bibinfo{author}{J.~Hogan}, \bibinfo{author}{A.~J. Morris},
    \newblock \bibinfo{title}{Quaternionic wavelets},
    \newblock \bibinfo{journal}{Numer. Funct. Anal. Optim.} \bibinfo{volume}{33}
    (\bibinfo{year}{2012}) \bibinfo{pages}{1031--1062}.
    \bibitem[{Morris(2014)}]{morristhesis}
    \bibinfo{author}{A.~J. Morris}, \bibinfo{title}{{F}ourier and Wavelet Analysis
        of {C}lifford-Valued Functions}, Ph.D. thesis, University of Newcastle,
    Australia, Callaghan, NSW, 2308, Australia, \bibinfo{year}{2014}.
    \bibitem[{Fletcher(2018)}]{fletcher}
    \bibinfo{author}{P.~Fletcher},
    \newblock \bibinfo{title}{Discrete wavelets with quaternion and {C}lifford
        coefficients},
    \newblock \bibinfo{journal}{Adv. Appl. Clifford Algebras} \bibinfo{volume}{28}
    (\bibinfo{year}{2018}) \bibinfo{pages}{1--30}.
    \bibitem[{Brackx et~al.(2005)Brackx, De~Schepper, and Sommen}]{brackx1}
    \bibinfo{author}{F.~Brackx}, \bibinfo{author}{N.~De~Schepper},
    \bibinfo{author}{F.~Sommen},
    \newblock \bibinfo{title}{The {C}lifford--{F}ourier transform},
    \newblock \bibinfo{journal}{J. Fourier Anal. Appl.} \bibinfo{volume}{11}
    (\bibinfo{year}{2005}) \bibinfo{pages}{669--681}.
    \bibitem[{Brackx et~al.(2006)Brackx, De~Schepper, and Sommen}]{brackx2}
    \bibinfo{author}{F.~Brackx}, \bibinfo{author}{N.~De~Schepper},
    \bibinfo{author}{F.~Sommen},
    \newblock \bibinfo{title}{The two-dimensional {C}lifford--{F}ourier transform},
    \newblock \bibinfo{journal}{J. Math. Imaging Vis.} \bibinfo{volume}{26}
    (\bibinfo{year}{2006}) \bibinfo{pages}{5--18}.
    \bibitem[{Hogan and Morris(2013)}]{hoganmorris2}
    \bibinfo{author}{J.~Hogan}, \bibinfo{author}{A.~J. Morris},
    \newblock \bibinfo{title}{Translation-invariant {C}lifford operators},
    \newblock in: \bibinfo{booktitle}{AMSI International Conference on Harmonic
        Analysis and Applications}, \bibinfo{organization}{Centre for Mathematics and
        its Applications, Mathematical Sciences Institute}, \bibinfo{year}{2013}, pp.
    \bibinfo{pages}{48--62}.
    \bibitem[{Arag{\'o}n~Artacho et~al.(2019)Arag{\'o}n~Artacho, Campoy, and
        Tam}]{aacampoy}
    \bibinfo{author}{F.~J. Arag{\'o}n~Artacho}, \bibinfo{author}{R.~Campoy},
    \bibinfo{author}{M.~K. Tam},
    \newblock \bibinfo{title}{The {D}ouglas--{R}achford algorithm for convex and
        nonconvex feasibility problems},
    \newblock \bibinfo{journal}{Math. Method Oper. Res.}  (\bibinfo{year}{2019})
    \bibinfo{pages}{1--40}.
    \bibitem[{Borwein and Sims(2011)}]{bsims}
    \bibinfo{author}{J.~M. Borwein}, \bibinfo{author}{B.~Sims},
    \newblock \bibinfo{title}{The {D}ouglas--{R}achford algorithm in the absence of
        convexity},
    \newblock in: \bibinfo{booktitle}{Fixed-point Algorithms for Inverse Problems
        in Science and Engineering}, \bibinfo{publisher}{Springer},
    \bibinfo{address}{New York}, \bibinfo{year}{2011}, pp.
    \bibinfo{pages}{93--109}.
    \bibitem[{Dao and Tam(2019)}]{dtam}
    \bibinfo{author}{M.~N. Dao}, \bibinfo{author}{M.~K. Tam},
    \newblock \bibinfo{title}{A {L}yapunov-type approach to convergence of the
        {D}ouglas--{R}achford algorithm for a nonconvex setting},
    \newblock \bibinfo{journal}{J. Global Optim.} \bibinfo{volume}{73}
    (\bibinfo{year}{2019}) \bibinfo{pages}{83--112}.
    \bibitem[{Bauschke and Dao(2017)}]{bdao}
    \bibinfo{author}{H.~H. Bauschke}, \bibinfo{author}{M.~N. Dao},
    \newblock \bibinfo{title}{On the finite convergence of the
        {D}ouglas--{R}achford algorithm for solving (not necessarily convex)
        feasibility problems in {E}uclidean spaces},
    \newblock \bibinfo{journal}{SIAM J. Optim.} \bibinfo{volume}{27}
    (\bibinfo{year}{2017}) \bibinfo{pages}{507--537}.
    \bibitem[{Lions and Mercier(1979)}]{lions}
    \bibinfo{author}{P.~Lions}, \bibinfo{author}{B.~Mercier},
    \newblock \bibinfo{title}{Splitting algorithms for the sum of two nonlinear
        operators},
    \newblock \bibinfo{journal}{SIAM J. Numer. Anal.} \bibinfo{volume}{16}
    (\bibinfo{year}{1979}) \bibinfo{pages}{964--979}.
    \bibitem[{Svaiter(2011)}]{svaiter}
    \bibinfo{author}{B.~F. Svaiter},
    \newblock \bibinfo{title}{On weak convergence of the {D}ouglas--{R}achford
        method},
    \newblock \bibinfo{journal}{SIAM J. Control Optim.} \bibinfo{volume}{49}
    (\bibinfo{year}{2011}) \bibinfo{pages}{280--287}.
    \bibitem[{Pierra(1984)}]{pierra}
    \bibinfo{author}{G.~Pierra},
    \newblock \bibinfo{title}{Decomposition through formalization in a product
        space},
    \newblock \bibinfo{journal}{Math. Program.} \bibinfo{volume}{28}
    (\bibinfo{year}{1984}) \bibinfo{pages}{96--115}.
    \bibitem[{Bauschke and Combettes(2017)}]{bcombettes}
    \bibinfo{author}{H.~H. Bauschke}, \bibinfo{author}{P.~L. Combettes},
    \bibinfo{title}{Convex Analysis and Monotone Operator Theory in Hilbert
        Spaces}, \bibinfo{publisher}{Springer}, \bibinfo{address}{Cham},
    \bibinfo{year}{2017}. \DOIprefix\doi{10.1007/978-3-319-48311-5}.
    \bibitem[{Krivoshein et~al.(2016)Krivoshein, Protasov, and
        Skopina}]{krivoshein}
    \bibinfo{author}{A.~Krivoshein}, \bibinfo{author}{V.~Protasov},
    \bibinfo{author}{M.~A. Skopina}, \bibinfo{title}{Multivariate Wavelet
        Frames}, \bibinfo{publisher}{Springer}, \bibinfo{year}{2016}.
    \bibitem[{Cohen(1990)}]{cohenpaper}
    \bibinfo{author}{A.~Cohen},
    \newblock \bibinfo{title}{Ondelettes, analyses multir{\'e}solutions et filtres
        miroirs en quadrature},
    \newblock in: \bibinfo{booktitle}{Ann. Inst. H. Poincar{\'e} C, Anal. Non
        Lin{\'e}aire}, \bibinfo{year}{1990}, pp. \bibinfo{pages}{439--459}.
    \bibitem[{Sangwine and Bihan(2005)}]{qtfm}
    \bibinfo{author}{S.~J. Sangwine}, \bibinfo{author}{N.~L. Bihan},
    \bibinfo{title}{Quaternion toolbox for
        {MATLAB}\textsuperscript{\textregistered}}, \bibinfo{year}{2013 (First public
        release 2005)}. \URLprefix \url{http://qtfm.sourceforge.net/}.
    \bibitem[{Dizon and Hogan(2024)}]{dizon2024holistic}
    \bibinfo{author}{N.~D. Dizon}, \bibinfo{author}{J.~A. Hogan},
    \newblock \bibinfo{title}{Holistic processing of color images using novel
        quaternion-valued wavelets on the plane: A promising transformative tool},
    \newblock \bibinfo{journal}{IEEE Signal Processing Magazine}
    \bibinfo{volume}{41} (\bibinfo{year}{2024}) \bibinfo{pages}{51--63}.
    \bibitem[{Cegielski(2012)}]{cegielski}
    \bibinfo{author}{A.~Cegielski}, \bibinfo{title}{Iterative Methods for Fixed
        Point Problems in {H}ilbert Spaces}, volume \bibinfo{volume}{2057} of
    \textit{\bibinfo{series}{Lecture Notes in Mathematics}},
    \bibinfo{publisher}{Springer}, \bibinfo{year}{2012}.
    
\end{thebibliography}


}
